\documentclass{memo-l}

\usepackage{mathrsfs}
\usepackage{amssymb}
\usepackage{xy}
\xyoption{arrow}
\xyoption{curve}
\xyoption{matrix}
\xyoption{cmtip}
\xyoption{frame}
\SelectTips{cm}{}
\newdir{ >}{{}*!/-5pt/\dir{>}}

\newtheorem{thm}[equation]{Theorem}
\newtheorem*{thm*}{Theorem}
\newtheorem{cor}[equation]{Corollary}
\newtheorem{lem}[equation]{Lemma}
\newtheorem{prop}[equation]{Proposition}
\theoremstyle{definition}
\newtheorem{defn}[equation]{Definition}
\newtheorem{cons}[equation]{Construction}
\newtheorem{conv}[equation]{Convention}
\newtheorem{notn}[equation]{Notation}
\theoremstyle{remark}
\newtheorem{rem}[equation]{Remark}
\newtheorem{example}[equation]{Example}

\numberwithin{section}{chapter}
\numberwithin{equation}{section}

\newcommand{\THM}{N^{\cy}}
\newcommand{\TB}{B}
\newcommand{\Cell}{\mathrm{Cell}}

\newcommand{\ssdot}{\bullet}%{{\scriptscriptstyle\bullet}}
\newcommand{\supdot}{^\ssdot}
\newcommand{\subdot}{_\ssdot}
\newcommand{\dsubdot}{_{\ssdot\ssdot}}

\newcommand{\simpD}{\mathfrak{S}_{*}^{\aD}}
\newcommand{\simpDob}{\mathfrak{S}_{*}^{\Ob \aD}}
\newcommand{\simps}{\mathfrak{S}_{*}}

\newcommand{\wt}[1]{\widetilde{#1}\strut}
\newcommand{\trc}{\mathrm{trc}}

\newcommand{\cofiber}[2]{C}

\newcommand{\Spdot}[1][\ssdot]{S'_{#1}}
\newcommand{\Sdot}[1][\ssdot]{S_{#1}}

\newcommand{\Spndot}[1][\ssdot,\dotsc,\ssdot]{S^{\prime(n)}_{#1}}
\newcommand{\Sndot}[1][\ssdot,\dotsc,\ssdot]{S^{(n)}_{#1}}
\newcommand{\SpMdot}[1][\ssdot]{S^{\prime\, M}_{#1}}
\newcommand{\SpnMdot}[1][\ssdot,\dotsc,\ssdot]{S^{\prime(n)\,M}_{#1}}
\newcommand{\SMdot}[1][\ssdot]{S^{M}_{#1}}

\newcommand{\Fdot}[1][\ssdot]{F_{#1}}

\newcommand{\Ffdot}[1][\ssdot]{F^{f}_{#1}}

\mathchardef\varDelta="7101

\newcommand{\Wi}{\mathbf{W^{-1}}}
\newcommand{\C}{\mathbf{C}}

\newcommand{\csv}[1][{\aC}]{SF({L#1})}
\newcommand{\csvop}[1][{\aC}]{(SF({L#1}))^{\op}}
\newcommand{\wrcsv}[1][{\aC}]{\widetilde{#1}}
\newcommand{\swrcsv}[1][{\aA}]{\breve{#1}}
\newcommand{\wri}{\tilde\imath}
\newcommand{\swri}{\breve\imath}
\newcommand{\swrj}{\breve\jmath}

\newcommand{\Mod}[1][{\aC}]{\aM_{{#1}^{\op}}}

\newcommand{\fincell}{\aF_{\aC^{\op}}}
\newcommand{\wrfincell}{\wrcsv[\aF]_{\aC^{\op}}}

\def\myop#1{\mathop{\textstyle #1}\limits}
\let\iso\cong
\let\sma\wedge
\newcommand{\htp}{\simeq}
\renewcommand{\to}{\mathchoice{\longrightarrow}{\rightarrow}{\rightarrow}{\rightarrow}}
\newcommand{\from}{\mathchoice{\longleftarrow}{\leftarrow}{\leftarrow}{\leftarrow}}
\newcommand{\sto}{\rightarrow}
\newcommand{\overto}[1]{\xrightarrow{\,#1\,}}
\newcommand{\overfrom}[1]{\xleftarrow{\,#1\,}}

\newcommand{\phat}{^{\scriptscriptstyle\wedge}_{p}}

\newcommand{\abA}{\mathfrak{A}}
\newcommand{\abE}{\mathfrak{E}}

\let\catsymbfont\mathcal
\newcommand{\aA}{{\catsymbfont{A}}}
\newcommand{\aB}{{\catsymbfont{B}}}
\newcommand{\aC}{{\catsymbfont{C}}}
\newcommand{\aD}{{\catsymbfont{D}}}
\newcommand{\aE}{{\catsymbfont{E}}}
\newcommand{\aF}{{\catsymbfont{F}}}
\newcommand{\aG}{{\catsymbfont{G}}}
\newcommand{\aI}{{\catsymbfont{I}}}

\newcommand{\aL}{{\catsymbfont{L}}}
\newcommand{\aM}{{\catsymbfont{M}}}
\newcommand{\aN}{{\catsymbfont{N}}}

\newcommand{\aP}{{\catsymbfont{P}}}
\newcommand{\aQ}{{\catsymbfont{Q}}}
\newcommand{\aS}{{\catsymbfont{S}}}

\newcommand{\aX}{{\catsymbfont{X}}}
\newcommand{\aY}{{\catsymbfont{Y}}}

\newcommand{\fL}{\aL}
\newcommand{\fM}{\aM}
\newcommand{\fN}{\aN}
\newcommand{\fP}{\aP}
\newcommand{\fQ}{\aQ}
\newcommand{\fLL}[1][{}]{\fL^{\aB_{#1}}_{\aA_{#1}}}

\newcommand{\bF}{{\mathbb{F}}}
\newcommand{\bM}{{\mathbb{M}}}

\newcommand{\bR}{{\mathbb{R}}}

\newcommand{\bW}{{\mathbb{W}}}
\newcommand{\bZ}{{\mathbb{Z}}}

\newcommand{\bn}{\mathbf{n}}

\newcommand{\bz}{\mathbf{0}}

\def\quickop#1{\expandafter\DeclareMathOperator\csname
#1\endcsname{#1}}
\quickop{id}\quickop{Id}
\quickop{colim}\quickop{hocolim}\quickop{op}
\quickop{co}\quickop{Ar}\quickop{Hom}\quickop{Ho}
\quickop{ob}\quickop{diag}\quickop{perf}\quickop{Sp}\quickop{Wald}

\quickop{Ext}\quickop{Ob}
\quickop{Tot}\quickop{Lan}\quickop{Map}
\quickop{sk}
\quickop{cy}
\quickop{sd}
\quickop{Sing}

\quickop{Spec}
\quickop{holim}

\DeclareMathOperator*\lcolim{colim}

\newcommand{\on}{\mathbin{\mathrm{on}}}

\newcommand{\procite}[2][{}]{%
\relax\textup{(\cite[#1]{#2})}}

\newcommand{\textin}[1]{#1\index{#1}}
\newcommand{\noindexterm}[1]{\textit{#1}}
\newcommand{\indexterm}[2]{\noindexterm{#1}\index{#2}}
\newcommand{\term}[1]{\indexterm{#1}{#1}}
\newcommand{\specialterm}[2]{\noindexterm{#1}\index{#2@#1}}
\let\oldinput\input
\makeatletter
\def\input{\@ifnextchar\bgroup{\newinput}{\oldinput }}
\def\newinput#1{\message{>>TP:OPEN>>}\oldinput #1
\message{>>TP:CLOSE>>}}
\let\oldinclude\include
\def\include#1{\message{>>TP:OPEN>>}\oldinclude{#1}
\message{>>TP:CLOSE>>}}
\makeatother
\makeindex

\begin{document}

\frontmatter

\title[$THH$ and $TC$ of {W}aldhausen categories]%
{Localization for $THH(ku)$ and the topological {H}ochschild and cyclic
homology of {W}aldhausen categories}

%    author one information
\author{Andrew J. Blumberg}
\address{Department of Mathematics, The University of Texas,
Austin, TX \ 78712}
\email{blumberg@math.utexas.edu}
\thanks{The first author was supported in part by an NSF postdoctoral
fellowship, a Clay Liftoff fellowship, and NSF grants DMS-0111298,
DMS-0906105}

%    author two information
\author{Michael A. Mandell}
\address{Department of Mathematics, Indiana University,
Bloomington, IN \ 47405}
\email{mmandell@indiana.edu}
\thanks{The second author was supported in part by NSF grants
DMS-0504069, DMS-0804272, DMS-1105255}

%    \date is required; it is the date received by the editor.
\date{\today}

%    The 2010 edition of the Mathematics Subject Classification is
%    now available.  If you are citing a classification from the
%    new scheme, use the following input coding instead.
%\subjclass[2010]{Primary }


%\dedication{Dedication text (use \\[2pt] for line break if necessary)}

\begin{abstract}
We prove a conjecture of Hesselholt and Ausoni-Rognes, establishing
localization cofiber sequences of spectra
\[
THH(H\bZ) \to THH(ku) \to THH(ku|KU) \to \Sigma THH(H\bZ)
\]
and
\[
TC(H\bZ) \to TC(ku) \to TC(ku|KU) \to \Sigma TC(H\bZ)
\]
for the topological Hochschild and cyclic homology ($THH$ and $TC$) of
topological $K$-theory.  These sequences support Hesselholt's view of
the map $\ell \to ku$ as a ``tamely ramified'' extension of ring
spectra, and validate the hypotheses necessary for Ausoni's
simplified computation of $V(1)_* K(KU)$.

In order to make sense of the relative term $THH(ku|KU)$ and prove
these results, we develop a theory of $THH$ and $TC$ of Waldhausen
categories and prove the analogues of Waldhausen's theorems for
$K$-theory.  We resolve the longstanding confusion about 
localization sequences in $THH$ and $TC$, and establish a specialized
d\'{e}vissage theorem.
\end{abstract}

\maketitle

\tableofcontents

% (make sure TOC includes sections)
\addtocontents{toc}{\protect\setcounter{tocdepth}{1}}

%    Include unnumbered chapters (preface, acknowledgments, etc.) here.
% tp-set-source: tw-whole.tex
% tex-command: latex
% ultex-add-defs: list

%
% Copyright (C) 2008-12  Andrew J. Blumberg and Michael A. Mandell
%

\chapter*{Introduction}

Algebraic $K$-theory provides a high-level invariant of the homotopy
theory of categories with a notion of extension and equivalence.  The
component group, $K_{0}$, is the universal target for Euler
characteristics, and higher algebraic $K$-theory captures subtle
information intricately tied to number theory and geometry.  For the
algebraic $K$-theory of rings, trace methods using topological
Hochschild homology ($THH$) and topological cyclic homology ($TC$)
have proved remarkably successful at making $K$-theory computations
tractable via the methods of equivariant stable homotopy theory.

At first glance $K$-theory and $THH$ take very different inputs and
have very different formal properties.  For algebraic $K$-theory, the
input is typically a Waldhausen category: a category with
subcategories of cofibrations and weak equivalences.  For $THH$, the
basic input is a spectral category: a category enriched in spectra.
While $THH$ shares $K$-theory's additivity properties, $THH$ seems to
lack $K$-theory's approximation and localization
properties \cite{Dundas}.  A specific example of this failure was
studied at great length in the paper \cite{HM3}.  From the perspective
of the algebraic $K$-theory of rings and connective ring spectra,
where $THH$ is the stabilization of $K$-theory, this discrepancy is in
some ways surprising, as one might expect $THH$ to inherit the
fundamental properties of $K$-theory.

In this paper, we construct $THH$ for a general class of Waldhausen
categories, and show that much of the apparent mismatch of formal
properties is a consequence of the former mismatch of input data.  We
obtain an analogue of Waldhausen's \textin{Approximation Theorem}
\cite[1.6.7]{WaldhausenKT} for $THH$.  On the other hand, we observe that
$THH$ has two different analogues of the localization sequence in
Waldhausen $K$-theory (the ``\textin{Fibration Theorem}''
\cite[1.6.4]{WaldhausenKT}).
One of the localization sequences for $THH$ was developed in our
companion paper on localization in $THH$ of spectral categories
\cite[7.1]{BlumbergMandellTHHLoc} (see Theorem~\ref{thmgenone}
below); when applied to the $K$-theory of
schemes, this sequence produces an analogue of the localization
sequence of Thomason-Trobaugh \cite{ThomasonTrobaugh}.  The other
localization sequence generalizes the localization sequence of
Hesselholt-Madsen \cite{HM3}.  One of the principal contributions
of this paper is to provide a conceptual explanation of the
two localization sequences of $THH$ in relation to the localization
sequence of $K$-theory.

As we explain in Sections~\ref{secspec} and~\ref{futuresec}, a
Waldhausen category that admits factorizations has two spectral
categories associated to it, a connective and a non-connective
variant.  The non-connective theory is ``correct'' from the
perspective of abstract homotopy theory and satisfies localization for
cofiber sequences of spectral categories \cite[7.1]{BlumbergMandellTHHLoc},
but the connective theory is more closely related to $K$-theory.  We
show that the two theories agree under connectivity hypotheses that we
make explicit in Section~\ref{secspheretheorem}; in particular, for rings
and connective ring spectra both spectral categories produce the
expected $THH$.  For exact categories, the connective version agrees
with the $THH$ of exact categories defined by Dundas-McCarthy
\cite{DundasMcCarthy}.  For categories of complexes, the
non-connective version agrees
with the $THH$ of the spectral derived category studied in
\cite{BlumbergMandellTHHLoc}.  Working with the non-connective theory
gives the Thomason-Trobaugh style localization sequences, and working
with the connective theory gives the Hesselholt-Madsen style
localization sequences.

As a main application of this theory, we prove the localization 
sequence associated to the transfer map from $H\bZ$ to $ku$ that was
conjectured by Hesselholt and Ausoni~\cite{AusoniTHH,
AusoniK}. Specifically, we construct 
naturally out of the category of $ku$-modules a simplicial spectral
category $W^{\Gamma}(ku|KU)$ and cofiber sequences in the stable
category
\begin{gather*}
THH(\mathbb{Z})\to THH(ku)\to THH(ku|KU) \to \Sigma THH(\mathbb{Z})\\
\intertext{and}
TC(\mathbb{Z})\to TC(ku)\to TC(ku|KU) \to \Sigma TC(\mathbb{Z}),
\end{gather*}
compatible via a trace map with the localization cofiber sequence in 
$K$-theory established in \cite{BlumbergMandell}.  Corresponding
results hold for the Adams summand in the $p$-local and $p$-complete
cases; see Theorem~\ref{thmkuloc} below for details.  These 
localization sequences were conjectured by Hesselholt and
Ausoni-Rognes to explain the
relationship of the computations of $K(\ell)$ and $K(ku)$; they
support the perspective that $\ell \to ku$ should be an example of a 
``tamely ramified'' extension of ring spectra.  Furthermore, using
these localization sequences, one can dramatically
simplify Ausoni's computation of $K(ku)$ \cite[8.4]{AusoniK} by
mimicking the de Rham-Witt 
arguments in Hesselholt-Madsen \cite{HM3}.  These localization
sequences 
provide the chromatic level $1$ analogues of the chromatic level $0$
sequence of Hesselholt and Madsen \cite{HM3}.  
Another application of these localization sequences is to compute
$K(KU)$.  One would like to use Ausoni's computations of $K(ku)$ along
with the localization cofiber sequence 
\[
K(\mathbb{Z}) \to K(ku) \to K(KU) \to \Sigma K(\mathbb{Z})
\]
to evaluate $K(KU)$.  The transfer map in this sequence is controlled
by the behavior of the transfer map in the associated sequences in
$THH$ and $TC$, where it is easier to understand.  Following
Hesselholt, Ausoni  \cite[8.3]{AusoniK} observes that in light of his calculations, the
existence of the localization cofiber sequence in $THH$ along with an
algebraic fact would permit the complete identification of
$V(1)_*K(KU)$.

%*****
For higher chromatic levels, Rognes has conjectured $K$-theory
localization sequences of the form
\[
K(BP\langle n-1 \rangle\phat) \to K(BP\langle n \rangle\phat) 
\to K(E(n)\phat) \to \Sigma K(BP \langle n-1 \rangle\phat)
\]
as part of an ambitious program to provide a conceptual understanding
of Waldhausen's $A$-theory of a point.  Such sequences are attractive
because they would relate the algebraic $K$-theory of the
nonconnective ring spectrum $E(n)$, to which trace methods do not
apply directly, to the algebraic $K$-theory of connective ring spectra
$BP\langle n \rangle$, to which trace methods do apply.  The
corresponding conjectural localization sequences for $THH$ and $TC$
would then optimistically provide tools for organizing the trace method
computations.

So far, these sequences in both algebraic $K$-theory and $TC$ remain
conjectural for $n>1$, and
there is some reason to be suspicious about the existence of these
sequences\iffalse (see Section~\ref{sec:ifIdidit})\fi.  However, our methods both
in~\cite{BlumbergMandell} and in this paper do establish the existence
of the variant localization sequences
\[
K(\bW\bF_{p^{n}}[\![u_{1},\ldots, u_{n-1}]\!])\to K(BP_{n})\to K(E_{n})\to \Sigma K(\bW\bF_{p^{n}}[\![u_{1},\ldots, u_{n-1}]\!])
\]
and
\[
TC(\bW\bF_{p^{n}}[\![u_{1},\ldots, u_{n-1}]\!])\to TC(BP_{n})\to
TC(BP_{n} | E_{n})\to \Sigma TC(\bW\bF_{p^{n}}[\![u_{1},\ldots, u_{n-1}]\!])
\]
for all $n$, where $\bW$ denotes the $p$-typical Witt ring and
$BP_{n}$ is the connective cover of the Lubin-Tate spectrum $E_{n}$.
This gives a new approach to the continuation of the Rognes program,
using current technology.  This approach has three main advantages
over the program as laid out in \cite{AusoniRognes}:
\begin{enumerate}
\item The localization sequences for $K$-theory, $TC$, and $THH$
relating the spectra
$H\bW\bF[\![u_{1},\ldots,u_{n-1}]\!]$, $BP_{n}$, and $E_{n}$ are known to
exist (as mentioned above) in contrast to the sequences relating $BP\langle
n-1\rangle$, $BP\langle n\rangle$, and $E(n)$,  which are not (for $n>1$).
\item The relevant spectrum in the next step of Rognes' program for
understanding $A_{*}$, the $K$-theory of the sphere, is $K(E_{n})$
rather than $K(E(n))$.
\item The spectra $BP_{n}$ are known to be $E_{\infty}$ ring spectra,
whereas $BP\langle n\rangle$ is currently only known to be
$A_{\infty}$.  (The Ausoni-Rognes computations require more than an
$A_{\infty}$ structure on $BP\langle n\rangle$; the papers 
are written in terms of an $E_{\infty}$ structure, though somewhat
less will suffice).
\end{enumerate}
These localization sequences give the opportunity to continue the
Rognes program, with attention focused on the computation of
$TC(BP_{n})$ and evaluation of the transfer map.

%*****

One of the interesting aspects in the construction of the localization
sequences is the construction of the relative terms such as $THH(ku|KU)$ and
$TC(ku|KU)$: these relative terms ``mix'' the weak equivalences in the
category of $ku$-modules with the weak equivalences in the category of
$KU$-modules, in a way which does not arise in algebraic $K$-theory.
This mixing is the reason why there are two different localization
sequences.  In order to explain these sorts of relative terms,
Rognes~\cite{RognesLog} has developed a theory of log ring spectra
motivated by the appearance of log rings in the work of Hesselholt and
Madsen \cite{HM3}.  We expect that our relative terms agree with the
log $THH$ and $TC$ defined by Rognes.

Because our primary interest is the construction and explanation of
the localization sequences above, we have taken a technical shortcut that
drastically simplifies the theory.  In Section~\ref{secdefwald}, we
introduce the concept of a simplicially enriched Waldhausen category in
which the Waldhausen structure and the simplicial mapping spaces
satisfy strong consistency hypotheses.  The motivating example of such
a category is a subcategory of the cofibrant objects in a simplicial
model category with all objects fibrant; the model structure on the
module categories of \cite{EKMM} satisfy this condition.  For the
majority of the paper we work only with simplicially enriched
Waldhausen categories.  In Section~\ref{futuresec}, we argue that
simplicially enriched Waldhausen categories are not unduly restrictive
by showing that a closed Waldhausen subcategory of a Waldhausen
category that admits factorization is equivalent to a simplicially
enriched Waldhausen category (in fact, a simplicial model category
where every object is fibrant).  This equivalence is functorial up to
a zigzag of natural weak equivalences.

Although we have taken Waldhausen categories for the basic input to $THH$ and
$K$-theory in this paper,  alternatively, one could take
quasi-categories as the basic input.  At this stage, the
quasi-category approach would require serious background treatment of
the $THH$ of quasi-categories, which is not yet
formalized in the literature.  On the other hand, since our first step
is to replace a general Waldhausen category with a stable simplicial
model category, such a background treatment would be essentially
independent of the main work in this paper.  

In this paper, whenever we work with topological spaces, the reader
should understand that we are working in the category of compactly
generated weak Hausdorff spaces.  We use the words ``topological'' or
``topological space'' to highlight when we are using
topological spaces rather than simplicial sets; these words should not be
construed to imply the use of general topological spaces rather than
compactly generated weak Hausdorff spaces.

\iffalse
% (leave sections out of TOC)
\addtocontents{toc}{\protect\setcounter{tocdepth}{0}}

\section*{Outline and Overview}
goes here when all else is finalized.
\fi

\mainmatter
%    Include main chapters here.

% tp-set-source: tw-whole.tex
% tex-command: latex
% ultex-add-defs: list

%
% Copyright (C) 2007-13  Andrew J. Blumberg and Michael A. Mandell
%

\chapter{Review of $THH$, $TR$, and $TC$}
\label{tw-thhmv}

In this chapter we review the construction and basic properties of
$THH$, $TR$, and $TC$ of spectral categories.  We begin in
Section~\ref{secspecc} by reviewing the definition of spectral
categories (in symmetric spectra) and setting some conventions for the
rest of the paper.  In Section~\ref{secthh}, we review the
construction of $THH$ of spectral categories along the lines first
described by B\"okstedt \cite{Bokstedt} and the construction of $TR$
and $TC$ from $THH$.  In Sections~\ref{secthhprop}--\ref{secdwm}, we review the
fundamental invariance properties of the $THH$ of spectral categories,
including invariance under DK-equivalences, thick closure, and Morita
equivalence.  

None of the material in this chapter is new; it has previously
appeared in substantially similar form in the authors' previous paper
on $THH$, $TR$, and $TC$ of spectral categories
\cite{BlumbergMandellTHHLoc} and is reviewed here for easy reference.
Specifically, Section~\ref{secspecc} streamlines and rewrites
\cite[\S2]{BlumbergMandellTHHLoc} for symmetric spectra of topological
spaces.  Section~\ref{secthh} is based on and closely follows
\cite[\S3]{BlumbergMandellTHHLoc}, while
Sections~\ref{secthhprop}--\ref{secdwm} review the main results of
\cite[\S5--7]{BlumbergMandellTHHLoc} with most proofs omitted.

%%%%%%%%%%%%%%%%%%%%%%%%%%%%%%%%%%%%%%%%
\section{Review of spectral categories}\label{secspecc}

This section reviews the definition of and sets conventions for
spectral categories that we use throughout the remainder of the
paper. Although our most common constructions naturally live in the
context of symmetric spectra of simplicial sets, we occasionally need
symmetric spectra of topological spaces.

\begin{defn}
A \term{spectral category} is a category enriched over symmetric
spectra (of topological spaces).  Specifically, a spectral category $\aC$
consists of: 
\begin{enumerate}
\item A collection of objects $\ob\aC$ (which need not be a small set),
\item A symmetric spectrum $\aC(a,b)$ for each pair of objects $a,b\in \ob\aC$,
\item A unit map $S\to \aC(a,a)$ for each object $a \in \ob\aC$, and
\item A composition map $\aC(b,c)\sma \aC(a,b)\to \aC(a,c)$ for each
triple of objects $a,b,c\in \ob\aC$,
\end{enumerate}
satisfying the usual associativity and unit properties.  We say that a
spectral category is small when the objects $\ob \aC$ form a set.
\end{defn}

The previous definition makes
perfect sense also in the context of symmetric spectra of simplicial
sets (indeed that was the convention in
\cite{BlumbergMandellTHHLoc}); the geometric realization/singular
simplicial set adjunction that converts back and forth between
symmetric spectra of simplicial sets and symmetric spectra of
topological spaces is a (symmetric) monoidal functor and so converts
back and forth between spectral categories in the simplicial and
topological context by application to the mapping spectra.

The definition of spectral functor between spectral categories is the
usual definition of an enriched functor:

\begin{defn}
Let $\aC$ and $\aD$ be spectral categories.  A \term{spectral functor}
$F\colon \aC\to \aD$ is an enriched functor.  Specifically, a spectral
functor consists of:
\begin{enumerate}
\item A function on objects $F\colon \ob\aC\to \ob\aD$, and
\item A map of symmetric spectra $F_{a,b}\colon \aC(a,b)\to
\aD(Fa,Fb)$ for each pair of objects $a,b\in \ob\aC$,
\end{enumerate}
which is compatible with the units and the compositions in the obvious sense.
\end{defn}

We then have the following elementary notion of weak equivalence of spectral
categories.  (The more useful definition of \term{DK-equivalence} of spectral
categories is Definition~\ref{defdkequiv} below.) 

\begin{defn}\label{defwkequiv}
A \term{weak equivalence} of spectral categories is 
spectral functor that is a bijection on objects and a weak equivalence
(stable equivalence of symmetric spectra) on all mapping spectra.
\end{defn}

Small spectral categories generalize ring symmetric spectra and can be
viewed as \term{rings with many objects}.  From that perspective, we
have the following evident concepts of modules and bimodules over
spectral categories: 

\begin{defn}
Let $\aC$ and $\aD$ be spectral categories.  A left
$\aC$-module\index{left module}\index{module}
is a spectral functor from $\aC$ to symmetric spectra.  A right
$\aD$-module\index{right module}
is a spectral functor from $\aD^{\op}$ to symmetric
spectra.  A $(\aD,\aC)$-\textin{bimodule} is a spectral functor from
$\aD^{\op}\sma \aC$ to symmetric spectra; a $\aC$-bimodule is a
$(\aC,\aC)$-bimodule.
\end{defn}

Here $\aD^{\op}$ denotes the spectral category with the same objects
and mapping spectra as $\aD$ but the opposite composition map.  The
spectral category $\aD^{\op}\sma \aC$ has as its objects the cartesian
product of the objects, 
\[
\ob (\aD^{\op}\sma \aC)=\ob\aD^{\op}\times
\ob\aC=\ob\aD\times \ob\aC,
\]
and as its mapping spectra the smash product of the mapping
spectra
\[
(\aD^{\op}\sma \aC)((d,c),(d',c'))=\aD^{op}(d,d')\sma\aC(c,c'),
\]
with unit maps the smash product of the unit maps and composition maps
the smash product of the composition maps for $\aD^{\op}$ and $\aC$.
Explicitly, a $(\aD,\aC)$-bimodule $\fM$ consists of a choice of
symmetric spectrum $\fM(d,c)$ for each $d$ in $\ob \aD$ and $c$ in
$\ob \aC$, together with maps
\[
\aC(c, c') \sma \fM(d,c) \sma \aD(d',d) \to \fM(d',c')
\]
for each $d'$ in $\ob \aD$ and $c'$ in
$\ob \aC$, making the obvious unit and associativity diagrams
commute.  In particular, for any spectral category $\aC$, the mapping
spectra $\aC(-,-)$ define a $\aC$-bimodule. (This example
motivates the convention of listing the right module structure first.)

The work of \cite{SSMonoidalEq} provides the category of
$(\aD,\aC)$-bimodules with a closed model structure.

\begin{prop}\label{propSSME61}\procite[6.1]{SSMonoidalEq}
The category of $(\aD,\aC)$-bimodules forms a closed model category
where the fibrations are the objectwise fibrations and the weak
equivalences are the objectwise weak equivalences in the stable model
structure on symmetric spectra.
\end{prop}

The remainder of this section records some technical observations.  Because we
are working with spaces rather than simplicial sets, we will often
need to assume that base points are non-degenerate (include as
Hurewicz cofibrations) to avoid pathologies. For spectral categories,
modules, and bimodules, we use
the following terminology.

\begin{defn}\label{defnondegen}
A spectral category $\aC$ is \term{non-degenerately based} if each
space $\aC(a,b)(n)$ is non-degenerately based (for all objects $a$,
$b$, and all $n$) and each unit map $S^{0}\to \aC(a,a)(0)$ is a
Hurewicz cofibration (for all objects $a$); otherwise, we say that
$\aC$ is \term{degenerately based}.  A $\aC$-module or $(\aD,\aC)$-bimodule
$\fM$ is \term{non-degenerately based} if each space $\fM(c)(n)$ or
$\fM(d,c)(n)$ is non-degenerately based (for all objects $c,d$ and all
$n$); otherwise, we say that $\fM$ is \term{degenerately based}.
\end{defn}

The geometric realization of a spectral category or module in the
simplicial context is always non-degenerately based.  For an arbitrary
spectral category, we can find a weakly equivalent non-degenerately
based spectral category by taking the geometric realization of the
singular simplicial set functor applied to its mapping spectra,
\[
|\Sing \aC|(a,b):=|\Sing \aC(a,b)|.
\]
When $\fM$ is a $\aC$-bimodule, $|\Sing \fM|$ is a $|\Sing
\aC|$-bimodule.  More generally, for an arbitrary bimodule over a
non-degenerately based spectral category, we can find a weakly
equivalent non-degenerately based replacement by applying the
cofibrant replacement functor of Proposition~\ref{propSSME61}.

Another technical point arises when considering the homotopy groups of
symmetric spectra.  In general the object in the stable category
represented by a symmetric spectrum may not agree with the object
represented by its underlying prespectrum.  This happens for example
for the desuspension spectrum $F_{1}S^{0}$.  In such circumstances,
the only sensible convention is to regard the underlying prespectrum
as being incorrect.  Thus, throughout this paper, we
use the following convention. 

\begin{conv}
The \term{homotopy groups} of a symmetric spectrum $X$ always means the
homotopy groups of $X$ as an object of 
the stable category, i.e., the abelian groups of maps in the stable
category from $S^q$ to $X$ (for $q\in \bZ$), and we will denote these
as $\pi_{q}X$. 
A \term{weak equivalence} of symmetric spectra always means a
weak equivalence in the stable model structure.  
A weak equivalence is then precisely a map that
induces an isomorphism on homotopy groups.
\end{conv}

In practice, in many cases the underlying prespectrum does represent
the correct object in the stable category. We use the following
terminology for this.

\begin{defn}\label{defsemistable}
A symmetric spectrum is \term{semistable} when fibrant
approximation in the stable model structure is a weak equivalence of
underlying prespectra.
\end{defn}

When needed, we can replace an arbitrary small spectral category with
a weakly equivalent spectral category that has the same objects but
has mapping spectra that are $\Omega$-spectra.  For example, we can do
this using \cite[\S 6]{SSMonoidalEq} which constructs a cofibrantly
generated Quillen model category structure on the category of small
enriched categories with a fixed set of objects: The maps in this
category are the spectral functors that are the identity on object
sets, the fibrations are the maps $\aC\to \aD$ that restrict to
fibrations of symmetric spectra $\aC(x,y)\to \aD(x,y)$ for all $x,y$
and the weak equivalences are the maps that restrict to weak
equivalences $\aC(x,y)\to \aD(x,y)$ for all $x,y$.  Following the
terminology of \cite[\S 6]{SSMonoidalEq}: 

\begin{defn}
A small spectral category $\aC$ is
\indexterm{pointwise fibrant}{fibrant, pointwise}
if $\aC(x,y)$ is a fibrant symmetric spectrum (in the stable model
structure) for every pair of objects $x,y$.  Likewise, $\aC$ is said
to be \indexterm{pointwise cofibrant}{cofibrant, pointwise} 
if $\aC(x,y)$ is a cofibrant
symmetric spectrum for every pair of objects $x,y$.  For a spectral
functor of small spectral categories $F\colon \aC\to \aD$ that is the
identity on the object sets, we say that $F$ is a 
\indexterm{pointwise weak equivalence}{weak equivalence, pointwise} or 
\indexterm{pointwise level equivalence}{level equivalence, pointwise} 
if for every pair
of objects $x,y$, the map $F\colon \aC(x,y)\to \aD(x,y)$ is a weak
equivalence or level equivalence, respectively, of symmetric spectra.
\end{defn}

The fibrant replacement functors of \cite[\S 6]{SSMonoidalEq}, though
constructed in the context of a fixed object set still behave well with
respect to spectral functors that are not the identity on object
sets.  We then get the following proposition.

\begin{prop}\label{propfibrep}\procite[6.3]{SSMonoidalEq}
Given a small spectral category $\aC$, there exists a small spectral
category $\aC^{\Omega}$ and a spectral functor $R\colon \aC\to \aC^{\Omega}$
such that:
\begin{enumerate}
\item $\aC^{\Omega}$ has the same objects as $\aC$ and $R$ is the
identity map on objects,
\item $\aC^{\Omega}$ is pointwise fibrant, and
\item $R$ is a pointwise weak equivalence.
\end{enumerate}
Moreover, $(-)^{\Omega}$ and $R$ may be constructed as an endofunctor and
natural transformation on the category of small spectral categories.
\end{prop}

Applying cofibrant replacement in the model structure of \cite[\S
6]{SSMonoidalEq}, we obtain the following complementary proposition.

\begin{prop}\label{propcofrep}\procite[6.3]{SSMonoidalEq}
Given a small spectral category $\aC$, there exists a small spectral
category $\aC^{\Cell}$ and a spectral functor $Q\colon \aC^{\Cell}\to \aC$
such that:
\begin{enumerate}
\item $\aC^{\Cell}$ has the same objects as $\aC$ and $Q$ is the
identity map on objects,
\item $\aC^{\Cell}(x,y)$ is pointwise cofibrant, and 
\item $Q$ is a pointwise level equivalence.
\end{enumerate}
Moreover, $(-)^{\Cell}$ and $Q$ may be constructed as an endofunctor and
natural transformation on the category of small spectral categories.
\end{prop}

The analogous proposition in the setting of bimodules is also helpful.

\begin{prop}\label{propmodcof}
Assume that $\aC$ and $\aD$ are pointwise cofibrant small spectral categories.
If $\fM$ is a cofibrant $(\aD,\aC)$-bimodule, then $\fM$ is objectwise
cofibrant, i.e., 
$\fM(d,c)$ is a cofibrant symmetric spectrum for every $(d,c)$ in
$\aD^{\op}\sma \aC$.
\end{prop}

%%%%%%%%%%%%%%%%%%%%%%%%%%%%%%%%%%%%%%%%
\section{Review of the construction of $THH$, $TR$, and $TC$}\label{secthh}

In this section, we review the definition of $THH$, $TR$, and $TC$ of
small spectral categories.  We begin with a review of the cyclic bar
construction for small spectral categories and the variant defined by
B\"okstedt \cite{Bokstedt} and Dundas-McCarthy \cite{DundasMcCarthy}
necessary for the construction of 
$TC$.  We finish with a brief review of the definition of cyclotomic
spectra and the construction of $TR$ and $TC$.

Let $\aI$ be the category with objects the finite sets
$\bn=\{1,\ldots,n\}$ (including $\bz=\{\}$), and with morphisms the
injective maps.  For a symmetric spectrum $A$, write $A_{n}$ for the
$n$-th space.  The association $\bn\mapsto \Omega^{n}A_{n}$ extends
to a functor from $\aI$ to spaces.  More generally, given symmetric spectra
$A^{0},\dotsc,A^{q}$ and a space $X$, we obtain a functor from
$\aI^{q+1}$ to spaces that sends $\vec\bn=(\bn_{0},\dotsc,\bn_{q})$ to
\[
\Omega^{n_{0}+\dotsb+n_{q}}(A^{q}_{n_{q}} \sma \dotsb \sma
A^{0}_{n_{0}}\sma X),
\]
which is also natural in $X$.  Restricting to the case when $X$ is a
sphere $S^{n}$, we form this into a symmetric spectrum as follows.

% One variable case is \cite[3.1.1]{ShipleyD}
\begin{defn}\procite[4.2.1]{ShipleyD}\label{defD}
Let \specialterm{$D(A^{q},\dotsc,A^{0})$}{D(A} be the symmetric
spectrum with $n$-th space 
\[
D(A^{q},\dotsc,A^{0})(n)=\hocolim_{\vec \bn\in\aI^{q+1}}
\Omega^{n_{0}+\dotsb+n_{q}}(|A^{q}_{n_{q}} \sma \dotsb \sma
A^{0}_{n_{0}}|\sma S^{n}),
\]
and the evident structure maps.
\end{defn}

The following is the main lemma of \cite{ShipleyD}.

\begin{prop}\procite[4.2.3]{ShipleyD}
$D(A^{q},\dotsc,A^{0})$ is canonically isomorphic in the stable
category to the derived smash product of the $A^{i}$.
\end{prop}

This motivates the following definition, Dundas-McCarthy's
Hochschild-Mitchell version of B\"okstedt's variant of the cyclic bar
construction. 

\begin{defn}\label{defn:123}\index{THH (construction)@$THH$ (construction)}
Given a small spectral category $\aC$, a $\aC$-bimodule $\fM$,
and a space $X$, let
\specialterm{$\aG(\aC;\fM;X)_{\vec\protect\bn}$}{G(C;M;X)n} be the functor
from 
$\aI^{q+1}$ to spaces defined on $\vec\bn=(\bn_{0},\dotsc,\bn_{q})$ by
\[
\aG(\aC;\fM;X)_{\vec\bn}=
\Omega^{n_{0}+\dotsb+n_{q}}
(\bigvee \aC(c_{q-1},c_{q})_{n_{q}} \sma \dotsb \sma
\aC(c_{0},c_{1})_{n_{1}}\sma \fM(c_{q},c_{0})_{n_{0}}\sma X),
\]
where the wedge is over the $(q+1)$-tuples $(c_{0},\ldots,c_{q})$ of
objects of $\aC$.  Let 
\[
THH_{q}(\aC;\fM)(X)=\hocolim_{\vec\bn\in\aI^{q+1}} \aG(\aC;\fM;X)_{\vec\bn}.
\]
This assembles into a simplicial space, functorially in $X$, as
follows.  The degeneracy maps are induced by
the unit maps $S^{0}\to \aC(c_{i},c_{i})_{0}$ and the functor 
\[
(\bn_{0},\dotsc,\bn_{q})\mapsto (\bn_{0},\dotsc,\bz,\dotsc,\bn_{q})
\]
from $I^{q+1}$ to $I^{q+2}$.  The face maps are induced by the
two action maps on
$\fM$ (for $d_{0}$ and $d_{q}$) and the composition maps in $\aC$ 
(for $d_{1},\dotsc,d_{q-1}$) together with a functor $\aI^{q+1}\to
\aI^{q}$ induced by the appropriate disjoint union isomorphism
$(\bn_{i},\bn_{i+1})\mapsto \bn$ or $(\bn_{q},\bn_{0})\mapsto \bn$ for
$n=n_{i}+n_{i+1}$ or $n=n_{q}+n_{0}$. We write $THH(\aC;\fM)(X)$ for
the geometric realization.
\end{defn}

$THH(\aC;\fM)(X)$ is a continuous functor in the variable $X$, and so
by restriction to the spheres $S^{n}$ specifies a symmetric spectrum
which we denote $THH(\aC;\fM)$ or $THH(\aC)$ for $\fM=\aC$.  The fact
that the symmetric spectrum $THH$ is the restriction of a continuous
functor implies that it is semistable \cite[8.7]{MMSS} and so the
object that it represents in the stable category agrees with its
underlying prespectrum.  With additional hypotheses of ``convergence''
and ``connectivity'', $THH$ is often an $\Omega$-spectrum; see, for
example, Proposition~2.4 of~\cite{HM2}.

For most homotopical statements about $THH$, we will need to assume
that $\aC$ and $\fM$ are non-degenerately based.  When the unit maps
$S^{0}\to \aC(c,c)(0)$ are cofibrations, the simplicial spaces
$THH\subdot(\aC;\fM)(X)$ are ``\textin{proper}'', meaning that the degeneracy
maps are cofibrations, which is a sufficient for geometric realization
to preserve level weak equivalences.  The following proposition is
then clear since smash products of non-degenerately based spaces
preserve weak equivalences. It allows us to convert statements in
\cite{BlumbergMandellTHHLoc} (which works with spectral categories of
symmetric spectra in the context of simplicial sets) to the current
context of topological spaces.

\begin{prop}\label{propoldpaper}
If $\aC$ is a small non-degenerately based spectral category and $\fM$
is a non-degenerately based $\aC$-bimodule, then the canonical map 
\[
THH(|\Sing \aC|,|\Sing \fM|)(X)\to THH(\aC;\fM)(X)
\]
is a weak equivalence for all $X$.
\end{prop}

As immediate corollaries, we obtain the following the basic properties
of $THH$.

\begin{prop}\procite[3.6]{BlumbergMandellTHHLoc}\label{propTHHwe}
Let $F\colon \aC\to \aC'$ be a weak equivalence of small spectral
categories, $\fM'$ a $\aC'$-bimodule, $F^{*}\fM'$ the
$\aC$-bimodule obtained by restriction of scalars, and
$\fM\to F^{*}\fM'$ a weak equivalence of $\aC$-bimodules.  Then
the induced map $THH(\aC;\fM)\to THH(\aC';\fM')$ is a weak equivalence.
\end{prop}

\begin{prop}\procite[3.7]{BlumbergMandellTHHLoc}\label{propnibus}
Let $\aC$ be a small non-degenerately based spectral category.
\begin{enumerate}
\item A weak equivalence of non-degenerately based $\aC$-bimodules
$\fM\to \fM'$ induces a weak 
equivalence  $THH(\aC;\fM)\to THH(\aC;\fM')$.
\item A cofibration sequence of non-degenerately based $\aC$-bimodules
$\fM\to \fM'\to 
\fM''\to \Sigma \fM$ induces a homotopy cofiber sequence on $THH$.
\item A fibration sequence of non-degenerately based
$\aC$-bimodules $\Omega 
\fM''\to \fM\to \fM'\to \fM''$ induces a homotopy fibration sequence
on $THH$. 
\end{enumerate}
\end{prop}

Additionally, we observe the following two results that are useful in
arguments and applications in later chapters.

\begin{prop}\label{propthhcolim}
Let $\aC_{0}\to \aC_{1}\to \cdots$ be a sequence of spectrally
enriched functors of non-degenerately based spectral categories and
assume that either the functors are closed inclusions on
mapping spectra or are induced by geometric realization from spectral
functors of spectral categories enriched in symmetric spectra of
simplicial sets.  Let $\aC=\colim \aC_{n}$ and let $\fM$ be a
non-degenerately based $\aC$ bifunctor.  Then the induced map
\[
\hocolim THH(\aC_{n};\fM)\to THH(\aC;\fM)
\]
is a weak equivalence.
\end{prop}

\begin{proof}
The map $\hocolim \aG(\aC_{n};\fM;X)_{\vec m}\to \aG(\aC;\fM;X)_{\vec
m}$ is a weak equivalence for every $X$, $\vec m$.
\end{proof}

\begin{prop}\label{propthhreal}
Let $\aC\subdot$ be a simplicial object in non-degenerately based
spectral categories in which all the faces and degeneracies are the
identity on objects and are Hurewicz cofibrations on each space of
each mapping spectrum.  Then the canonical map $|THH(\aC\subdot)|\to
THH(|\aC\subdot|)$ is a weak equivalence.
\end{prop}

\begin{proof}
For each $\vec n$ and spectral category $\aC$, consider the symmetric spectrum
\[
\aG(\aC)_{\vec\bn}=
\Omega^{n_{0}+\dotsb+n_{q}}
(\bigvee \aC(c_{q-1},c_{q})_{n_{q}} \sma \dotsb \sma
\aC(c_{0},c_{1})_{n_{1}}\sma \aC(c_{q},c_{0})_{n_{0}}\sma S),
\]
the symmetric spectrum obtained from assembling the spaces
$\aG(\aC,\aC,S^{n})$ of Definition~\ref{defn:123}.
Then 
\[
THH(\aC)\iso |\hocolim_{\aI\supdot} \aG(\aC)_{\vec \bn}|.
\]
For any proper simplicial non-degenerately based space $X\subdot$, the
canonical map 
\[
|\Omega^{n}(X\subdot \sma S)|\to \Omega^{n}|X\subdot \sma S|
\]
is a weak equivalence, indeed a level equivalence after level $n$
\cite[12.3]{GILS}, and it follows that $|THH(\aC\subdot)|\to
THH(|\aC\subdot|)$ is a weak equivalence.
\end{proof}

We now give a minimal review of the definition of $TR$ and $TC$; we
refer the reader interested in more details to the excellent
discussions of $TR$ and $TC$ in \cite{HM2,HM3}.  For an $S^{1}$-space
$X$, the space $THH(\aC)(X)$ has two $S^{1}$-actions, one coming from
$X$ and the other coming from the cyclic structure.  Using the
diagonal action and restricting to representation spheres $S^{V}$
makes $THH(\aC)(-)$ into an equivariant orthogonal spectrum \cite[\S
II.2]{MMSS}; however, $THH(\aC)$ has even more structure, that of a
\term{cyclotomic spectrum} \cite[\S1.1]{HM3}, \cite[Def.~2.2]{HM2}.
We refer the reader to \cite[\S4]{BlumbergMandellTHHLoc} or
\cite[\S4]{BlumbergMandellCyclo} for a precise definition of the
category of cyclotomic spectra, but in brief the structure on $THH$
derives from the fundamental fixed point map
\[
(THH(\aC)(X))^{H} \to THH(\aC)(X^{H})
\]
for $S^{1}$-spaces $X$ and finite subgroups $H$ of $S^{1}$.
This induces maps in the equivariant stable category
\[
r_{H}\colon \rho^{\#}_{H}\Phi^{H} THH(\aC)\to THH(\aC)
\]
that are non-equivariant weak equivalences.  Here $\Phi^{H}$ denotes
the (derived) geometric fixed point spectrum, and when $H$ is the
subgroup with $n$ elements, $\rho_{H}$ is the $n$-th root isomorphism
$S^{1}\iso S^{1}/H$; $\rho^{\#}_{H}$ converts the $S^{1}/H$-spectrum
$\Phi^{H}THH(\aC)$ back to an $S^{1}$-spectrum via the isomorphism
$\rho$.  Essentially, a cyclotomic spectrum consists of an
$S^{1}$-equivariant spectrum indexed on a complete universe together
with weak equivalences $r_{H}$ of the form above, called
\term{cyclotomic structure maps}, satisfying certain coherence
properties \cite[Def.~2.2]{HM2}, \cite[\S1.1]{HM3}.  By
\cite[4.9]{BlumbergMandellTHHLoc} (and the obvious equivariant
refinement of Proposition~\ref{propoldpaper}), $THH$ defines a functor
from small non-degenerately based spectral categories to the point-set
category of cyclotomic spectra.

For a fixed prime $p$ and each $n$, let $C_{p^{n}}\subset S^{1}$
denote the cyclic subgroup of order $p^{n}$.  We then have maps in the
(non-equivariant) stable category
\[
F,R\colon THH(\aC)^{C_{p^{n}}}\to THH(\aC)^{C_{p^{n-1}}}
\]
where $F$ is the inclusion of the fixed points and $R$ is the map
induced by the composite of the map from the fixed point spectrum to
the geometric fixed point spectrum $THH(\aC)^{C_{p}}\to
\Phi^{C_{p}}THH(\aC)$ and the cyclotomic structure map
$r_{C_{p}}\colon \Phi^{C_{p}}THH(\aC)\to THH(\aC)$; see
\cite[\S1.1]{HM3}, \cite[\S2.2]{HM2}, or
\cite[\S4]{BlumbergMandellTHHLoc}.  We need functorial point-set
versions of these maps to construct $TC$ as a functor on small
spectral categories.  In \cite{HM3}, the connectivity and convergence
hypotheses used there imply that $THH(\aC)$ is an equivariant
$\Omega$-spectrum relative to the family of finite subsets of $S^{1}$;
the point-set maps $F,R$ in \cite{HM3} are then constructed using the
point-set fixed point spectra as models for the derived fixed point
spectra.  In our context, we need to use an $\Omega$-spectrum
replacement functor in the category of cyclotomic spectra: For such a
functor $Q$, we get appropriate point-set maps 
\[
F,R\colon Q(T)^{C_{p^{n}}}\to Q(T)^{C_{p^{n-1}}}.
\]
which are functorial in the cyclotomic spectrum $T$.

\begin{defn}\label{defTC}\index{TC (construction)@$TC$
(construction)}\index{TR (construction)@$TR$ (construction)}
Let $Q$ be an $\Omega$-spectrum replacement functor in the category of
cyclotomic spectra and write $T(\aC)$ for $Q(THH(\aC))$. Then 
$TR\supdot(\aC)$ is the pro-spectrum $\{T(\aC)^{C_{p^{n}}}\}$ under the
maps $R$, and $TR(\aC)$ is the homotopy limit.  $TC(\aC)$ and
$TC\supdot(\aC)$ are the spectrum and 
pro-spectrum obtained from $TR(\aC)$ and $TR\supdot(\aC)$ as the homotopy
equalizer of the maps $F$ and $R$. 
\end{defn}

Note that a map in the $S^{1}$-equivariant stable category induces a
(non-equivariant) weak equivalence on fixed point spectra for all
finite subgroups of $S^{1}$ if and only if it induces a
(non-equivariant) weak equivalence on geometric fixed point spectra
for all finite subgroups \cite[XVI.6.4]{MayAlaska}.  It follows that a
cyclotomic map of cyclotomic spectra induces a weak equivalence of
fixed point spectra for all finite subgroups of $S^{1}$ if and only if
it is a non-equivariant weak equivalence.  In particular, we obtain
the following proposition.

\begin{prop}\label{propTCequiv}
A spectral functor of small non-degenerately based spectral categories
$\aC\to \aD$ that induces a weak 
equivalence on $THH$ induces a weak equivalence on $TR$ and $TC$.
\end{prop}

Likewise, using the same principle on the cofiber of a map of
cyclotomic spectra, we obtain the following proposition.  Applying
this proposition in examples when $THH(\aC)$ is contractible,
localization cofibration sequences on $TR$ and $TC$ follow from ones on
$THH$.

\begin{prop}\label{proplazy}
For a strictly commuting square of small non-degenerately based
spectral categories 
\[
\xymatrix{%
\aA\ar[r]\ar[d]&\aB\ar[d]\\
\aC\ar[r]&\aD,
}
\]
if the induced square on $THH$ is homotopy cocartesian, then so are the
induced squares on $TR$ and $TC$.
\end{prop}

%%%%%%%%%%%%%%%%%%%%%%%%%%%%%%%%%%%%%%%%
\section{Review of the invariance properties of $THH$}\label{secthhprop}

In this section, we review the invariance properties of $THH$. This
includes invariance under Dwyer-Kan equivalence, cofinal embeddings,
and thick closure. We review the Localization Theorem of
\cite[7.1]{BlumbergMandellTHHLoc} and the closely related theorem
\cite[7.2]{BlumbergMandellTHHLoc} on triangulated quotients formed
from ``localization pairs''.  We review only definitions and
statements in this section and defer to \cite{BlumbergMandellTHHLoc}
for proofs.

\begin{defn}\label{defdkequiv}
Let $F\colon \aC\to \aD$ be a spectral functor.  We say that $F$ is a
\indexterm{Dwyer-Kan embedding}{Dwyer-Kan embedding|seeonly{DK-embedding}} or \term{DK-embedding} when for every
$a,b\in \ob\aC$, the map $\aC(a,b)\to \aD(Fa,Fb)$ is a weak
equivalence.

We say that $F$ is a \indexterm{Dwyer-Kan equivalence}{Dwyer-Kan
equivalence|seeonly{DK-equivalence}} or
\term{DK-equivalence} when $F$ is a DK-embedding and for every $d\in
\ob\aD$, there exists a $c\in\ob\aC$ such that $\aD(-,d)$ and
$\aD(-,Fc)$ represent naturally isomorphic enriched functors from
$\aD^{\op}$ to the stable category.
\end{defn}

We can rephrase this definition in terms of homotopy categories.

\begin{defn}\label{defhtycat}
For a spectral category $\aC$, the \indexterm{homotopy
category}{homotopy category (of a spectral category)}
$\pi_{0}\aC$ is the Ab-category with the same objects, with morphism
abelian groups $\pi_{0}\aC(a,b)$, and with units and composition
induced by the unit and composition maps of $\aC$.  
The \term{graded homotopy category} is the $\mathrm{Ab}_{*}$-category
with objects $\ob\aC$ and morphisms $\pi_{*}\aC(a,b)$.
\end{defn}

The following proposition gives an equivalent formulation of
DK-equivalence in terms of homotopy categories.

\begin{prop}\label{propdkequiv}
A spectral functor $\aC\to \aD$ is a Dwyer-Kan equivalence if and only
if it induces an equivalence of graded homotopy categories
$\pi_{*}\aC\to\pi_{*}\aD$.
\end{prop}

We then have the following invariance property for DK-equivalences.

\begin{thm}\label{thmdkequiv}\procite[5.9]{BlumbergMandellTHHLoc}
A DK-equivalence of small non-degenerately based spectral categories
$\aC\to \aD$ induces a weak equivalence $THH(\aC)\to THH(\aD)$. 
\end{thm}

We also have the following more general theorem for bimodule
coefficients.  In the statement, the $\aC$-bimodule $F^{*}\fN$ is the
bimodule obtained by restriction of scalars; it is the spectral
functor from $\aC^{\op}\sma \aC$ to symmetric spectra defined by first 
applying $F$ to each variable and then applying $\fN$.

\begin{thm}\label{thmbimodequiv}\procite[5.10]{BlumbergMandellTHHLoc}
Let $F\colon \aC\to \aD$ be a DK-equivalence of small non-degenerately based spectral categories, $\fM$ a
$\aC$-bimodule and $\fN$ a $\aD$-bimodule.  A weak
equivalence $\fM\to F^{*}\fN$ induces a weak equivalence
$THH(\aC;\fM)\to THH(\aD;\fN)$.
\end{thm}

The next theorem generalizes from DK-equivalences to cofinal
DK-embeddings. For objects $a$ and $c$ of $\aD$, say that $c$ is a
\term{homotopy factor} of $a$ if it is a factor in the graded homotopy
category $\pi_{*}\aD$, i.e., if there exists an object $b$ in $\aD$
and a natural isomorphism $\pi_{*}\aD(-,c)\iso \pi_{*}\aD(-,a) \times
\pi_{*}\aD(-,b)$ of contravariant functors from $\pi_{*}\aD$ to the
category of graded abelian groups.  We say that a spectral functor
$F\colon \aC\to \aD$ is \term{homotopy cofinal} if it induces weak
equivalences on mapping spaces and each object of $\aD$ is a homotopy
factor of the image of some object in $\aC$.  The following is the
most basic Morita invariance result for $THH$.

\begin{thm}\label{thmfactor}\procite[5.11]{BlumbergMandellTHHLoc}
A homotopy cofinal spectral functor $\aC\to \aD$ of small
non-degenerately based spectral categories induces a weak
equivalence $THH(\aC)\to THH(\aD)$.
\end{thm}

The previous theorem generalizes further to the ``thick closure''.
This is easiest to state and to explain in the context of
pretriangulated spectral categories, which we now review.

\begin{defn}\label{deftriang}\procite[5.4]{BlumbergMandellTHHLoc}
A spectral category $\aC$ is \term{pretriangulated} means:
\begin{enumerate}
\item There is an object $0$ in $\aC$ such that the right $\aC$-module
$\aC(-,0)$ is homotopically trivial (weakly 
equivalent to the constant functor with value the one-point symmetric
spectrum $*$).
\item Whenever a right $\aC$-module $\fM$ has the property that $\Sigma\fM$
is weakly equivalent to a representable $\aC$-module $\aC(-,c)$ (for
some object $c$ in $\aC$), then
$\fM$ is weakly equivalent to a representable $\aC$-module
$\aC(-,d)$ for some object $d$ in $\aC$.
\item Whenever the right $\aC$-modules $\fM$ and $\fN$ are weakly equivalent to
representable $\aC$-modules $\aC(-,a)$ and $\aC(-,b)$ respectively, then the
homotopy cofiber of any map of right $\aC$-modules $\fM\to \fN$ is weakly
equivalent to a representable $\aC$-module.
\end{enumerate}
\end{defn}

The first condition ensures the existence of a zero object in the
homotopy category $\pi_{0}\aC$: the usual argument shows that the left
module $\aC(0,-)$ is also homotopically trivial (in $\pi_{0}\aC$, the
identity map of $0$ is the same as the zero map).  The second
condition gives a desuspension functor on $\pi_{0}\aC$ and the third
condition in particular produces a suspension functor on $\pi_{0}\aC$:
We choose $\Sigma^{-1}a$ and $\Sigma a$ representing
$\Sigma^{-1}\aC(-,a)$ and $\Sigma \aC(-,a)$, respectively, in the
derived category of right $\aC$-modules.  Then $\Sigma^{-1}a$ and
$\Sigma a$ in particular represent the functors $\pi_{1}\aC(-,a)$ and
$\pi_{-1}\aC(-,a)$, respectively, from $\pi_{0}\aC$ to sets, and so
are unique up to unique isomorphism in $\pi_{0}\aC$. See
\cite[5.4ff]{BlumbergMandellTHHLoc} for more discussion.

The terminology ``pretriangulated'' derives from the fact that the
homotopy category is triangulated.  The third condition above indicates
how to form triangles.

\begin{defn}\label{defnPuppe}
In a pretriangulated spectral category $\aC$, we say that a sequence  
\[
a\to b\to c \to \Sigma a
\]
in $\pi_{0}\aC$ is a 
\indexterm{four term Puppe sequence}{Puppe sequence} 
if there exists
right $\aC$-modules $\aM$ and $\aN$ and a map of right $\aC$-modules
$f\colon \aM\to \aN$ such that the four term Puppe sequence of $f$
\[
\aM\to \aN\to Cf\to \Sigma \aM
\]
in the category of right $\aC$-modules
is isomorphic in the derived category of right $\aC$-modules to the
sequence  
\[
\aC(-,a)\to \aC(-,b)\to \aC(-,c) \to \aC(-,\Sigma a)
\]
such that the isomorphism $\Sigma \aM\to \aC(-,\Sigma a)\iso \Sigma
\aC(-,a)$ is the suspension of the isomorphism $\aM\to \aC(-,a)$.
\end{defn}

\begin{thm}\label{thmtriang}\procite[5.6]{BlumbergMandellTHHLoc}
If the spectral category $\aC$ is pretriangulated, then its homotopy
category is triangulated with distinguished triangles the four term
Puppe sequences.
A spectral functor between pretriangulated spectral
categories induces a triangulated functor on homotopy categories.
\end{thm}

\begin{cor}\label{cortriang}\procite[5.7]{BlumbergMandellTHHLoc}
A spectral functor $\aC\to \aD$ between pretriangulated spectral
categories is a Dwyer-Kan equivalence if and only
if it induces an equivalence of homotopy categories
$\pi_{0}\aC\to\pi_{0}\aD$. 
\end{cor}

The following theorem indicates that there is no loss of generality in
considering spectral subcategories of pretriangulated spectral categories.

\begin{thm}\label{thmtriangenv}\procite[5.5]{BlumbergMandellTHHLoc}
Any small spectral category $\aC$ DK-embeds in a small pretriangulated
spectral category $\tilde \aC$.
\end{thm}

Given a
set $C$ of objects in a pretriangulated spectral category
$\aD$, the \term{thick closure} of $C$ is the set of objects in
the thick subcategory of $\pi_{0}\aD$ generated by $C$.  In terms of
the spectral category $\aD$, the thick closure of $C$ is the smallest
set $\bar C$ of objects of $\aD$ containing $C$ and satisfying:
\begin{enumerate}
\item If $a$ is a homotopy factor of an object of $\bar C$, then $a$ is
in $\bar C$.
\item If the right $\aD$-module $\Sigma \aD(-,a)$ is weakly equivalent
to $\aD(-,c)$ for some $c$ in $\bar C$, then $a$ is in $\bar C$. 
%\item If the right $\aD$-module $\Sigma\aD(-,a)$ is weakly equivalent
%to $\aD(-,c)$ for some $c$ in $\bar C$, then $a$ is in $\bar C$. 
\item If the right $\aD$-module $\aD(-,a)$ is weakly equivalent to the
cofiber of a map of right $\aD$-modules $\fM\to \fM'$ with $\fM$, $\fM'$
weakly equivalent to $\aD(-,c)$, $\aD(-,c')$ for $c,c'$
in $\bar C$, then $a$ is in $\bar C$.
\end{enumerate}
A set is \term{thick} if it is its own thick closure. 

\begin{thm}\label{thmthick}\procite[5.12]{BlumbergMandellTHHLoc}
Let $\aD$ be a pretriangulated spectral category.  Let $C$ be a set of
objects of $\aD$, $\bar C$ its thick closure, and $C'$ a set
containing $C$ and contained in $\bar C$. Let $\aC$ and $\aC'$ be the
full spectral subcategories of $\aD$ on the objects in $C$ and $C'$
respectively.  If $\aC$ and $\aC'$ are non-degenerately based, then
the inclusion $\aC\to \aC'$ induces a weak  
equivalence $THH(\aC) \to THH(\aC')$.
\end{thm}

The next theorem is the Localization Theorem of
\cite[7.1]{BlumbergMandellTHHLoc}.

\begin{thm}[Localization Theorem~{\cite[7.1]{BlumbergMandellTHHLoc}}]
\label{thmgenone}
Let $F\colon \aB\to \aC$ be a spectral functor between small
pretriangulated spectral categories, and let $\aA$ be the full spectral
subcategory of $\aB$ consisting of the objects $a$ such that $F(a)$ is
isomorphic to zero in the homotopy category $\pi_{0}\aC$. If the
induced map from the triangulated quotient $\pi_{0}\aB/\pi_{0}\aA$ to
$\pi_{0}\aC$ is cofinal, then
$THH(\aC)$ is weakly equivalent the homotopy cofiber of $THH(\aA)\to
THH(\aB)$. 
\end{thm}

There is also a version for triangulated quotients.
A \term{localization pair} $(\aB,\aA)$ consists of a
pretriangulated spectral category $\aB$ and a full spectral
subcategory $\aA$ such that $\pi_{0}\aA$ is thick in $\pi_{0}\aB$; we
say that the localization pair is small when the spectral category
$\aB$ is small and non-degenerately based when $\aB$ is
non-degenerately based.  (The requirement that $\aA$ be thick is for
convenience rather than necessity by Theorem~\ref{thmthick}.)  The
following theorem says that the cofiber of $THH$ is an
invariant of a localization pair.

\begin{thm}\procite[7.2]{BlumbergMandellTHHLoc}\label{thmgentwo}
Let $F\colon (\aB_{1},\aA_{1})\to (\aB_{2},\aA_{2})$ be a map of
small non-degenerately based localization pairs.  If the induced
map of triangulated quotients 
\[
\pi_{0}\aB_{1}/\pi_{0}\aA_{1} \to \pi_{0}\aB_{2}/\pi_{0}\aA_{2}
\]
is cofinal, then the induced map
\[
C(THH(\aA_{1})\rightarrow THH(\aB_{1})) \quad \to \quad
C(THH(\aA_{2})\rightarrow THH(\aB_{2}))
\]
is a weak equivalence.
\end{thm}

\section{The Dennis-Waldhausen Morita Argument}\label{secdwm}

The main tool in the proof of $THH$ invariance results is a trick due
to Dennis and Waldhausen \cite[p.~391]{WaldhausenA2} that we review in
this section.  We need it in the proof of the Sphere Theorem in
Section~\ref{secpfsphere}.  The argument is based on an explicit
bisimplicial construction, which uses the Hochschild-Mitchell complex
in place of $THH$.

\begin{defn}\index{Hochschild-Mitchell complex}\index{Ncy(C;M)@$\THM(\aC;\fM)$}
For a small spectral category $\aC$ and $\aC$-bimodule $\fM$,
let
\[
\THM_{q}(\aC;\fM)=\bigvee \aC(c_{q-1},c_{q}) \sma \dotsb \sma
\aC(c_{0},c_{1}) \sma \fM(c_{q},c_{0}),
\]
where the sum is over the $(q+1)$-tuples $(c_{0},\dotsc,c_{q})$ of
objects of $\aC$.  This becomes a simplicial object in symmetric
spectra using the usual
cyclic bar construction face and degeneracy maps: The unit
maps of $\aC$ induce the degeneracy maps, and the two action maps on
$\fM$ (for $d_{0}$ and $d_{q}$) and the composition maps in $\aC$ (for
$d_{1},\dotsc,d_{q-1}$) induce the face maps.  We denote the 
geometric realization symmetric spectrum as $\THM(\aC;\fM)$ and
write $\THM(\aC)$ for 
$\THM(\aC;\aC)$. 
\end{defn}

The following proposition, which is essentially the ``many objects''
version of \cite[4.2.8-9]{ShipleyD}, follows from
Proposition~\ref{propoldpaper} and the
theory developed in \cite{ShipleyD}.  It allows us to sometimes
substitute the Hochschild-Mitchell complex for $THH$

\begin{prop}\label{propTHHvsTHM}\procite[3.5]{BlumbergMandellTHHLoc}
There is a natural map in the stable category from $THH(\aC;\fM)$ to
$\THM(\aC;\fM)$ that is an isomorphism when $\aC$ is pointwise
cofibrant.
\end{prop}

In addition to the Hochschild-Mitchell complex, we also need the
two-sided bar construction. 

\begin{defn}\label{defn:tsbc}
Let $\aC$ be a small spectral category, $\fM$ a right $\aC$-module, and $\fN$
a left $\aC$-module.  The \term{two-sided bar construction}
$\TB(\fM;\aC;\fN)$ is the geometric realization of the simplicial
symmetric spectrum $\TB\subdot(\fM;\aC;\fN)$, where
\[
\TB_{q}(\fM;\aC;\fN)
=\bigvee \fM(c_{q})\sma \aC(c_{q-1},c_{q}) \sma \dotsb \sma
\aC(c_{0},c_{1}) \sma \fN(c_{0}),
\]
where the sum is over the $(q+1)$-tuples $(c_{0},\dotsc,c_{q})$ of
objects of $\aC$.  We make this a simplicial object with the usual
two-sided bar construction face and degeneracy maps: the zeroth
face map is induced by the action of $\aC$ on $\aN$, the last face map
is induced by the action of $\aC$ on $\aM$, and the remaining face maps
are induced by the composition in $\aC$. The degeneracy maps are induced
by the unit maps $S\to \aC(c_{i},c_{i})$.
\end{defn}

The following proposition is the \indexterm{Dennis-Waldhausen 
Morita Argument}{Dennis-Waldhausen Morita Argument|(}.  
In the statement (and elsewhere when necessary for
clarity), we write 
\[
\TB(\fM(x); x,y \in \aC; \fN(y))
\qquad \text{and}\qquad 
\THM(x,y\in \aC;\fP(x,y))
\]
for $\TB(\fM;\aC;\fN)$ and $\THM(\aC;\fP)$, especially when $\fM$,
$\fN$, and/or $\fP$ depend on other variables.   

\begin{prop}[Dennis-Waldhausen Morita Argument 
{\cite[6.2]{BlumbergMandellTHHLoc}}]\label{propcoremorita}
Let $\aC$ and $\aD$ be small spectral categories. Let $\fP$ be
a $(\aD,\aC)$-bimodule and $\fQ$ a $(\aC,\aD)$-bimodule.  Then there
is a natural isomorphism of symmetric spectra 
\[
\THM(\aC,\TB(\fP,\aD,\fQ))\iso \THM(\aD;\TB(\fQ,\aC,\fP)),
\]
that is, 
\begin{multline*}
\THM(x,y\in \aC;\TB(\fP(w,y);w,z\in \aD;\fQ(x,z)))\\
\iso
\THM(w,z\in \aD;\TB(\fQ(x,z);x,y\in \aC;\fP(w,y))).
\end{multline*}
\end{prop}

As the proof is easy, we repeat it here.

\begin{proof}
We can identify both symmetric spectra
\[
\THM(\aC;B(\fP;\aD;\fQ))\qquad \text{and}\qquad
\THM(\aD;B(\fQ;\aC;\fP))
\]
as the diagonal of the bisimplicial spectrum
with $(q,r)$-simplices as pictured.
\[
\xymatrix@=1pc{%
&\aC(c_{q-1},x)\sma \dotsb \sma \aC(y,c_{1})\\
\fQ(x,z)
\ar@{{}{}{}}[dr]_(.3){\relax\textstyle\sma}
\ar@{{}{}{}}[ur]^(.3){\relax\textstyle\sma}
&&\fP(w,y)
\ar@{{}{}{}}[dl]^(.3){\relax\textstyle\sma}
\ar@{{}{}{}}[ul]_(.3){\relax\textstyle\sma}
\\
&\aD(z,d_{1})\sma \dotsb \sma \aD(d_{r-1},w)
}
\]
These two constructions are therefore canonically isomorphic in the
point-set category of symmetric spectra.
\end{proof}
\index{Dennis-Waldhausen Morita Argument|)}

The following lemma complements Proposition~\ref{propcoremorita} in
the applications.  Its proof is the usual simplicial contraction (see
for example \cite[9.8]{GILS}) and requires no cofibrancy or
non-degenerate base point hypotheses.

\begin{lem}[Two-Sided Bar Lemma]\label{lemtwobar}
Let $\aC$ be a small spectral category, let $\aM$ be a right
$\aC$-module, and let $\aN$ be a left $\aC$-module. For any object $c$
in $\aC$, the composition maps
\[
\TB\subdot(\fM;\aC;\aC(c,-))\to \fM(c)
\qquad\text{and}\qquad 
\TB\subdot(\aC(-,c);\aC;\fN)\to \fN(c)
\]
are simplicial homotopy equivalences.
\end{lem}

The applications we need are the following.

\begin{thm}\label{thmdwmapp}
Let $F\colon \aC\to \aD$ be a spectral functor between pointwise cofibrant
spectral categories and let $\fL$ be the $\aD$-bimodule
\[
\fL(a,b)=B(\aD(F(-),b),\aC,\aD(a,F(-))).
\]
Then $THH(\aD;\fL)$ is weakly equivalent to $THH(\aC;F^{*}\aD)$.
\end{thm}

\begin{proof}
The proof is essentially the same as the proof of
\cite[7.6]{BlumbergMandellTHHLoc}. 
It suffices to produce a weak equivalence between $\THM(\aD;\fL)$ and
$\THM(\aC;F^{*}\aD)$.  For this we apply Proposition~\ref{propcoremorita} with $\aP=\aC$, and $\aQ=\aC$ to obtain a natural isomorphism
\[
\THM(\aD;\fLL)=
\THM(\aD;\TB(\aD;\aC;\aD))
\iso 
\THM(\aC;F^{*}\TB(\aD;\aD;\aD)).
\]
The natural map
\[
THH(\aC;F^{*}\TB(\aD;\aD;\aD))\to THH(\aC;F^{*}\aD)
\]
is a weak equivalence 
by the Two-Sided Bar Lemma~\ref{lemtwobar}.
\end{proof}

\begin{thm}\label{thmcompcrit}\procite[6.4]{BlumbergMandellTHHLoc}
Let $\aC$ and $\aD$ be small spectral categories and let
$F\colon \aC \to \aD$ be a spectral functor. Let $\fM$ be a
$\aC$-bimodule, $\fN$ a $\aD$-bimodule and $\fM\to
F^{*}\fN$ a weak equivalence.  Assume that $\aC$ and $\aD$ are
pointwise cofibrant and that $\fM$ and $\fN$ are
non-degenerately based.  If the map of symmetric spectra
\[
\TB(\aD(F-,z);\aC;\fN(w,F-))\to 
\TB(\aD(-,z);\aD;\fN(w,-))
\]
is a weak equivalence for each fixed $w$,$z$ in $\aD$.  
Then the map
\[
THH(\aC;\fM)\to THH(\aD;\fN)
\]
is a weak equivalence.
\end{thm}

As the proof is identical to the proof
of~\cite[6.4]{BlumbergMandellTHHLoc}, we omit it here.

\chapter{$THH$ and $TC$ of simplicially enriched {W}aldhausen categories}
\label{tw-wald}

A \term{Waldhausen category} consists of a category $\aC$ together
with a (chosen) zero
object $*$, a
subcategory of \noindexterm{cofibrations} $co\aC$, and a subcategory of
\noindexterm{weak equivalences} that satisfy the following properties
\cite[\S1.1--1.2]{WaldhausenKT}: 
\begin{enumerate}
\item (Cof 1, Weq 1) $co\aC$ and $w\aC$ contain all the isomorphisms.
\item (Cof 2) For every object $a$, the map $*\to a$ is a cofibration.
\item (Cof 3) Cofibrations admit cobase change: If $a\to b$ is a
cofibration, and $a\to c$ is any map, then $b\cup_{a}c$ exists and
$c\to b\cup_{a}c$ is a cofibration.
\item (Weq 2) \term{Gluing Axiom}. Given a commutative diagram
\[
\xymatrix@-1pc{%
b\ar[d]_{\sim}&a\ar@{ >->}[l]\ar[r]\ar[d]_{\sim}&c\ar[d]_{\sim}\\
b'&a'\ar@{ >->}[l]\ar[r]&c'
}
\]
where the leftward arrows are cofibrations and the vertical arrows are
weak equivalences, the induced map 
\[
b\cup_{a}c\to b'\cup_{a'}c'
\]
is a weak equivalence.
\end{enumerate}
Waldhausen~\cite[\S1.3]{WaldhausenKT} constructs the algebraic
$K$-theory spectrum associated to a Waldhausen category using the
$\Sdot$ construction (which we review in Section~\ref{sec:moore1}
below).  The purpose of this chapter is to construct $THH$ and $TC$
for Waldhausen categories that have an additional compatible
simplicial enrichment.
(We extend this definition to Waldhausen categories much more broadly
in Chapter~\ref{tw-gen}.)

The contents of the chapter are as follows.
Section~\ref{secdefwald} defines simplicially enriched, enhanced
simplicially enriched, and simplicially tensored Waldhausen
categories, giving some examples.  Section~\ref{secspec} constructs
spectral categories from simplicially enriched Waldhausen categories.
Section~\ref{sec:moore1} reviews the $\Sdot$ construction and
introduces the \noindexterm{Moore nerve} construction, which is a
version of the nerve construction that behaves better homotopically on
enriched categories.  Section~\ref{sec:spMdot} introduces the
Moore $\Spdot$ construction and iterated $\Spdot$, which generalizes
the iterated $\Sdot$ construction and is needed for the construction
of the cyclotomic trace.  Section~\ref{secdefthh} constructs
$THH$, $TR$, and $TC$ for simplicially enriched Waldhausen categories
and the cyclotomic trace from $K$-theory to $TC$.

\section{Simplicially enriched {W}aldhausen categories}\label{secdefwald}

In this section we introduce the structure of a simplicially enriched
Waldhausen category.  This structure compatibly combines a simplicial
enrichment with a Waldhausen structure in a way that we make precise
in Definition~\ref{defsimpwaldcat}.  Although this structure suffices
for us to define an associated spectral category in the next section,
more conditions are necessary to ensure that the homotopy theory of
the enrichment matches up with the intrinsic homotopy theory of the
Waldhausen category; we make these conditions precise in the
definition of DK-compatible enrichment in
Definition~\ref{defdkcompat}.  In practice,
and as we explain in Section~\ref{futuresec}, without much loss of
generality, we typically have the stronger structures that we describe
in Definitions~\ref{deftenswaldcat} and~\ref{defusuwaldcat}.   We
begin with the most basic structure in the following definition. 

\begin{defn}\label{defsimpwaldcat}
A \term{simplicially enriched Waldhausen category} consists of a
category $\aC=\aC\subdot$ enriched in simplicial sets together with a
Waldhausen category structure on $\aC_{0}$ such that:
\begin{enumerate}
\item The zero object $*$ in $\aC_{0}$ is a zero object for $\aC$,
\item Pushouts over cofibrations in $\aC_{0}$ are pushouts in $\aC$,
\item Cofibrations $x\to y$ induce Kan fibrations
$\aC(y,z)\to\aC(x,z)$ for all objects $z$, and
\item A map $x\to y$ is a weak equivalence if and only if $\aC(y,z)\to
\aC(x,z)$ is a weak equivalence for all objects $z$ if and only if
$\aC(z,x)\to \aC(z,y)$ is a weak equivalence for all objects $z$.
\end{enumerate}
An \term{enriched exact functor} between such categories is a simplicial
functor $\phi \colon \aC\to \aD$ that 
restricts to an exact functor of Waldhausen categories $\aC_{0}\to
\aD_{0}$. 
\end{defn}

Since the initial map $* \to x$ is always a cofibration in a
Waldhausen category, Definition~\ref{defsimpwaldcat} implies that all
the mapping spaces $\aC(x,y)$ are Kan complexes.  The fact that weak
equivalences are detected on the simplicial mapping spaces implies
that weak equivalences in $\aC_{0}$ are closed under retracts and
satisfy the two out of three property.

As explained by Dwyer and Kan, any category with a
subcategory of weak equivalences has an intrinsic homotopy theory
in terms of a functorial simplicially enriched category called the
Dwyer-Kan simplicial localization \cite{DKSimpLoc}.  Technically, we
will use exclusively the variant called the hammock localization
\cite{DKHammock}, which we will denote by $L$. Then for a simplicial
Waldhausen category $\aC$, the Dwyer-Kan simplicial localization of
the underlying category with weak equivalences, denoted $L\aC_{0}$, provides a
second simplicially enriched category expanding $\aC_{0}$.  In
general, we see no reason why these two simplicial enrichments should
be equivalent; we therefore introduce the following
terminology.

\begin{defn}\label{defdkcompat}
Let $\aC$ be a simplicially enriched Waldhausen category.  We say
that $\aC$ is \term{DK-compatible} if for all objects $x,y$ in $\aC$,
the maps
\[
\aC(x,y)\to \diag L\aC\subdot(x,y)\from L\aC_{0}(x,y)
\]
are weak equivalences of simplicial sets. Here we regard $\aC_{n}$ as
a category with weak equivalences by declaring a map in
$\aC_{n}$ to be a weak equivalence if and only if some (or,
equivalently, every) iterated face map takes it to a weak equivalence
in $\aC_{0}$.
\end{defn}

As in Definition~\ref{defdkequiv}, for categories enriched in
simplicial sets, spaces, or spectra, an 
enriched functor $\phi \colon \aC\to \aD$ is called a
\term{DK-embedding} when it induces a weak equivalence $\aC(x,y)\to
\aD(\phi(x),\phi(y))$ for all objects $x$, $y$.  A DK-embedding is a
\term{DK-equivalence} when it induces an equivalence $\pi_{0}\aC\to
\pi_{0}\aD$ on categories of components.  On the other hand, for
discrete categories $\aC_{0}$ and $\aD_{0}$ with subcategories of weak
equivalences, a functor $\aC_{0}\to\aD_{0}$ that preserves weak
equivalences is called a \term{DK-embedding} or \term{DK-equivalence}
when it induces one on the Dwyer-Kan simplicial localizations.
The main purpose of the previous definition is the following easy
observation.

\begin{prop}\label{propequiv}
Let $\aC$ and $\aD$ be simplicially enriched Waldhausen categories and
$\phi \colon \aC\to \aD$ a simplicial functor (not necessarily
exact).  Then:
\begin{enumerate}
\item $\phi_{0}\colon \aC_{0}\to \aD_{0}$ preserves weak equivalences.
\item Assume furthermore that $\aC$ and $\aD$ are both DK-compatible.
Then $\phi$ is a DK-embedding or DK-equivalence of simplicially
enriched categories if and only if
$\phi_{0}$ is a DK-embedding or DK-equivalence (respectively) of
categories with weak equivalences.
\end{enumerate}
\end{prop}

The following is an easy but important class of examples of
DK-compatible simplicially enriched Waldhausen categories.

\begin{example}\label{exexact}
An exact category, or more generally, a Waldhausen category whose weak
equivalences are the isomorphisms 
becomes a DK-compatible simplicially enriched Waldhausen category by
regarding its mapping sets as discrete simplicial sets.
\end{example}

We also have the following less trivial examples.

\begin{example}\label{exofinterest}
Let $\aC$ be a Waldhausen subcategory of cofibrant objects in
simplicial closed model category $\aM$ in which all objects are
fibrant.  Then $\aC$ is a simplicially enriched Waldhausen category
with its natural simplicial mapping spaces and Waldhausen structure
inherited from $\aM$.  If $\aC$ is closed under tensors with finite
simplicial sets, then $\aC$ is a DK-compatible (see 
Theorem~\ref{thmdkdiag} below).  Examples of this type include: 
\begin{enumerate}
\item Finite cell $R$-modules for an EKMM $S$-algebra $R$,
or (for $R$ connective with $\pi_{0}$ noetherian) cell $R$-modules
that have finite stage finitely generated Postnikov towers as in
\cite{BlumbergMandell}.
\item The category of finite cell modules over a
simplicial ring $A$, or the category of finite cell modules built out
of finitely generated projective $A$-modules.
\item The category of simplicial objects on an abelian category with
the ``split-exact'' model structure (where the
cofibrations are the levelwise split 
mono\-morph\-isms and the weak equivalences are the simplicial homotopy
equivalences). 
\item The category of levelwise projectives in the category
of simplicial objects on an abelian category with enough projectives
(with the standard projective model structure).  Likewise, the
opposite category of the levelwise injectives in the category of
cosimplicial objects on an abelian category with enough injectives
(with the standard injective model structure).
\end{enumerate}
\end{example}

In addition to being DK-compatible, the previous class of examples
has an additional structure that we employ to construct non-connective
spectral enrichments in the next section.  We abstract this structure
in the following definition.

\begin{defn}\label{deftenswaldcat}
A \term{simplicially tensored Waldhausen category} is a simplicially enriched
Waldhausen category in which tensors with finite simplicial sets exist
and satisfy the pushout-product axiom.  A tensored exact functor
between simplicially tensored Waldhausen categories is a enriched
exact functor that preserves tensors with finite simplicial sets.
\end{defn}

In the previous definition, the pushout-product axiom
\cite[2.1]{SSAlgMod} asserts that given a cofibration $x \to y$ in
$\aC_{0}$ and a cofibration $A \to B$ of finite simplicial sets, the map
\[
(x \otimes B) \cup_{x \otimes A} (y \otimes A) \to y \otimes B
\]
is a cofibration in $\aC_{0}$.  This axiom implies that the 
usual mapping cylinder construction endows $\aC_{0}$ with a cylinder
functor satisfying the cylinder axiom (in the sense of
\cite[\S1.6]{WaldhausenKT}). The Kan condition on the mapping spaces
combined with the tensor adjunction implies the following proposition.

\begin{prop}\label{propcswprop}
Let $\aC$ be a simplicially tensored Waldhausen category.
\begin{enumerate}
\item For any object $x$ in $\aC$, the tensor $x\otimes(-)$ preserves weak
equivalences in simplicial sets.
\item For any finite simplicial set $X$, the tensor $(-)\otimes
X$ preserves weak equivalences in $\aC$.
\item For objects $x$ and $y$ in $\aC$, the simplicial set $\aC(x,y)$
is canonically isomorphic to $\aC_{0}(x\otimes \Delta[\cdot],y)$.
\end{enumerate}
\end{prop}

Definition~\ref{deftenswaldcat} provides the strongest background structure that
we use; in Section~\ref{futuresec}, we see that Waldhausen categories
quite generally admit equivalent models of this type.  In our study of
the $THH$ localization sequence in Chapter~\ref{tw-loc}, however, we require
slightly more 
flexibility.  Using a simplicially tensored Waldhausen category as 
an ambient category, we will sometimes need to restrict to a subcategory.

\begin{defn}\label{defusuwaldcat}
An \term{enhanced simplicially enriched Waldhausen category} is a pair
$\aA\subset \aC$ where $\aC$ is a simplicially tensored Waldhausen category
and $\aA$ is a full subcategory such that $\aA_{0}$ is a closed
Waldhausen subcategory.  For $\aA\subset \aC$ and $\aB\subset \aD$ 
enhanced simplicially enriched
Waldhausen categories, an \term{enhanced exact functor}
$\aA\to \aB$ is a tensored exact functor of simplicially tensored Waldhausen
categories $\aC\to \aD$ that restricts to a functor $\aA\to \aB$. 
\end{defn}

As in \cite[\S 1.2]{WaldhausenKT}, a \term{Waldhausen subcategory}
$\aA$ is a full subcategory of a Waldhausen category $\aC$ that itself
becomes a Waldhausen category by taking a weak equivalence to be a
weak equivalence in $\aC$ between objects of $\aA$ and a cofibration
to be a cofibration in $\aC$ between objects of $\aA$ for which the
cofiber is in $\aA$ (up to isomorphism).  A
\indexterm{closed}{Waldhausen subcategory!closed}%
\index{closed (Waldhausen subcategory)}
Waldhausen
subcategory is a Waldhausen subcategory $\aA\subset \aC$ that contains
every object of $\aC$ that is weakly equivalent to an object of $\aA$.
An enhanced simplicially enriched Waldhausen category inherits tensors with
homotopically trivial finite simplicial sets (but not necessarily
arbitrary finite simplicial sets) as well as properties~(i) and~(iii)
of Proposition~\ref{propcswprop}.  We also have the following
compatibility result.

\begin{thm}\label{thmdkdiag}
An enhanced simplicially enriched Waldhausen category $\aA\subset \aC$
is DK-compatible.  
\end{thm}

\begin{proof}
Fix objects $a,b$.  Regarding $\aA_{n}(a,b)$ as $\aA_{0}(a\otimes
\Delta[n],b)$, each category $\aA_{n}$ admits a homotopy calculus of
left fractions \cite[6.1]{DKHammock} (see, for example, the argument
for \cite[5.5]{BlumbergMandellUW}) and so we can replace
$L\aA_{n}(a,b)$ with the nerve of the 
category of words of the form $\Wi\C$, which we will temporarily
denote as $L_{n}(a,b)$.  An object of this category consists of a zigzag
\[
a\to x\from b
\]
of maps in $\aA_{n}$, where the map $x\from b$ is a weak equivalence; a
map in this category is a map in $\aA_{n}$ of $x$ that is under $a$
and $b$.  We check that both maps
\[
L_{0}(a,b)\to \diag L\subdot(a,b)\from \aA(a,b)
\]
are weak equivalences (i.e., induce weak equivalences on nerves).

For the map $L_{0}(a,b)\to \diag L\subdot(a,b)$, we show that each
iterated degeneracy $s_{0}^{n}\colon L_{0}(a,b)\to L_{n}(a,b)$ is a
weak equivalence.  Iterating the last face map gives a functor
$\tilde\partial^{n}\colon L_{n}(a,b)\to L_{0}(a,b)$ such that the
composite is the identity on $L_{0}(a,b)$.  We need to check that that
the composite $s_{0}^{n}\tilde\partial^{n}$ on $L_{n}(a,b)$ is a weak
equivalence.  Since both inclusions of $a$ in $a\otimes \Delta[1]$ and
both inclusions of $b$ in $b\otimes \Delta[1]$ are weak equivalences,
they induce weak equivalences
\[
I_{0},I_{1}\colon L_{n}(a\otimes \Delta[1],b\otimes \Delta[1])\to L_{n}(a,b).
\]
The contracting homotopy $c\colon \Delta[n]\times
\Delta[1]\to \Delta[n]$ from the identity map to the inclusion of the
last vertex induces a functor $C\colon \aA_{n}\to \aA_{n}$ sending $x$ to 
$x\otimes \Delta[1]$ as follows: For a map $f\colon x\to y$ in
$\aA_{n}$, viewed as a map $\tilde f\colon x\otimes \Delta[n]\to y$ in
$\aA_{0}$, $C(f)$ is represented by the map 
\[
x\otimes \Delta[1]\otimes \Delta[n]
\iso x\otimes (\Delta [n]\times \Delta[1])
\to y\otimes \Delta[1]
\]
in $\aA_{0}$ induced by $\tilde f$, $c$, and the diagonal map on
$\Delta[1]$. We then get a functor 
\[
C\colon L_{n}(a,b) \to L_{n}(a\otimes \Delta[1],b\otimes \Delta[1]).
\]
The composite functor 
\[
I_{0}\circ C\colon L_{n}(a,b)\to L_{n}(a,b)
\]
admits a natural transformation from the identity functor, and so
induces a homotopy equivalence on nerves.  It follows that $C$
is a weak equivalence.  The composite functor
\[
I_{0}\circ C\colon L_{n}(a,b)\to L_{n}(a,b)
\]
is therefore also a weak equivalence.  We have a natural
transformation from $s_{0}^{n}\tilde\partial^{n}$ to $I_{0}\circ C$,
and so the induced maps on nerves are simplicially homotopic.  This
then shows that $s_{0}^{n}\tilde\partial^{n}$ is a weak equivalence.

It remains to see that the map
$\aA(a,b)\to \diag L\subdot(a,b)$ is a weak equivalence.
We can identify $\diag 
L\subdot(a,b)$ as the diagonal of the bisimplicial set whose
simplicial set of $q$-simplices is
\[
\aA(a,x_{0})\times w\aA(b,x_{0}) \times 
w\aA(x_{0},x_{1})\times \dotsb \times w\aA(x_{q-1},x_{q}),
\]
where $w\aA$ denotes the components with (any, or equivalently, all)
vertices in $w\aA_{0}$, the subcategory of weak equivalences of the
Waldhausen category $\aA_{0}$.  The map $\aA(a,b)\to\diag
L\subdot(a,b)$ 
factors through a bisimplicial map from
the bisimplicial set $X\dsubdot$ whose simplicial set of $q$-simplices
$X_{q\ssdot}$ is 
\[
\aA(a,b)\times w\aA(b,x_{0}) \times 
w\aA(x_{0},x_{1})\times \dotsb \times w\aA(x_{q-1},x_{q}).
\]
The inclusion $\aA(a,b)\to \diag X\dsubdot$ is clearly a simplicial
homotopy equivalence, and the bisimplicial map $X\dsubdot \to
L\subdot(a,b)$ is a degreewise weak equivalence.
\end{proof}

%%%%%%%%%%%%%%%%%%%%%%%%%%%%%%%%%%%%%%%%
\section[Spectral categories of {W}aldhausen categories]%
{Spectral categories associated to simplicially enriched {W}aldhausen categories}
\label{secspec}

In this section we produce for a simplicially enriched Waldhausen
an associated spectral category, which is natural in
enriched exact functors.  The mapping spectra in this category are
prolongations of $\Gamma$-spaces, and as such, are always connective.
For an enhanced simplicially enriched Waldhausen category, we 
associate another spectral category, typically non-connective, using
the suspensions in the ambient simplicially tensored Waldhausen
category; it is natural in enhanced exact functors.  
We also explore the basic properties of these categories.  We begin with
the construction.

\begin{defn}\label{defgammaenrich}
Let $\aC$ be a simplicially enriched Waldhausen category.  
Define $\aC^{\Gamma}$, the $\Gamma$-category associated to $\aC$, to
have objects the objects of $\aC$ and mapping $\Gamma$-spaces
\[
\aC^{\Gamma}_{q}(x,y)=\aC(x,\myop\bigvee_{q}y).
\]
By abuse, we will also write $\aC^{\Gamma}$ for the enrichment in
symmetric spectra obtained by prolongation.
We will refer to $\aC^{\Gamma}$ as the \term{connective spectral
enrichment} of $\aC$ or the connective spectral
category associated to $\aC$.  
\end{defn}

Here the composition
\[
\aC^{\Gamma}_{r}(y,z) \sma \aC^{\Gamma}_{q}(x,y) \to \aC^{\Gamma}_{qr}(x,z).
\]
comes from the $\Sigma_{q}\wr\Sigma_{r}$-equivariant map
\[
\aC(y,\myop\bigvee_{r}z) \to \prod_{q}\aC(y,\myop\bigvee_{r} z) \to \aC(\myop\bigvee_{q}y,\myop\bigvee_{rq}z)
\]
and composition
\[
\aC(\myop\bigvee_{q}y,\myop\bigvee_{rq}z)\sma \aC(x,\myop\bigvee_{q}y) \to \aC(x,\myop\bigvee_{rq}z).
\]
This composition of $\Gamma$-spaces then induces the composition 
on the associated symmetric spectra.  The following proposition is
immediate from the construction. 

\begin{prop}\label{propfunctG}
For simplicially enriched Waldhausen
categories $\aC$ and $\aD$, an enriched exact functor $\phi \colon
\aC\to \aD$ induces a spectral functor $\phi^{\Gamma}\colon
\aC^{\Gamma} \to  \aD^{\Gamma}$.  If $\aC$ and $\aD$ are DK-compatible
and $\phi$ is a DK-embedding or DK-equivalence, then so is
$\phi^{\Gamma}$.
\end{prop}

In general, we can not expect the $\Gamma$-spaces
$\aC^{\Gamma}(x,y)$ to be special or very special.  On the other hand,
as a prolongation of a $\Gamma$-space, the associated symmetric
spectrum is semistable (Definition~\ref{defsemistable}), meaning that
it represents the same object in the stable category as its underlying
spectrum.

\begin{prop}\label{propsemistableG}
The mapping symmetric spectra in $\aC^{\Gamma}$ are \textin{semistable}.
\end{prop}

\begin{example}\label{remexact}
For $\abE$ be an exact category, simplicially enriched as in Example~\ref{exexact},
\[
\abE^{\Gamma}_{q}(x,y) = \abE(x, \bigoplus_{i=1}^{q} y)\iso \prod_{i=1}^{q}\abE(x,y).
\]
Prolonging to symmetric spectra, we get 
\[
\abE^{\Gamma}(x,y)(n) = \abE(x,y) \otimes \tilde{\bZ}[S^n],
\]
where $\tilde{\bZ}[X] = \bZ[X] / \bZ[*]$.  This is 
precisely the spectral category associated to an exact category
studied by Dundas-McCarthy~\cite{DundasMcCarthy} and Hesselholt-Madsen~\cite{HM3}. 
\end{example}

When $\aC$ is a simplicially tensored Waldhausen category, we can
construct another enrichment in symmetric spectra using suspensions:
for an object $x$ in $\aC$, let
$\Sigma x$ be the cofiber of the map 
\[
x\otimes \partial \Delta[1]\to x\otimes \Delta[1].
\]
Suspension defines a tensored
exact functor from $\aC$ to itself.
Commuting colimits and tensors, and applying the associativity
isomorphism for tensors, we can describe the iterated
suspension $\Sigma^{n}x$ as the cofiber of the map
\[
x \otimes \partial(\Delta[1]^{n})\to x\otimes \Delta[1]^{n},
\]
where $\Delta[1]^{n}=\Delta[1]\times \dotsb \times \Delta[1]$.  The $n$-th
suspension inherits from $\Delta[1]^{n}$ an action of the symmetric group
$\Sigma_{n}$. 

\begin{defn}\label{defenrich}
Let $\aA\subset \aC$ be an enhanced simplicially enriched Waldhausen
category.  Define $\aA^{S}$ be the spectral category with objects the
objects of $\aA$ and mapping symmetric spectra
\[
\aA^{S}(x,y)(n)=\aC(x,\Sigma^{n}y).
\]
We will refer to this as the \term{non-connective spectral enrichment} of $\aA$
or the non-connective spectral category associated to $\aA$.
\end{defn}

In the previous definition, we obtain the composition on $\aA^{S}$,
\[
\aA^{S}(y,z) \sma \aA^{S}(x,y) \to \aA^{S}(x,z)
\]
from the $\Sigma_{n}\times \Sigma_{m}$-equivariant maps
\[
\aC(y,\Sigma^{m}z) \sma \aC(x, \Sigma^{n}y) 
\to
\aC(\Sigma^{n}y,\Sigma^{m+n}z) \sma \aC(x, \Sigma^{n} y) 
\to
\aC(x,\Sigma^{m+n}z).
\]
Note that for a enhanced simplicially enriched Waldhausen category
$\aA\subset \aC$, the suspension of an object of $\aA$ is an object of
$\aC$ but need not be an
object in $\aA$. As a
consequence, the non-connective enrichment $\aA^{S}$ depends strongly
on the ambient simplicially tensored Waldhausen category $\aC$.
Recall that an enhanced exact functor has as part of its structure a
tensored exact functor on the ambient simplicially tensored Waldhausen
categories; the following functoriality is immediate from the
construction.

\begin{prop}\label{propfunctS}
An enhanced exact functor $\phi \colon \aA\to \aB$ between enhanced
simplicially enriched Waldhausen categories induces a spectral functor
$\phi^{S}\colon \aA^{S}\to \aB^{S}$.  If $\phi$ is a 
DK-equivalence and a DK-embedding on the ambient simplicially tensored
categories, then $\phi^{S}$ is a DK-equivalence.
\end{prop}

Using Proposition~\ref{propcswprop}.(i) and the Kan condition, we see
that the action of any even permutation on 
$\aA^{S}(x,y)(n)=\aC(x,\Sigma^{n}y)$ is homotopic to the identity.
Then \cite[3.2]{SchwedeSymmHom} gives us the following proposition.

\begin{prop}\label{propsemistableS}
The mapping symmetric spectra in $\aA^{S}$ are semistable.
\end{prop}

\begin{example}\label{exexactnoncon}
Let $\abA$ be an abelian category with enough projectives (e.g., the
opposite category of an abelian category with enough injectives), and
let $\abE\subset \abA$ be an exact category (with exact sequences the
sequences in $\abE$ that are exact in $\abA$).  Let $\aC$ be the simplicially
tensored Waldhausen category of levelwise projectives in the category
of simplicial objects of $\abA$, as in
Example~\ref{exofinterest}.(iv).  Let $\aA\subset \aC$ be the full
subcategory of $\aC$ consisting of those objects $x$ such that
$\pi_{0}x$ is in $\abE$ and $\pi_{n}x=0$ for $n>0$.
Then $\aA\subset \aC$ is an enhanced simplicially enriched Waldhausen
category and $\pi_{0}$ gives a enriched exact functor $\aA\to
\abE$.  This functor induces a DK-equivalence of the connective
spectral enrichments $\aA^{\Gamma}\to \abE^{\Gamma}$.  On the other
hand $\aA$ has a non-connective spectral enrichment $\aA^{S}$, where
$\pi_{n}\aA^{S}(x,y)$ is $0$ for $n>0$ and
$\Ext^{-n}(\pi_{0}x,\pi_{0}y)$ for $n\leq0$.
\end{example}

\begin{example}\label{exexacttwo}
As an example to demonstrate the significance of the ambient
simplicially tensored Waldhausen category, 
let $\aC$ be the Waldhausen category of countable cell EKMM $S$-modules 
and let $\aC'$ the Waldhausen category of countable cell EKMM
$H\bZ$-modules (for some countable cell $S$-algebra model of $H\bZ$).
Let $\aA$
and $\aA'$ be the Waldhausen subcategories of Eilenberg-Mac Lane
spectra with homotopy groups concentrated in degree zero in
$\aC$ and $\aC'$ respectively.  The forgetful functor $\aC'\to
\aC$ is exact and sends $\aA'$ into $\aA$, inducing a DK-equivalence
and hence a DK-equivalence $\aA^{\prime\Gamma}\to \aA^{\Gamma}$  but
not a DK-equivalence $\aA^{\prime S}\to \aA^{S}$.
\end{example}

The next two propositions explore the relationship between
$\aA^{\Gamma}$ and $\aA^{S}$.

\begin{prop}
There is a canonical spectral functor $\aA^{\Gamma}\to \aA^{S}$, natural in
enhanced simplicially enriched Waldhausen categories $\aA\subset \aC$.
\end{prop}

\begin{proof}
The maps of simplicial sets 
\[
\aA(x,y)\otimes \Delta[1]^{n}\to
\aA(x,y\otimes \Delta[1]^{n})\to \aA(x,\Sigma^{n} y)
\]
induce equivariant maps of based simplicial sets $\Sigma^{n} \aA(x,y)\to
\aA(x,\Sigma^{n} y)$, which assemble into the spectral functor
$\aA^{\Gamma}\to \aA^{S}$.
\end{proof}

In Example~\ref{exexacttwo}, and in fact in the examples of
Example~\ref{exofinterest}, the canonical map $\aA^{\Gamma}\to
\aA^{S}$ of the previous proposition makes $\aA^{\Gamma}$ a connective
cover of $\aA^{S}$, i.e., induces an isomorphism on the non-negative
homotopy groups.  The following proposition gives a sufficient general
condition for this to hold.

\begin{prop}\label{prop:conncover}
Let $\aA\subset \aC$ be an enhanced simplicially enriched Waldhausen
category, and assume that for every $a,b\in \aA$ the suspension map
$\aC(a,b)\to \aC(\Sigma a,\Sigma b)$ is a weak equivalence.  Then
$\aA^{\Gamma}(a,b)\to \aA^{S}(a,b)$ is a \textin{connective cover}.
\end{prop}

\begin{proof}
Fix $a,b$ and consider the functor $F(-)=\aC(a,(b\otimes
-)/(b\otimes *))$ as a functor from based finite simplicial sets to  based
simplicial sets; we then get $\aA^{\Gamma}(a,b)$ by viewing $F$ as
$\Gamma$-space and $\aA^{S}(a,b)$ by viewing $\{F(S^{n})\}$ as a
symmetric spectrum.  By the hypothesis of the proposition, the
canonical map $F(-)\to \Omega F(\Sigma -)$ is a weak equivalence.  The
argument of \cite[17.9]{MMSS} shows that $F$ is ``linear'' meaning
that it takes homotopy pushouts to homotopy pullbacks, and in
particular, as a $\Gamma$-space $F$ is very special
\cite[18.6]{MMSS}.  The homotopy groups of $\aA^{\Gamma}(a,b)$ are
then the homotopy groups of $F(S^{0})=\aA(a,b)$.  Likewise,
$\{F(S^{n})\}$ is an $\Omega$-spectrum, so its non-negative homotopy
are also the homotopy groups of $F(S^{0})$.
\end{proof}

In the absence of the stability hypothesis of the previous
proposition, $\aA^{S}$ tends to better capture the stable homotopy
theory of $\aA\subset \aC$, as indicated for example in the following
proposition.

\begin{prop}\label{propfibercofiber}
Let $\aA\subset \aC$ be an enhanced simplicially enriched Waldhausen
category.
\begin{enumerate}
\item For any $x,y$ in $\aA$, the map $\aA^{S}(x,y)\to \aA^{S}(\Sigma
x,\Sigma y)$ is a weak equivalence.
\item For a cofibration $f\colon a\to b$, $Cf$ the homotopy
cofiber, and any object $z$, the sequences
\[
\xymatrix@R-1.5pc@C-1pc{
\Omega\aA^{S}(a,z)\ar@{{}{}{}}[r]|-{\textstyle \iso}&\aA^{S}(\Sigma a,z) \ar[r] \ar[r] & \aA^{S}(Cf,z) \ar[r] & \aA^{S}(b,z) \ar[r] & \aA^{S}(a,z) \\
&\aA^{S}(z,a) \ar[r] \ar[r] & \aA^{S}(z,b) \ar[r] & \aA^{S}(z,Cf) \ar[r] & \aA^{S}(z,\Sigma a)\ar@{{}{}{}}[r]|-{\textstyle\simeq}&\Sigma \aA^{S}(z,a) \\
}
\]
form a fiber sequence and a cofiber sequence in the stable category, respectively.
\end{enumerate}
\end{prop}

\begin{proof}
Part~(i) and the statement about the first sequence in part~(ii) are
clear.  The statement about the second sequence follows from part~(i)
and the argument in~\cite[\S III.2.1]{LMS} or~\cite[7.4.vi]{MMSS}.
\end{proof}

The proposition indicates that for a simplicially tensored Waldhausen
category $\aC$, the spectral category $\aC^{S}$ is nearly
pretriangulated (Definition~\ref{deftriang}).  In
fact, we have the following easy corollary:

\begin{cor}\label{cor:pretriang}
Let $\aC$ be a simplicially tensored Waldhausen category in which
every object is weakly equivalent to a suspension.  Then the category
$\aC^{S}$ is \textin{pretriangulated}, and in particular, the category
of components 
$\pi_{0}\aC^{S}$ has the structure of a triangulated category with
triangles coming from the \textin{Puppe sequence}s and translation
from the suspension.  
\end{cor}

\begin{rem}
As the preceding results indicate, the construction of the mapping
spectra described above provides a version of stabilization of
the simplicial Waldhausen category $\aC$, when we regard the objects
of $\aC$ as being compact.  In particular, the zeroth
space of (a fibrant replacement of) the mapping spectrum 
$\aC^{S}(x,y)$ is given by 
\[
\colim_n \Omega^n \aC(x, \Sigma^n y) \cong \colim_n \aC(\Sigma^n
x, \Sigma^n y).
\]
It is possible to explicitly compare $\aC^{S}$ to a model of the
formal stabilization in terms of symmetric spectrum objects in $\aC$.
We give an example below, but general theorems of this sort are
encumbered with technical hypotheses, and since we do not need such
results we leave them to the interested reader.
\end{rem}

\begin{example}[Spectral categories and stabilization in Waldhausen's
algebraic $K$-theory of spaces]
Let $G$ be a group-like topological monoid, let
$W$ be a CW-complex on which $G$ acts, and let $R(W,G)$ denote the
category of $G$-spaces which have $W$ as a retract.  When restricting
to objects satisfying some kind of finiteness condition, $R(W,G)$ provides
Waldhausen's motivating example for a Waldhausen category and one of
the models underlying the algebraic $K$-theory of spaces. We can give
$R(W,G)$ the
model structure in which the weak equivalences are the equivariant maps
that induce underlying equivalences of spaces.  The category $R(W,G)$
is in no sense stable (for example, when $G$ and $W$ are trivial,
$R(W,G)$ is the category of based spaces), and the spectral category
$R(W,G)^{S}$ is equivalent to the evident subcategory of free
$\Sigma^{\infty}_+ G$-spectra, as expected. 
\end{example}

\section{The $\Sdot$ and {M}oore nerve constructions}\label{sec:moore1}

As part of the construction of $THH$ and $TC$ of simplicially enriched
Waldhausen categories and the construction of the cyclotomic
trace in the Section~\ref{secdefthh}, we need to extend Waldhausen's
$\Sdot$ construction and the nerve category construction to the
context of simplicially enriched Waldhausen categories.  We begin with
the $\Sdot$ construction, where no difficulties arise.

Let $\Ar[n]$ denote the lexicographically ordered set of ordered pairs
of integers $i,j$ where $0\leq i\leq j\leq n$.
Recall that for a Waldhausen category $\aC_{0}$, $\Sdot[n]\aC_{0}$ is
the full subcategory of the category of functors $A=a_{-,-}\colon
\Ar[n]\to \aC_{0}$ such that:
\begin{enumerate}
\item $a_{i,i}=*$,
\item $a_{i,j}\to a_{i,k}$ is a cofibration, and
\item $a_{i,i}\cup_{a_{i,j}}a_{i,k}\to a_{j,k}$ is an isomorphism
\end{enumerate}
for all $i\leq j\leq k$.  A map in $\Sdot\aC_{0}$ is simply a natural
transformation of 
functors $\Ar[n]\to \aC_{0}$.  This becomes a Waldhausen category with weak
equivalences defined objectwise and cofibrations defined to be the
objectwise cofibrations 
$A\to B$ such that each map $a_{i,k}\cup_{a_{i,j}}b_{i,j}\to b_{i,k}$
is a cofibration.

\begin{defn}
For a simplicially enriched Waldhausen category $\aC$, let
$\Sdot[n]\aC$ be the simplicially enriched category with objects the
same as $\Sdot[n]\aC_{0}$ and with the simplicial set of maps
$\Sdot[n]\aC(A,B)$ the 
simplicial set of natural transformations of functors $\Ar[n]\to \aC$
from $A$ to $B$.\index{Sdot construction@$\Sdot$ construction}
\end{defn}

Condition~(iii) in the definition of $\Sdot$ implies that a map $A\to
B$ is completely determined by the maps $a_{0,j}\to b_{0,j}$.  Since
the maps $a_{0,j}\to a_{0,j+1}$ are cofibrations, we can identify the
simplicial set of maps $\Sdot[n]\aC(A,B)$ as a pullback over fibrations
\begin{equation}\label{eqsn}
\Sdot[n]\aC(A,B)\iso
\aC(a_{0,1},b_{0,1})\times_{\aC(a_{0,1},b_{0,2})}
\times \dotsb
\times_{\aC(a_{0,n-1},b_{0,n})}
\aC(a_{0,n},b_{0,n}).
\end{equation}
That is, the simplicial set of maps computes a homotopy limit.  Using
this formulation of the maps, the following becomes an easy check of
the definitions and standard properties of pullbacks of fibrations of
Kan complexes.

\begin{prop}\label{prop:Sdotinherit}
Let $\aC$ be a simplicially enriched Waldhausen category.  Then:
\begin{enumerate}
\item  $\Sdot[n]\aC$
is a simplicially enriched Waldhausen category.
\item If $\aC$ is simplicially
tensored or enhanced, then so is $\Sdot[n]\aC$.
\item The face and degeneracy maps $\Sdot[m]\aC\to \Sdot[n]\aC$ are
enriched exact.
\item If $\aC$ is simplicially
tensored or enhanced then the face and degeneracy maps $\Sdot[m]\aC\to
\Sdot[n]\aC$ are tensored exact or enhanced exact.
\end{enumerate}
Moreover, $\Sdot[n]$ preserves enriched exact, tensored exact, and
enhanced exact functors.
\end{prop}

Applying the spectral category constructions of the previous section,
we get a simplicial spectral category $\Sdot\aC^{\Gamma}$, natural in
enriched exact functors of $\aC$.  When $\aC$ is simplicially
tensored or enhanced, we get a simplicial spectral category
$\Sdot\aC^{S}$, natural in tensored exact or enhanced exact functors
of $\aC$.  The formula~\eqref{eqsn} for the mapping spaces then
implies the following results for spectral categories.

\begin{prop}\label{propfunct}
Let $\phi \colon \aC\to \aD$ be an enriched exact functor between
simplicially enriched Waldhausen categories that are DK-compatible.
If $\phi$ is a DK-embedding, then
$\Sdot[n]\phi^{\Gamma}\colon \Sdot[n]\aC^{\Gamma}\to
\Sdot[n]\aD^{\Gamma}$ is a DK-embedding.
\end{prop}

\begin{prop}\label{propfunctwo}
Let $\phi \colon (\aA \subset \aC) \to (\aB \subset \aD)$ be an
enhanced exact functor between enhanced simplicially enriched
Waldhausen categories.  If $\phi \colon
\aC \to \aD$ is a DK-embedding, then 
\[
\Sdot[n]\phi^{S}\colon \Sdot[n]\aA^{S}\to \Sdot[n]\aB^{S}
\]
is a DK-embedding.
\end{prop}

In Proposition~\ref{propfunct}, we do not necessarily get a
DK-equivalence $\Sdot[n]\aC^{\Gamma}\to \Sdot[n]\aD^{\Gamma}$ from a
DK-equivalence $\aC\to \aD$.  Applying the results of
\cite{BlumbergMandellUW}, we can do slightly better in
Proposition~\ref{propfunctwo}.

\begin{prop}\label{propfuncthree}
Under the hypotheses of Proposition~\ref{propfunctwo}, if $\phi \colon
\aA\to \aB$ and $\phi \colon \aC \to \aD$ are DK-equivalences, then
\[
\Sdot[n]\phi^{\Gamma}\colon \Sdot[n]\aA^{\Gamma}\to \Sdot[n]\aB^{\Gamma}
\qquad \text{and}\qquad 
\Sdot[n]\phi^{S}\colon \Sdot[n]\aA^{S}\to \Sdot[n]\aB^{S}
\]
are also DK-equivalences. 
\end{prop}

\begin{proof}
It suffices to show that for any sequence of cofibrations $b_{1}\to
\dotsb \to b_{n}$ in $\aB$, there exists a sequence of cofibrations
$a_{1}\to \dotsb \to a_{n}$ in $\aA$ and a commutative diagram
\[
\xymatrix@C-1pc{%
\phi (a_{1})\ar[r]\ar[d]_{\sim}&\phi (a_{2})\ar[r]\ar[d]_{\sim}
&\dotsb \ar[r]&\phi(a_{n-1})\ar[r]\ar[d]_{\sim}
&\phi(a_{n})\ar[d]_{\sim}\\
b_{1}\ar[r]&b_{2}\ar[r]&\dotsb \ar[r]&b_{n-1}\ar[r]&b_{n}
}
\]
with the vertical maps weak equivalences.  We argue by induction on
$n$, the base case of $n=1$ following from the fact that $\phi$ is a
DK-equivalence and all weak equivalences have homotopy inverses.
Having constructed the diagram
\[
\xymatrix@C-1pc{%
\phi (a_{1})\ar[r]\ar[d]_{\sim}
&\phi (a_{2})\ar[r]\ar[d]_{\sim}
&\dotsb \ar[r]&\phi(a_{n-1})\ar[d]_{\sim}\\
b_{1}\ar[r]&b_{2}\ar[r]&\dotsb \ar[r]&b_{n-1}\ar[r]&b_{n}
}
\]
by induction, we know from \cite[1.4]{BlumbergMandellUW} that the
homotopy category of objects in $\aC$ under $a_{n-1}$ is equivalent to
the homotopy category of objects in $\aD$ under $\phi(a_{n-1})$.  We
then get an object $a'$ a map $a_{n-1}\to a'$ in $\aC$ and a zigzag of weak
equivalences under $\phi(a_{n-1})$ in $\aD$ from $\phi(a')$ to
$b_{n}$.  Since $b_{n}$ is in $\aB$, by the embedding hypotheses, we
see that $a'$ is in $\aA$.
Using an appropriate generalized interval $J$, we let $a_{n}=(a_{n-1}\otimes
J)\cup_{a_{n-1}}a'$.  The inclusion of $a_{n-1}$ in $a_{n}$ is a
cofibration in $\aC$, and we get a weak equivalence under
$\phi(a_{n-1})$ from $\phi(a_{n})$ to $b_{n}$.  To complete the
argument we need to see that $a_{n-1}\to a_{n}$ is a cofibration in
$\aA$, i.e., that its cofiber is in $\aA$. This follows since
$\phi(a_{n}/a_{n-1})$ is weakly equivalent to $b_{n}/b_{n-1}$, which
is in $\aB$ since by hypothesis $b_{n-1}\to b_{n}$ is a cofibration in
$\aB$. 
\end{proof}

Waldhausen constructed the $K$-theory spectrum $K\aC_{0}$ as
$w\subdot\Sdot^{(n)}\aC_{0}$, where $\Sdot^{(n)}$ is the iterated
$\Sdot$-construction and $w\subdot$ is the nerve of the subcategory of
weak equivalences.  The previous proposition extends the iterated
$\Sdot$ construction to simplicially enriched categories.  We could
likewise consider the simplicially enriched categories $w_{n}\aC$ with
objects the sequences of weak equivalences
\[
a_{0}\overto{\sim}\dotsb \overto{\sim}a_{n}
\]
and simplicial sets of maps the natural transformations.  Then  for
objects $A$ and $B$, the simplicial set of $w_{n}\aC(A,B)$
becomes
\[
\aC(a_{0},b_{0})\times_{\aC(a_{0},b_{1})}
\dotsb
\times_{\aC(a_{n-1},b_{n})}
\aC(a_{n},b_{n}).
\]
While this works formally, it does not work well homotopically because
the pullbacks are not over fibrations and so the mapping spaces are
not homotopy limits.  

We can sometimes resolve this problem by working with the
simplicially enriched categories $\bar w_n \aC$, where the objects
are the sequences of maps which are weak equivalences and
cofibrations; we use this construction in Section~\ref{sec:loc}.
However, this is often inconvenient and does not always produce the
correct result, and so instead we 
describe a general technique for fixing the problem by putting choices of
homotopies in the mapping spaces.  As a first case, consider the
following construction.

\begin{cons}\label{consvone}
Let $\aC$ be a simplicially enriched category and let $v\aC_{0}$ be a
subcategory of $\aC_{0}$.  We construct a topologically enriched
category $v_{1}^{M}\aC$ as follows. An object consists of a map
$\alpha_{0}\colon a_{0}\to a_{1}$ in $v\aC_{0}$.  The space of maps
$v_{1}\aC(A,B)$ 
consists of elements $f_{0},f_{1}$ of the geometric realizations
$|\aC(a_{0},b_{0})|$, $|\aC(a_{1},b_{1})|$ (respectively), a
non-negative real number $r$, and a homotopy $f_{0,1}$ of length $r$
in $|\aC(a_{0},b_{1})|$ from $\beta_{0}\circ f_{0}$ to $f_{1}\circ
\alpha$; we topologize this as a subspace of
\[
|\aC(a_{0},b_{0})|\times |\aC(a_{1},b_{1})|\times \bR \times |\aC(a_{0},b_{1})|^{I}.
\]
Composition is induced by composition of maps and homotopies.
\end{cons}

In the notation ``$M$'' stands for Moore, as this employs the
\textin{Moore trick} for making homotopy composition associative.  In this
construction, the mapping space $v_{1}^{M}\aC(A,B)$ is homotopy
equivalent to the homotopy pullback
\[
|\aC(a_{0},b_{0})|\times_{|\aC(a_{0},b_{1})|} 
|\aC(a_{0},b_{1})|^{I} \times_{|\aC(a_{0},b_{1})|} 
|\aC(a_{1},b_{1})|.
\]
The Moore trick generalizes from paths to maps out of higher simplices
\cite[\S2]{McClureSmith}.  We understand the $n$-simplex of length
$r>0$ to be the subspace $\Delta[n]_{r}$ of points
$(t_{0},t_{1},\dotsc,t_{n})$ of 
$\bR^{n+1}$ with $t_{i}\geq 0$ and $\sum t_{i}=r$.  Then given
$r,s>0$, the maps $\sigma^{i,n-i}_{r,s}\colon \Delta[i]_{r}\times
\Delta[n-i]_{s}\to \Delta[n]_{r+s}$ defined by
\[
\sigma^{i,n-i}_{r,s}\colon (t_{0},\dotsc,t_{i}),(u_{0},\dotsc,u_{n-i})\longmapsto
(t_{0},\dotsc,t_{i}+u_{0},u_{1},\dotsc,u_{n-i})
\]
decompose $\Delta[n]_{r+s}$ as a union of prisms
\[
\psi^{n}_{r,s}\colon \Delta[n]_{r+s}\iso \bigcup_{i=0}^{n}
\Delta[i]_{r}\times \Delta[n-i]_{s}.
\]
(See Proof of Theorem~2.4 in \cite[p.~162]{McClureSmith}.)  This
decomposition clearly commutes with the simplicial face and degeneracy
operations, and it is associative in that the following diagram
commutes.
\[
\xymatrix@C+1em{%
\Delta[i]_{q}\times\Delta[j]_{r}\times \Delta[k]_{s}
\ar[r]^{\sigma^{i,j}_{q,r}\times \id}
\ar[d]_{\id\times \sigma^{j,k}_{r,s}}
&\Delta[i+j]_{q+r}\times \Delta[k]_{s}
\ar[d]^{\sigma^{i+j,k}_{q+r,s}}\\
\Delta[i]_{q}\times \Delta[j+k]_{r+s}
\ar[r]_{\sigma^{i,j+k}_{q,r+s}}
&\Delta[i+j+k]_{q+r+s}
}
\]

\begin{cons}[Moore Nerve]\label{consmoore}
For $\aC$ a simplicially enriched category and $v\aC_{0}$ a
subcategory of $\aC_{0}$, define the topologically enriched category
$v^{M}_{n}\aC$ as follows.  The objects consist of the sequences of
$n$ composable maps in $v\aC_{0}$
\[
a_{0}\overto{v}\dotsb \overto{v}a_{n}.
\]
For convenience in what follows, we denote the structure map $a_{i}\to
a_{j}$ as $\alpha_{i,j}$, for $i\leq j$ (and $\beta_{i,j}$,
$\gamma_{i,j}$ similarly for objects $B$,$C$).
An element of the space of maps from $A$ to $B$ consists of the
following data: 
\begin{enumerate}
\item An non-negative real number $r$
\item For each $0\leq m\leq n$ and each $0\leq i_{0}<\dotsb <i_{m}<n$
a map 
\[
f_{i_{0},\dotsc,i_{m}}\colon \Delta[m]_{r}\to |\aC(a_{i_{0}},b_{i_{m}})|
\]
for $r>0$, or an element of $|\aC(a_{i_{0}},b_{i_{m}})|$ for $r=0$.
\end{enumerate}
such that for any subset $i_{j_{0}},\dotsc,i_{j_{\ell}}$ of
$i_{1},\dotsc,i_{m}$, the map
\[
\beta_{i_{j_{\ell}},i_{m}} \circ 
f_{i_{j_{0}},\dotsc,i_{j_{\ell}}} \circ 
\alpha_{i_{0},i_{j_{0}}} \colon \Delta[\ell]_{r}\to
|\aC(a_{i_{0}},b_{i_{m}})|
\]
is the restriction to the face of $f_{i_{0},\dotsc,j_{m}}$ spanned by
$i_{j_{0}},\dotsc,i_{j_{\ell}}$.  We topologize this as a subset of
the evident product.  Composition is induced by the prismatic
decomposition above: for $F\colon A\to B$ of length $r>0$ and
$G\colon B\to C$ of length $s>0$, the composition $H\colon A\to C$ of
length $r+s$ is defined by taking $h_{i_{0},\dotsc,i_{m}}$ to be the
map
\[
(g_{i_{j},\dotsc,i_{m}}(u_{0},\dotsc,u_{m-j})\circ \alpha_{i_{0},i_{j}})
\circ 
(\gamma_{i_{j},i_{m}}\circ f_{i_{0},\dotsc,i_{j}}(t_{0},\dots,t_{j}))
\]
on the $\Delta[j]_{r}\times\Delta[m-j]_{s}$ prism in the
$\psi^{m}_{r,s}$ decomposition of $\Delta[m]_{r+s}$.  For $r=0$ or
$s=0$, composition is induced by composition in $\aC$.
\end{cons}

A straightforward check of the formulas verifies that this defines a
topological category. Moreover, $v^{M}\subdot\aC$ assembles into a
simplicial topological category with the following naturality
property.  (It applies in particular to the important special case
$\aC=\aD$ with $v\aC_{0}\subsetneqq v\aD_{0}$.)

\begin{prop}\label{propmoorefunct}
Given simplicially enriched categories $\aC$ and $\aD$, a
simplicially enriched functor $\phi \colon \aC\to\aD$ that takes
$v\aC_{0}$ into $v\aD_{0}$ induces a topologically enriched simplicial functor
$v^{M}\subdot\aC\to v^{M}\subdot\aD$.
\end{prop}

For objects $A$ and $B$, $v^{M}_{n}(A,B)$ is homotopy equivalent to
the homotopy end of $\aC(a_{i},b_{i})$ for $n>0$, while
$v^{M}_{0}\aC(a,b)=|\aC(a,b)|\times [0,\infty)$.  In particular $\aC$
includes in $v^{M}_{0}\aC$ (after geometric realization) as the
subcategory of maps of length zero.  More generally, the nerve
categories $v_{n}\aC$ include (after geometric realization) as the
subcategories of the Moore nerve categories $v^{M}_{n}\aC$ of the maps
of length zero.  Restricting to simplicially enriched Waldhausen
categories, we get the following proposition.

\begin{prop}\label{propmoorecof}
Let $\aC$ be a simplicially enriched Waldhausen category and
$v\aC_{0}$ a subcategory of $\aC_{0}$.
\begin{enumerate}
\item If $v\aC_{0}\subset w\aC_{0}$, then the inclusion of $\aC$ in
$v^{M}_{n}\aC$ is a DK-equivalence.
\item If $v\aC_{0}\subset \co\aC_{0}$, then the inclusion of
$v_{n}\aC$ in $v^{M}_{n}\aC$ is a DK-equivalence.
\end{enumerate}
\end{prop}

Finally, we use the following notation.

\begin{defn}\label{defn:vnerve}\index{vMC@$v^{M}\subdot\aC$}
Let $\aC$ be a simplicial enriched Waldhausen category, and let
$v\aC_{0}$ be a subcategory of $\aC_{0}$.  Define
$v^{M}\subdot\aC^{\Gamma}$ to be the simplicial spectral category
obtained from the simplicial $\Gamma$-category with
\[
v^{M}_{n}\aC^{\Gamma}_{q}(X,Y)=v^{M}_{n}\aC(X,\myop\bigvee_{q}Y).
\]
For $\aA\subset \aC$ an enhanced simplicially enriched Waldhausen
category, define 
$v^{M}\subdot\aA^{S}$ to be the simplicial spectral category
with
\[
v^{M}_{n}\aA^{S}(X,Y)(q)=v^{M}_{n}\aC(X,\Sigma^{q}Y).
\]
\end{defn}

In the formula, $\myop\bigvee$ denotes the entry-wise coproduct;
although this is not the coproduct in $v^{M}_{n}\aC$, we can identify 
\[
v^{M}_{n}\aC(\myop\bigvee_{q} Y,Z)\subset \prod_{q} v^{M}_{n}\aC(Y,Z)
\]
as the subspace of $q$-tuples of maps, all having the same length.
We then obtain $\Gamma$-category composition as in
Section~\ref{secspec}.  Likewise, in the enhanced context, although
$\Sigma^{n} Y$ is not a based tensor in $v^{M}_{n}\aC$, we
nevertheless have a continuous functor
\[
v^{M}_{n}\aC(Y,Z)\to v^{M}_{n}(\Sigma^{n}Y,\Sigma^{n}Z)
\]
and we obtain the spectral category composition as in
Section~\ref{secspec}.  

\section{The {M}oore $\Spdot$ construction}\label{sec:spMdot}

Although the $\Sdot$ construction translates naturally to the enriched
context, it is often useful to be able to weaken the cocartesian
condition in the construction and instead work with an equivalent
construction defined in terms of \term{homotopy cocartesian} squares
called the $\Spdot$ construction \cite[\S2]{BlumbergMandell}.  This
flexibility plays a key role in the proof of the d\'evissage theorem for
the localization theorem for $THH(ku)$ in Section~\ref{sec:dev}.  Such
a definition also provides models of $K$-theory and $THH$ which are
functorial in functors ``exact up to homotopy'' as explained in
Section~\ref{sec:weaklyexact}.
In this section we introduce an appropriately enriched
version of the $\Spdot$ construction, using 
the Moore ideas from the previous section to construct
the homotopically correct enrichment.  

We begin by reviewing the $\Spdot$ construction.  For this, recall
from \cite[\S2]{BlumbergMandellUW} that a \term{weak cofibration} is a
map that is weakly equivalent (by a zigzag) to a cofibration in the
category $\Ar\aC_{0}$ of arrows in $\aC_{0}$, and a \term{homotopy
cocartesian square} is a square diagram that is weakly equivalent (by
a zigzag) to a pushout square where one of the parallel sets of arrows
consists of cofibrations.

\begin{cons}
Let $\aC_{0}$ be a Waldhausen category.  Define 
$\Spdot[n]\aC_{0}$ to be the full subcategory of functors $A\colon
\Ar[n]\to \aC_{0}$ such that: 
\begin{itemize}
\item The initial map $*\to a_{i,i}$ is a weak equivalence for all $i$, 
\item The map $a_{i,j}\to a_{i,k}$ is a weak cofibration for all $i
\leq j \leq k$, and 
\item The diagram
\[  \xymatrix@-1pc{%
a_{i,j}\ar[r]\ar[d]&a_{i,k}\ar[d]\\a_{j,j}\ar[r]&a_{j,k}
} \]
is a homotopy cocartesian square for all $i \leq j \leq k$.
\end{itemize}
\end{cons}

We define a map $A\to B$ to be a weak equivalence when each
$a_{i,j}\to b_{i,j}$ is a weak equivalence.  Clearly $\Spdot$
assembles into a simplicial category with the usual face and
degeneracy functors.

In order to use $\Spdot \aC_0$ to construct $K$-theory, we need a 
mild hypothesis on $\aC_0$.
We say that a Waldhausen category $\aC_{0}$ admits \term{factorization}
when any map $f\colon a\to b$ in $\aC_{0}$ factors as a cofibration
followed by a weak equivalence 
\[
\xymatrix{%
a\ar@{ >->}[r]\ar@{..>}@/_1em/[rr]_{f}&Tf\ar[r]^{\sim}&b,
}
\]
We say that $\aC_{0}$ admits functorial factorization if this
factorization may be chosen functorially in $f$ in the category
$\Ar\aC_{0}$ of arrows in $\aC_{0}$.  More generally, we say that
$\aC_{0}$ admits factorization of weak cofibrations (FWC) or
functorial factorization of weak cofibrations (FFWC) when the weak
cofibrations can be factored as above.  Enhanced simplicially enriched
Waldhausen categories always admit FFWC using the standard mapping
cylinder construction.

\begin{prop}
If $\aA$ is an enhanced simplicially enriched Waldhausen category,
then $\aA$ admits FFWC.
\end{prop}

The significance of the hypothesis of FFWC is the following comparison
result \cite[2.9]{BlumbergMandell}.

\begin{prop}\label{propbm}
Let $\aC_{0}$ be a Waldhausen category admitting FFWC.
Then for each $n$, the inclusion $w\Sdot[n]\aC_{0}\to
w\Spdot[n]\aC_{0}$ induces a weak equivalence on nerves.
\end{prop}

The previous proposition implies that $w\subdot\Spdot\aC_{0}$ models
the $K$-theory space of $\aC_{0}$.  Using an iterated $\Spdot$
construction $\Spndot\aC_{0}$ as a full subcategory of
functors $\Ar[\ssdot]\times \dotsb \times \Ar[\ssdot]$ to $\aC_{0}$ (see
\cite[A.5.4]{BlumbergMandellUW}) gives a model
$w\subdot\Spndot\aC_{0}$ for the  $K$-theory spectrum.

For a simplicially enriched Waldhausen category $\aC$, we need a
version of $\Spdot\aC$ (or more generally $w\subdot\Spndot\aC$) with
the correct mapping spaces.  As in the construction of the Moore nerve
in~\ref{consmoore}, we do this using the Moore trick, this time with
the full generality of the McClure-Smith construction of the Moore Tot
\cite[\S2]{McClureSmith} of a cosimplicial object.

\begin{cons}\label{cons:moore}
Let $\aC$ be a category enriched in simplicial sets, let $D$ be a small
category, and let $D\aC_{0}$ be the category of $D$-diagrams in
$\aC_{0}$.  For $A=(a_{d})$ and $B=(b_{d})$ in $D\aC_{0}$, let
$D^{M}\aC(A,B)$ be the McClure-Smith Moore Tot (denoted $\Tot'$ in
\cite[\S2]{McClureSmith}) of the cosimplicial object
\[
D^{M}\aC(A,B)^{q}=\prod_{d_{0}\to \dotsb \to d_{q}}
|\aC(a_{d_{q}},b_{d_{0}})|
\]
(the cosimplicial object for the homotopy end of $|\aC(A,B)|$).
We let $D^{M}\aC$ be the topologically enriched category with objects
the objects of $D\aC_{0}$, maps the spaces $D^{M}\aC(A,B)$ above, and
composition induced by the ``cup-pairing'' \cite[2.1]{McClureSmith}
\[
\prod_{d_{p}\to \dotsb \to d_{0}}
|\aC(b_{d_{p}},c_{d_{0}})|
\times 
\prod_{d'_{q}\to \dotsb \to d'_{0}}
|\aC(a_{d'_{q}},b_{d'_{0}})|
\to 
\prod_{d_{p+q}\to \dotsb \to d_{0}}
|\aC(a_{d_{p+q}},c_{d_{0}})|.
\]
Here the map is induced on the $d_{p+q}\to \dotsb
\to d_{0}$ coordinate of the target by composition 
\[
\aC(b_{d_{p}},c_{d_{0}})\times \aC(a_{d_{p+q}},b_{d_{p}})
\to \aC(a_{d_{p+q}},c_{d_{0}})
\]
of the maps on the $d_{p}\to\dotsb \to d_{0}$ and $d_{p+q}\to \dotsb
\to d_{p}$ (i.e., $d'_{i}=d_{p+i}$) coordinates of the source.
\end{cons}

As in the previous section, we obtain a connective spectral enrichment
using the objectwise coproduct and (when defined) a non-connective
spectral enrichment using the objectwise suspension.

We use analogous notation for the enriched categories associated to
full subcategories of diagram categories, obtaining for example
$\SMdot[n]\aC$ and $\SpMdot[n]\aC$ as full subcategories of the
functors $\Ar[n] \to \aC$.  Because the Moore Tot always has the
homotopy type of the homotopy end (containing it as a deformation
retract), we obtain the following result as an immediate consequence. 

\begin{prop}
For a simplicially enriched Waldhausen category $\aC$, the inclusion
of the topologically enriched category $|\Sdot[n]\aC|$ in
$\SMdot[n]\aC$ as the length zero part is a 
DK-equivalence. 
\end{prop}

Considering more complicated diagrams, this also applies to
$w_{p}\Sndot[q_{1},\dotsc,q_{n}]\aC$.  Thinking of these
categories as subcategories of $w_{p}\Spndot[q_{1},\dotsc,q_{n}]\aC$,
the more restricted homotopies in
$w_{p}^{M}\Sndot[q_{1},\dotsc,q_{n}]\aC$ make its mapping spaces
subspaces of $(w_{p}\Spndot[q_{1},\dotsc,q_{n}])^{M}\aC$, and we
get the following result.

\begin{prop}\label{prop:zeroinclude}
For a simplicially enriched Waldhausen category $\aC$, the inclusion
of $w_{p}^{M}\Sndot[q_{1},\dotsc,q_{n}]\aC$ in
$(w_{p}\Spndot[q_{1},\dotsc,q_{n}])^{M}\aC$ is a
DK-embedding.  If $\aC_{0}$ admits FFWC,
then it is a DK-equivalence.
\end{prop}

We write $(w_{p}\Spndot[q_{1},\dotsc,q_{n}])^{M}\aC^{\Gamma}$ and
when appropriate $(w_{p}\Spndot[q_{1},\dotsc,q_{n}])^{M}\aC^{S}$ for
the associated spectrally enriched categories.

%%%%%%%%%%%%%%%%%%%%%%%%%%%%%%%%%%%%%%%%
\section{$THH$, $TC$, and the cyclotomic trace}\label{secdefthh}

In this section, we apply the constructions of $THH$ and $TC$ of
spectral categories in the
context of the spectral enrichments associated to a simplicially enriched
Waldhausen category $\aC$.  For the connective enrichments,
we require Waldhausen's $\Sdot$ construction in order to properly handle
extension sequences in the Waldhausen structure for reasons first
observed by McCarthy~\cite[3.3.5]{McCarthyThesis} and Dundas-McCarthy
\cite[2.3.4]{DundasMcCarthy}; for the non-connective enrichment, the
$\Sdot$ construction turns out to be superfluous.

\begin{defn}\label{defwthhG}
\index{THH (construction)@$THH$ (construction)}
\index{TR (construction)@$TR$ (construction)}
\index{TC (construction)@$TC$ (construction)}
\index{WTHH (construction)@$WTHH$ (construction)}
\index{WTR (construction)@$WTR$ (construction)}
\index{WTC (construction)@$WTC$ (construction)}
For a simplicially enriched Waldhausen category $\aC$, we define 
\begin{align*}
WTHH^{\Gamma}\aC&=\Omega |THH(\Sdot \aC^{\Gamma})|\\
WTR^{\Gamma}\aC&=\Omega |TR(\Sdot \aC^{\Gamma})|\\
WTC^{\Gamma}\aC&=\Omega |TC(\Sdot \aC^{\Gamma})|.
\end{align*}
If $\aC$ is a simplicially tensored Waldhausen category and
$\aA\subset \aC$ is an enhanced simplicially enriched Waldhausen category, then
we define
\begin{align*}
WTHH\aA&=\Omega |THH(\Sdot \aA^{S})|\\
WTR\aA&=\Omega |TR(\Sdot \aA^{S})|\\
WTC\aA&=\Omega |TC(\Sdot \aA^{S})|.
\end{align*}
\end{defn}

In other words, we apply $THH$, $TR$, or $TC$ first to get simplicial
(or multisimplicial) cyclotomic spectra or pro-spectra.  Then we take the
geometric realization in the simplicial directions, followed by loops.
We have the following naturality properties.

\begin{prop}\label{propnat}
An enriched exact functor induces maps on $WTHH^{\Gamma}$,
$WTR^{\Gamma}$, and $WTC^{\Gamma}$.  A tensored exact or enhanced exact
functor induces maps on $WTHH$, $WTR$, and $WTC$.   Naturally weakly
equivalent functors induce the same map in the stable category.
\end{prop}

\begin{proof}
The only part not immediate from the construction is the last
statement.  We use Construction~\ref{consvone} for $v\aD_{0}$ the
subcategory of weak equivalences $w\aD_{0}$.  We have a pair of
simplicial spectrally enriched functors $w^{M}_{1}\Sdot\aD^{\Gamma}\to
\Sdot\aD^{\Gamma}$ each split by the inclusion $\Sdot\aD^{\Gamma}\to
w^{M}_{1}\Sdot\aD^{\Gamma}$.  Since the inclusion induces a DK-equivalence
$\Sdot\aD^{\Gamma}\to w^{M}_{1}\Sdot\aD^{\Gamma}$, both maps
$w^{M}_{1}\Sdot\aD^{\Gamma}\to \Sdot\aD^{\Gamma}$ induce the same map
in the stable category on $THH$, $TR$, and $TC$.  Now, given enriched
exact functors $\phi_{0},\phi_{1}\colon \aC\to \aD$  and $h$ a natural weak equivalence between them, we get
a simplicial spectrally enriched functor
$\Sdot\aC^{\Gamma}\to w^{M}_{1}\Sdot\aD^{\Gamma}$ (factoring through
the length zero part $w_{1}\Sdot\aD^{\Gamma}$).  The two composites
\[
\Sdot\aC^{\Gamma}\to w^{M}_{1}\Sdot\aD^{\Gamma}\to \Sdot\aD^{\Gamma}
\]
are the maps induced by $\phi_{0}$ and $\phi_{1}$.  For tensored exact or
enhanced exact functors $\phi_{0}$ and $\phi_{1}$ and a natural weak
equivalence between them, the same argument applies to show that the
maps on $WTHH$, $WTR$, and $WTC$ induced by $\phi_{0}$ and $\phi_{1}$
coincide in the stable category.
\end{proof}

Applying Proposition~\ref{propfunct} and~\ref{propfuncthree} we obtain
the following homotopy invariance properties.

\begin{prop}\label{propapproxone}
Let $\phi \colon \aC\to \aD$ be an enriched exact functor between
simplicially enriched Waldhausen categories that are DK-compatible.  Assume
that $\phi$ is a DK-embedding and that every object of $\Sdot[n]\aD$
is weakly equivalent to an object in the image of $\Sdot[n]\phi$ (for
all $n$). Then $\phi$ induces weak equivalences on $WTHH^{\Gamma}$, $WTR^{\Gamma}$, and $WTC^{\Gamma}$.
\end{prop}

\begin{prop}\label{propapprox}
Let $\phi \colon \aA\to \aB$ be an enhanced exact functor between
enhanced simplicially enriched Waldhausen categories.  If $\phi$ is a
DK-equivalence of the ambient simplicially tensored
categories and a DK-equivalence $\aA\to \aB$, then $\phi$
induces a weak 
equivalence on $WTHH$, $WTR$, and $WTC$.
\end{prop}

Implicitly in the previous propositions we passed from a level weak
equivalence of simplicial spectra $X\subdot\to Y\subdot$ to a weak
equivalence on geometric realization $|X\subdot|\to |Y\subdot|$.
Using the standard geometric realization, we need hypotheses on
$X\subdot$ and $Y\subdot$ for this to work (cf.~\ref{propoldpaper}).
One sufficient hypothesis
is that $X\subdot$ and $Y\subdot$ are 
\noindexterm{spacewise proper}\index{proper!spacewise}: we say
that a simplicial symmetric spectrum of topological spaces
$X\subdot$ is space-wise proper when the 
simplicial space $X\subdot(n)$ is proper for every $n$, i.e., for each $k$, each
degeneracy map $X_{k}(n)\to X_{k+1}(n)$ is a Hurewicz cofibration
(satisfies the homotopy extension property).
The following proposition applies to verify this property for the
constructions in the previous propositions and the many other
constructions in this paper.  Its proof requires the details of the
$THH$ construction in Definition~\ref{defn:123} 
but is then straightforward given the standard
properties of Hurewicz cofibrations.

\begin{prop}\label{prop:realok}
Let $\aC\subdot$ be a simplicial object in the category of spectral
categories (in topological symmetric spectra).  Assume that for all
$k$ the category $\aC_{k}$ is non-degenerately based and that for all
objects $x,y$ of $\aC_{k}$ and for each space of the mapping
each degeneracy map $s^{i}\colon \aC_{k}(x,y)(n)\to
\aC_{k+1}(s^{i}x,s^{i}y)(n)$ is a Hurewicz cofibration.  Then the
simplicial spectrum $THH(\aC\subdot)$ is spacewise proper.
\end{prop}

Waldhausen's approximation property provides a convenient formulation
for the conditions in Propositions~\ref{propapproxone} and~\ref{propapprox} that often holds in
practice.  We say that exact functor $\phi \colon \aC\to\aD$ has
the \term{approximation property}\index{Waldhausen's approximation
property} when:
\begin{enumerate}
\item A map $f\colon a\to b$ is a weak equivalence in $\aC$ only if
the map $\phi(f)$ in $\aD$ is a weak equivalence.
\item For every map $f\colon \phi(a)\to x$ in $\aD$, there exists a
map $g\colon a\to b$ in $\aC$ and a weak equivalence $h\colon \phi(b)\to x$ in
$\aD$ such that $g=h\circ \phi(g)$.
%\[
%\xymatrix{%
%\phi(a)\ar[d]_{\phi(g)}\ar[r]^{f}&x\\\phi(b)\ar[ur]_{h}^{\sim}
%}
%\]
\end{enumerate}
We then have the following $THH$ analogue of Waldhausen's
\textin{Approximation Theorem}.  The proof is that under factorization
hypotheses, the approximation property implies that $\phi$ is a
DK-equivalence; see \cite[1.4--1.5]{BlumbergMandellUW}.

\begin{thm}\label{thmapprox}
Let $\phi \colon (\aA\subset \aC)\to (\aB\subset \aD)$ be an enhanced
exact functor between 
enhanced simplicially enriched Waldhausen categories, and suppose that
$\phi_{0}\colon \aC_{0}\to \aD_{0}$
satisfies the approximation property.  If every object of $\aB$ is
weakly equivalent to the image of an object of $\aA$, then 
$\phi$ induces weak equivalences on $WTHH^{\Gamma}$, $WTR^{\Gamma}$,
$WTC^{\Gamma}$ and on $WTHH$, $WTR$, $WTC$.
\end{thm}

In many situations, the underlying Waldhausen category $\aC_{0}$ of
a simplicially enriched Waldhausen category $\aC$ admits a second
subcategory of weak equivalences $v\aC_{0}$ (not necessarily related
to the simplicial structure, or even satisfying the two out of three
property). When $v\aC_{0}$ contains all the
isomorphisms and satisfies the Gluing Axiom (stated as~(iv) in the
introduction to this chapter), each Waldhausen category 
$\Sdot[n]\aC_{0}$ inherits a subcategory $v\Sdot[n]\aC_{0}$ also
satisfying these properties.  In this context, we have additional
variants of $THH$, $TR$, and $TC$.

\begin{defn}\label{deflocthh}
\index{THH (construction)@$THH$ (construction)}
\index{TR (construction)@$TR$ (construction)}
\index{TC (construction)@$TC$ (construction)}
\index{WTHH (construction)@$WTHH$ (construction)}
\index{WTR (construction)@$WTR$ (construction)}
\index{WTC (construction)@$WTC$ (construction)}
Let $\aC$ be a simplicially enriched Waldhausen category, and let
$v\aC_{0}$ be a subcategory of $\aC_{0}$ containing all the
isomorphisms and satisfying the Gluing Axiom. Then we define the
connective relative $THH$, $TC$, and $TR$ of $(\aC|v)$ as indicated
below (on the left).  When $\aC$ is simplicially tensored and
$\aA\subset \aC$ is an enhanced simplicially enriched Waldhausen
category, we define the non-connective $THH$, $TC$, and $TR$ of
$(\aA|v)$ as indicated below (on the right).
\begin{align*}
WTHH^{\Gamma}(\aC|v)&=\Omega |THH(v^{M}\subdot\Sdot \aC^{\Gamma})|
&WTHH(\aA|v)&=\Omega |THH(v^{M}\subdot\Sdot \aA^{S})|\\
WTR^{\Gamma}(\aC|v)&=\Omega |TR(v^{M}\subdot\Sdot \aC^{\Gamma})|
&WTR(\aA|v)&=\Omega |TR(v^{M}\subdot\Sdot \aA^{S})|\\
WTC^{\Gamma}(\aC|v)&=\Omega |TC(v^{M}\subdot\Sdot \aC^{\Gamma})|
&WTC(\aA|v)&=\Omega |TC(v^{M}\subdot\Sdot \aA^{S})|
\end{align*}
\end{defn}

In the special case when $v\aC_{0}$ is the category of weak
equivalences $w\aC_{0}$, the inclusion of each $\Sdot[n]\aC^{\Gamma}$
into $w^{M}_{q} \Sdot[n] \aC^{\Gamma}$ and (when defined) $\Sdot[n]\aA^{S}$
into $w^{M}_{q} \Sdot[n] \aA^{S}$ is a DK-equivalence.  This implies the
following proposition. 

\begin{prop}
For $\aC$ a simplicially enriched Waldhausen category and $\aA$ an
enhanced simplicially enriched Waldhausen category, the maps
\begin{align*}
WTHH^{\Gamma}\aC&\to WTHH^{\Gamma}(\aC|w)
&WTHH\aA&\to WTHH(\aA|w)\\
WTR^{\Gamma}\aC&\to WTR^{\Gamma}(\aC|w)
&WTR\aA&\to WTR(\aA|w)\\
WTC^{\Gamma}\aC&\to WTC^{\Gamma}(\aC|w)
&WTC\aA&\to WTC(\aA|w)
\end{align*}
are weak equivalences
\end{prop}

Using the model of $THH$ in the previous proposition, we have the
following sharper version of Proposition~\ref{propnat}. 

\begin{prop}\label{propnathomotopy}
Let $\phi_{0},\phi_{1} \colon \aC\to \aD$ be exact functors between
simplicially enriched Waldhausen categories.  A natural weak
equivalence from $\phi_{0}$ to $\phi_{1}$ induces a simplicial
homotopy of the induced functors
\begin{align*}
WTHH^{\Gamma}(\aC|w)&\to WTHH^{\Gamma}(\aD|w)\\
WTR^{\Gamma}(\aC|w)&\to WTR^{\Gamma}(\aD|w)\\
WTC^{\Gamma}(\aC|w)&\to WTC^{\Gamma}(\aD|w)
\end{align*}
and similarly for the non-connective enrichments when appropriate.
\end{prop}

\begin{proof}
The proof is the usual one, using the $n+2$ functors 
\[
w^{M}_{n}\aX\to w^{M}_{n+1}\aY
\]
obtained by inserting the natural transformation in each position
(giving it length 0 in the Moore construction).
\end{proof}

To construct the cyclotomic trace, we need a final variant of these
constructions where we iterate the $\Sdot$ construction. 

\begin{defn}\label{deftildethh}
\index{THH (construction)@$THH$ (construction)}
\index{TR (construction)@$TR$ (construction)}
\index{TC (construction)@$TC$ (construction)}
\index{WTHH (construction)@$WTHH$ (construction)}
\index{WTR (construction)@$WTR$ (construction)}
\index{WTC (construction)@$WTC$ (construction)}
Let 
\[
\wt{WTHH}^{\Gamma}\aC(n)=|THH(w^{M}\subdot S^{(n)}\subdot\aC^{\Gamma})|.
\]
The simplicial maps of spectral categories
\[
\Sigma^{n-m}\subdot w^{M}\subdot S^{(m)}\subdot\aC^{\Gamma}
\to w^{M}\subdot S^{(n)}\subdot\aC^{\Gamma}
\]
induce maps
\[
\Sigma^{n-m} \wt{WTHH}^{\Gamma}\aC(m)
\to \wt{WTHH}^{\Gamma}\aC(n)
\]
(as in \cite[\S1.3]{WaldhausenKT}).  These maps assemble
$\wt{WTHH}^{\Gamma}\aC$ 
into a symmetric spectrum in the category of cyclotomic spectra.
We define $\wt{WTR}$ and
$\wt{WTC}$ to be the $TR$ and $TC$ pro-spectra constructed from
$\wt{WTHH}$. 
\end{defn}

As a consequence of the {Additivity Theorem}~\ref{thmadditivity}, we
prove the following lemma in  Section~\ref{secadd}.

\begin{lem}\label{lemadd}
The map $\Sigma^{n-m} \wt{WTHH}^{\Gamma}\aC(m)
\to \wt{WTHH}^{\Gamma}\aC(n)$ in Definition~\ref{deftildethh}  is a
weak equivalence for all $n\geq m>0$.
\end{lem}

We have analogous constructions and results in the relative case
(using $v^{M}\subdot$ in place of $w^{M}\subdot$)
and non-connective case (using $\aA^{S}$ for $\aA^{\Gamma}$
when $\aA$ is enhanced).
The identity $WTHH^{\Gamma}\aC=\Omega \wt{WTHH}^{\Gamma}\aC(1)$
then immediately implies the following result.

\begin{thm}\label{thmwtadd}
We have natural isomorphisms in the stable category 
\[
WTHH^{\Gamma}\aC \htp 
\wt{WTHH}^{\Gamma}\aC
\qquad 
WTR^{\Gamma}\aC \htp 
\wt{WTR}^{\Gamma}\aC
\qquad
WTR^{\Gamma}\aC \htp 
\wt{WTC}^{\Gamma}\aC
\]
and likewise for the relative and non-connective variants when these
are defined.
\end{thm}

We can now define the cyclotomic trace.

\begin{defn}
For a simplicially enriched Waldhausen category $\aC$, the 
\term{cyclotomic trace}
\[
K(\aC_{0})\to \wt{WTC}^{\Gamma}\aC\to \wt{WTHH}^{\Gamma}\aC
\]
is the map induced by the inclusion of objects 
\[
K\aC_{0}=\Ob(w\subdot\Sdot^{(n)}\aC_{0})=
\Ob(w^{M}\subdot\Sdot^{(n)}\aC^{\Gamma})\to
\wt{WTHH}^{\Gamma}\aC.
\]
For $v\aC_{0}$ a subcategory of
$\aC_{0}$ containing the isomorphisms and satisfying the Gluing Axiom,
the relative cyclotomic trace is the map
\[
K(\aC_{0}|v) \to \wt{WTC}^{\Gamma}(\aC|v)\to \wt{WTHH}^{\Gamma}(\aC|v)
\]
induced by the inclusion of objects
\[
K(\aC_{0}|v)=\Ob(v\subdot\Sdot^{(n)}\aC_{0})=
\Ob(v^{M}\subdot\Sdot^{(n)}\aC^{\Gamma})\to
\wt{WTHH}^{\Gamma}(\aC|v).
\]
\end{defn}

Finally, to compare the definitions of this section with the theories
used in \cite{BlumbergMandellTHHLoc} and Chapter~\ref{tw-thhmv}, we
state the following two theorems. 
The first is a consequence of the 
{Additivity Theorem}~\ref{thmadditivitynoncon} and proved in 
Section~\ref{secadd}. 

\begin{thm}\label{thmnoSnoncon}
Let $\aA$ be an enhanced simplicially enriched Waldhausen category.
The inclusion of $THH(\aA^{S})$ in $WTHH(\aA)$ is a weak equivalence
of cyclotomic spectra. 
\end{thm}

The second is a special case of the Sphere Theorem from
Section~\ref{secspheretheorem}; see Corollaries~\ref{cor:ekmm}
and~\ref{cor:simpring}.  

\begin{thm}\label{thmspecsphere}
Let $R$ be a ring, a simplicial ring, or a connective ring spectrum,
and let $\aA$ be the simplicially tensored Waldhausen category of
finite cell
modules (built out of free or finitely generated projective modules) in
Example~\ref{exofinterest}.(i) or (ii) (as appropriate). 
Then the natural map $WTHH^{\Gamma}(\aA)\to WTHH(\aA)$ is a weak
equivalence of cyclotomic spectra.
\end{thm}

Thus, for a ring, simplicial ring, or connective ring spectrum, we
have weak equivalences of cyclotomic spectra
\[
WTHH^{\Gamma}(\aA)\overto{\sim} 
WTHH(\aA)\overfrom{\sim} 
THH(\aA) \overfrom{\sim} THH(R),
\]
where the last weak equivalence is a special case of
Theorem~\ref{thmthick}.

%%%%%%%%%%%%%%%%%%%%%%%%%%%%%%%%%%%%%%%%
%%%%%%%%%%%%%%%%%%%%%%%%%%%%%%%%%%%%%%%%

% tp-set-source: tw-whole.tex
% tex-command: latex
% ultex-add-defs: list

%
% Copyright (C) 2008-12  Andrew J. Blumberg and Michael A. Mandell
%

\chapter{$K$-theory theorems in $THH$ and $TC$}
\label{tw-ktt}

The purpose of this chapter is to take the standard theorems of
$K$-theory as proved in \cite{WaldhausenKT} and describe versions of
these for $THH$.  Because weak equivalences in $THH$ and cofiber
sequences in $THH$ automatically produce weak equivalences and cofiber
sequences for $TC$, we typically make statements just for $THH$.

Section~\ref{secadd} reviews the \noindexterm{Additivity Theorem}
(q.v.~\cite[1.3.2]{WaldhausenKT}) for 
$THH$, based on work of McCarthy \cite{McCarthyAdditivity}.
Section~\ref{seccofiber} proves a $THH$ version of 
\cite[1.5.5]{WaldhausenKT}, which constructs a cofibration sequence of
$THH$ spectra associated to a map.  Section~\ref{sec:loc} proves a
$THH$ version of Waldhausen's \noindexterm{Fibration
Theorem}~\cite[1.6.4]{WaldhausenKT}.  Section~\ref{secspheretheorem}
proves the \noindexterm{Sphere Theorem}
(cf.~\cite[\S1.7]{WaldhausenKT}), which in certain cases identifies the
$THH$ of simplicially tensored Waldhausen category as the $THH$ of a
subcategory of a subcategory of generators.  Section~\ref{secpfsphere}
proves the Sphere Theorem.

%%%%%%%%%%%%%%%%%%%%%%%%%%%%%%%%%%%%%%%%
\section{The Additivity Theorem}\label{secadd}

In this section, we present the \textin{Additivity Theorem} for the $THH$ of
Waldhausen categories. 
The modern viewpoint, implicit in \cite{WaldhausenKT} but first
written explicitly by Staffeldt \cite{Staffeldt}, holds the Additivity
Theorem as the fundamental property of $K$-theory.  Following this
perspective, we deduce the remaining $K$-theoretic properties of $THH$
from the Additivity Theorem in the next three sections.

To state the Additivity Theorem, we use the following
notation. For $\aC$ a simplicially enriched Waldhausen category, let 
$\aE(\aC)=\Sdot[2]\aC$ be the simplicially enriched Waldhausen
category with objects the cofiber sequences $x \to y \to z$ (in
$\aC_{0})$.  We have enriched exact functors $\alpha,\beta,\gamma$
from $\aE(\aC)$ to $\aC$ defined by
\[
\alpha(x\sto y\sto z)=x,
\qquad
\beta(x\sto y\sto z)=y,
\qquad
\gamma(x\sto y\sto z)=z.
\]

\begin{thm}[Additivity Theorem]\label{thmadditivity}
\index{Additivity Theorem}
For a simplicially enriched Waldhausen category $\aC$, the
enriched exact  functors $\alpha$ and $\gamma$ induce 
a weak equivalence of cyclotomic spectra
\[
WTHH^\Gamma(\aE(\aC)) \to WTHH^\Gamma(\aC) \times  WTHH^\Gamma(\aC)
\htp WTHH^\Gamma(\aC) \vee  WTHH^\Gamma(\aC).
\]
\end{thm}

McCarthy's proof of the Additivity Theorem for
$K$-theory~\cite{McCarthyAdditivity} provides a very general argument
for showing that the map $(\alpha ,\gamma)\colon \Sdot\aE(\aC)\to
\Sdot\aC\times \Sdot\aC$ induces a homotopy equivalence in various
contexts.  The elaboration in~\cite[\S3.4-3.5]{McCarthyThesis} to
prove the Additivity Theorem for cyclic homology of $k$-linear
categories carries over essentially word for word to prove the
Additivity Theorem above, just replacing ``$CN$'' with ``$THH$'' and
``$k$-linear'' with ``spectral''.  (The only property of $THH$ or $CN$
needed is that it takes simplicial homotopy equivalences of simplicial
(spectrally or $k$-linearly) enriched categories to weak equivalences
of spectra or simplicial sets.)

The following result is both a generalization and a corollary of the
Additivity Theorem above.  Recall that a sequence of natural
transformations of exact functors $f\to g\to h$ from $\aC_{0}$ to
$\aD_{0}$ forms a \term{cofiber sequence of exact functors}, when
(taken together) they define an exact functor $\aC$ to $\aE(\aD)$.

\begin{cor}\label{coradditivity}
Let $\aC$ and $\aD$ be simplicially enriched Waldhausen categories,
and let $f\to g\to h$ be a sequence of enriched exact functors $\aC\to
\aD$ that forms a cofiber sequence of exact functors.
Then the maps
\[
WTHH^{\Gamma}(g)
\qquad \text{and}\qquad 
WTHH^{\Gamma}(f) \vee WTHH^{\Gamma}(h).
\]
from $WTHH^{\Gamma}(\aC)$ to $WTHH^{\Gamma}(\aD)$
agree in the homotopy category of cyclotomic spectra.
\end{cor}

\begin{proof}
The functor $\aD\times \aD\to \aE(\aD)$ sending $(a,b)$ to $a \to
a\vee b\to b$ is an enriched exact functor, and the composite map
\[
WTHH^{\Gamma}(\aD)\vee WTHH^{\Gamma}(\aD)
\to WTHH^{\Gamma}(\aD\times \aD)
\to WTHH^{\Gamma}(\aE(\aD))
\]
splits the zigzag of weak equivalences in the Additivity Theorem and
is therefore a weak equivalence of cyclotomic spectra.  It follows
that $\beta$ and $\alpha \vee \gamma$ induce the same map $\aE(\aD)\to
\aD$ in the homotopy category of cyclotomic spectra.  Precomposing
with the map $\aC\to \aE(\aD)$ defined by $f\to g\to h$ proves the
corollary. 
\end{proof}

This corollary provides the key tool for even more general additivity
statements.  For example, the map
$\Sdot[n]\aC\to  \aC\times \Sdot[n-1]\aC$
defined by sending $X=(x_{i,j})$ to $(x_{0,1},d_{0}X)$
induces a weak equivalence on $WTHH^{\Gamma}$.  To see this, consider
the map $\aC\times \Sdot[n-1]\aC\to \Sdot[n]\aC$ sending $(x,Y)$ to
$Z=(z_{i,j})$ with 
\[
z_{i,j}=\begin{cases}
x\vee y_{0,j-1}&i=0\\
y_{i-1,j-1}&i>0.
\end{cases}
\]
The composite map on $\aC\times \Sdot[n-1]\aC$ is the identity, and
the composite map on $\Sdot[n]\aC$ is $f\vee h$ for exact functors $f$
and $h$ that fit in a cofiber sequence of exact functors $f\to g\to h$
with $g$ the identity.  We will use this argument many times in what
follows. 

When $\aA$ is an enhanced simplicially enriched Waldhausen category, so is
$\aE(\aA)$ and the functors $\alpha,\beta,\gamma$ are enhanced exact.
We have precise analogues of the previous results (with the same
proof).  In fact, we have the following stronger version of the
Additivity Theorem for the non-connective enrichment
(cf.~\cite[10.8]{BlumbergMandellTHHLoc}).

\begin{thm}\label{thmadditivitynoncon}
Let $\aA$ be an enhanced simplicially enriched Waldhausen category.
The enhanced exact functors $\alpha$ and $\gamma$ induce a
weak equivalence of cyclotomic spectra
\[
THH(\aE^{S}(\aA)) \to
THH(\aA^{S})\times THH(\aA^{S}) \htp
THH(\aA^{S}) \vee THH(\aA^{S}).
\]
\end{thm}

\begin{proof}
By Theorems~\ref{thmtriangenv} and~\ref{thmthick}, it suffices to
consider the case when $\aA^{S}$ is pretriangulated.  Then the functor
$a\mapsto (a\to a \to *)$ embeds $\aA^{S}$ as a triangulated
subcategory of $\aE^{S}(\aA)$ and the functor $a\mapsto (*\to a\to a)$
induces an equivalence of $\pi_{0}\aA^{S}$ with the triangulated
quotient $\pi_{0}\aE^{S}(\aA)/\pi_{0}\aA^{S}$.  The statement now
follows from Theorem~\ref{thmgenone}.
\end{proof}

As a consequence of the Additivity Theorems~\ref{thmadditivity}
and~\ref{thmadditivitynoncon}, we can now prove Lemma~\ref{lemadd} and
Theorem~\ref{thmnoSnoncon}. 

\begin{proof}[Proof of Lemma~\ref{lemadd} and
Theorem~\ref{thmnoSnoncon}]
We will treat Lemma~\ref{lemadd} in detail;
Theorem~\ref{thmnoSnoncon} follows from the same argument using $\aA$
in place of $\Sdot\aC$ and Theorem~\ref{thmadditivitynoncon} in place
of Theorem~\ref{thmadditivity}.  To prove Lemma~\ref{lemadd}, it
suffices to show that the map $\Sigma WTHH^{\Gamma}(\aC)\to
WTHH^{\Gamma}(\Sdot \aC)$ is a weak equivalence.  This map is induced by a
simplicial map 
\[
\myop\bigvee_{n} WTHH^{\Gamma}(\aC)\to WTHH^{\Gamma}(\Sdot[n]\aC),
\]
which is a levelwise weak equivalence by the Additivity Theorem and
the argument following Corollary~\ref{coradditivity}. 
\end{proof}

%%%%%%%%%%%%%%%%%%%%%%%%%%%%%%%%%%%%%%%%
\section{The Cofiber Theorem}\label{seccofiber}

This section is the first of three %four?
that apply the Additivity Theorem to prove standard
$K$-theory theorems in $THH$ and $TC$.  This section provides a
general cofibration sequence for $THH$ (and $TC$) associated to a map of
Waldhausen categories by identifying the cofiber term as a version of
$THH$ (cf.~\cite[\S1.5]{WaldhausenKT}).  We call this 
theorem the ``Cofiber Theorem''.  

We begin with the construction of the cofiber term.  For $f\colon
\aC\to \aD$ an enriched exact functor, we define a simplicially
enriched Waldhausen category $\Sdot[n] f$ as follows.  An object
consists of an object $Y=(y_{i,j})$ of $\Sdot[n]\aC$ together with an
object $X=(x_{i,j})$ of $\Sdot[n+1]\aD$ such that $d_{0}X=f(Y)$, that
is, $x_{i+1,j+1}=f(y_{i,j})$, with the structure maps for this
subdiagram in $X$ identical with $f(Y)$.  For objects $(X,Y),(X',Y')$,
the simplicial set of maps consists of the simplicial set of natural
transformations.  We make this a Waldhausen category by declaring a
map $(X,Y)\to (X',Y')$ to be a cofibration (resp., weak equivalence)
when the restrictions $X'\to X$ (in $\Sdot[n+1]\aD$) and $Y'\to Y$ (in
$\Sdot[n]\aC$) are both cofibrations (resp., weak equivalences).  This
assembles into a simplicial object in the category of simplicially
enriched Waldhausen categories using the usual face and degeneracy
maps on $\Sdot[n] \aC$ and the last $n+1$ face and degeneracy maps on
$\Sdot[n+1]\aD$.

\begin{defn}\label{defrelthh}
\index{THH (construction)@$THH$ (construction)}
\index{WTHH (construction)@$WTHH$ (construction)}
For $f\colon \aC\to \aD$ an enriched exact functor, define
\[
WTHH^{\Gamma}(f)=|WTHH^{\Gamma}(\Sdot f)|.
\]
\end{defn}

We note that when $f\colon \aA\to \aB$ is an enhanced exact functor
between enhanced simplicially enriched Waldhausen categories, then
$\Sdot[n]f$ is also an enhanced simplicially enriched Waldhausen
category and $\Sdot f$ is a simplicial object in enhanced simplicially
enriched Waldhausen categories.  We write $WTHH(f)=|WTHH(\Sdot f)|$.

To put this construction in perspective, we have an alternative
description of $\Sdot f$ as a pullback.  For any simplicial object
$Z\subdot$, we can form the ``path'' object $PZ\subdot$ precomposing
with the shift operation $[n] \mapsto [n+1]$ in the category of
standard simplices (or finite ordered sets).  In this notation, we
have a pullback square
\[
\xymatrix{%
\Sdot f\ar[r]^{X}\ar[d]_{Y}&P\Sdot \aD\ar[d]^{d_{0}}\\
\Sdot \aC\ar[r]_{f}&\Sdot \aD
}
\]
in the category of simplicial simplicially enriched categories.  The
usual extra degeneracy argument produces a simplicial null homotopy on
$P\Sdot \aC$, and applying $WTHH^{\Gamma}$ and (when appropriate)
$WTHH$, we get  commutative  squares of cyclotomic spectra 
\[
\xymatrix@C-1pc{%
WTHH^{\Gamma}(f)\ar[r]\ar[d]&|WTHH^{\Gamma}(P\Sdot \aD)|\ar[d]
&WTHH(f)\ar[r]\ar[d]&|WTHH(P\Sdot \aB)|\ar[d]\\
|WTHH^{\Gamma}(\Sdot\aC)|\ar[r]&|WTHH^{\Gamma}(\Sdot\aD)|
&|WTHH (\Sdot\aA)|\ar[r]&|WTHH(\Sdot\aB)|
}
\]
where the top right entry comes with a canonical null homotopy through
cyclotomic maps.    We therefore get a map of cyclotomic spectra from
$WTHH^{\Gamma}(f)$ to the homotopy fiber of the map 
$|WTHH^{\Gamma} (\Sdot\aC)|\to |WTHH^{\Gamma} (\Sdot\aD)|$, which is
equivalent to the 
homotopy cofiber of the map $WTHH^{\Gamma}\aC\to WTHH^{\Gamma}\aD$.
Likewise in the enhanced exact context, we get a map of cyclotomic
spectra from $WTHH(f)$ to the homotopy cofiber of the map
$WTHH\aA\to WTHH \aB$. The Cofiber Theorem asserts that these maps are
weak equivalences.

\begin{thm}[Cofiber Theorem]\label{thmcofiber}
For $f \colon \aC \to \aD$ an enriched exact functor, we have a
cofiber sequence of cyclotomic spectra
\[
WTHH^{\Gamma}(\aC)\to WTHH^{\Gamma}(\aD)\to WTHH^{\Gamma}(f)
\to |WTHH^{\Gamma}(\Sdot\aC)|.
\]
For $f\colon \aA\to \aB$ an enhanced exact functor, we have a cofiber
sequence of cyclotomic spectra
\[
WTHH(\aA)\to WTHH(\aB)\to WTHH(f)\to |WTHH(\Sdot\aA)|.
\]
\end{thm}

\begin{proof}(cf.~\cite[1.5.5]{WaldhausenKT})
The argument for the connective and non-connective enrichments are
identical; we treat the connective case in detail.
Consider the map 
\[
WTHH^{\Gamma}(\aD)\vee \myop\bigvee_{n} WTHH^{\Gamma}(\aC)
\to WTHH^{\Gamma}(\aD\times \Sdot[n]\aC)
\to WTHH^{\Gamma}(\Sdot[n]f)
\]
induced by sending $b, a_{1},\dotsc,a_{n}$ to $(b,Y)$ and then $(X,Y)$
with  $Y=(y_{i.j})$ for
\[
y_{i,j}=a_{i+1}\vee\dotsb \vee a_{j}
\]
and $X=(x_{i,j})$ for
\[
x_{i,j}=\begin{cases}
b\vee f(y_{0,j-1})&i=0\\
f(y_{i-1,j-1})&i>0
\end{cases}
\]
with the canonical maps induced by inclusions and quotients of
summands.  Applying the argument following
Corollary~\ref{coradditivity}, we see that this 
map is a weak equivalence.  Letting $n$ vary, these assemble into a
simplicial map where we regard the domain as the simplicial cyclotomic
spectrum 
\[
WTHH^{\Gamma}(\aD)\cup_{WTHH^{\Gamma}(\aC)}
WTHH^{\Gamma}(\aC)\sma \Delta[1].
\]
On geometric realization, this induces a map
from the homotopy cofiber
\[
\cofiber{WTHH^{\Gamma}(\aD)}{WTHH^{\Gamma}(\aC)} =
WTHH^{\Gamma}(\aD)\cup_{WTHH^{\Gamma}(\aC)}
WTHH^{\Gamma}(\aC)\sma I
\]
to $ WTHH^{\Gamma}(f)$
that we see is a weak equivalence.  The composite map
\[
\cofiber{WTHH^{\Gamma}(\aD)}{WTHH^{\Gamma}(\aC)}
\to WTHH^{\Gamma}(f)
\to WTHH^{\Gamma}(\Sdot \aC)
\]
factors as the connecting map
$\cofiber{WTHH^{\Gamma}(\aD)}{WTHH^{\Gamma}(\aC)}
\to \Sigma WTHH^{\Gamma}(\aC)$
composed with the weak equivalence $\Sigma WTHH^{\Gamma}(\aC)\to
|WTHH^{\Gamma}(\Sdot \aC)|$.
\end{proof}

Using the alternate models $\wt{WTC}$ and $\wt{WTHH}$ of
Definition~\ref{deftildethh}, we get constructions
$\wt{WTC}^{\Gamma}(f)$ and 
$\wt{WTHH}^{\Gamma}(f)$ that admit a cyclotomic trace from $K$-theory.
Because on objects, the map constructed in the proof of
Theorem~\ref{thmcofiber} agrees with the corresponding map in
cofiber sequence on $K$-theory, we get the following theorem as an
immediate consequence.

\begin{thm}
For $f\colon \aC\to \aD$ an enriched exact functor, the following
diagram commutes.
\[
\xymatrix@-1pc{%
K(\aC_{0})\ar[r]\ar[d]_{\trc}
&K(\aD_{0})\ar[r]\ar[d]_{\trc}
&K(f)\ar[r]\ar[d]_{\trc}
&K(\Sdot\aC_{0})\ar[d]_{\trc}\\
\wt{WTC}^{\Gamma}(\aC)\ar[r]\ar[d]
&\relax\wt{WTC}^{\Gamma}(\aD)\ar[r]\ar[d]\ar[d]
&\relax\wt{WTC}^{\Gamma}(f)\ar[r]\ar[d]
&\relax\wt{WTC}^{\Gamma}(\Sdot\aC)\ar[d]\\
\relax\wt{WTHH}^{\Gamma}(\aC)\ar[r]
&\relax\wt{WTHH}^{\Gamma}(\aD)\ar[r]
&\relax\wt{WTHH}^{\Gamma}(f)\ar[r]
&\relax\wt{WTHH}^{\Gamma}(\Sdot\aC)
}
\]
\end{thm}

Returning to Theorem~\ref{thmcofiber}, we have the following corollary
that allows us to study the cofibers of exact functors in 
``$THH$-theoretic'' terms.

\begin{cor}\label{corcocart}
Let $f\colon \aA\to \aB$ and $g\colon \aC\to \aD$ be enriched
exact functors.  Then the commutative square of cyclotomic spectra on
the left is homotopy (co)cartesian.
\[
\xymatrix@C-1pc{%
WTHH^{\Gamma}(\aB)\ar[r]\ar[d]&WTHH^{\Gamma}(f)\ar[d]
&WTHH(\aB)\ar[r]\ar[d]&WTHH(f)\ar[d]\\
WTHH^{\Gamma}(\aC)\ar[r]&WTHH^{\Gamma}(g\circ f)
&WTHH(\aC)\ar[r]&WTHH(g\circ f)
}
\]
If $f$ and $g$ are enhanced exact then the commutative square of
cyclotomic spectra on the right is homotopy cartesian.
\end{cor}

In the special case when $\aC$ is a simplicially enriched Waldhausen
subcategory of $\aD$ 
and $f$ is the inclusion, $S\subdot f$ admits an equivalent but
smaller variant where we omit the choices of subquotients.

\begin{defn}\label{defF}
\index{Fdot(D,C)@$\Fdot(\aD,\aC)$}
We say that $\aC\subset \aD$ is a 
\term{simplicially enriched Waldhausen subcategory}%
\index{Waldhausen subcategory!simplicially enriched}
when $\aC\subset \aD$ is full as a
simplicially enriched category and $\aC_{0}$ is a Waldhausen
subcategory of $\aD_{0}$. In this case we define $\Fdot(\aD,\aC)$
to be the simplicially enriched Waldhausen subcategory of the nerve of
the cofibrations in $\aD$ whose quotients lie in $\aC$.
\end{defn}

Concretely, $\Fdot[n](\aD,\aC)$ has as objects the  composable sequences
of $n$ cofibrations
\[
\xymatrix@C-1pc{%
x_{0}\ar@{ >->}[r]&x_{1}\ar@{ >->}[r]&\dotsb\ar@{ >->}[r]&x_{n}
}
\]
such that $x_{i+1}/x_{i}$ is an object of $\aC$ for all $i$, with maps
the simplicial sets of natural transformations.  We have a forgetful
functor from $S\subdot (\aC\sto \aD)$ to $\Fdot(\aD,\aC)$ that
throws away the subquotients, i.e., sending $(X,Y)$ in
$\Sdot[n+1]\aD\times \Sdot[n]\aC$ to 
\[
\xymatrix@C-1pc{%
x_{0,1}\ar@{ >->}[r]&x_{0,2}\ar@{ >->}[r]&\dotsb\ar@{ >->}[r]&x_{0,n+1}
}
\]
in $\Fdot[n](\aD,\aC$), where $X=(x_{i,j})$.  At each simplicial level
this map is an equivalence of simplicial Waldhausen categories, and in
particular induces a DK-equivalence
\[
\Sdot[m]\Sdot[n](\aC\sto \aD) \to \Sdot[m]\Fdot[n](\aD,\aC).
\]
We therefore obtain the following observation, useful in combination
with Theorem~\ref{thmcofiber}.

\begin{prop}\label{propcofiber}
For $\aC\subset \aD$ a simplicially enriched Waldhausen subcategory,
the forgetful functor from $S\subdot(\aC\to\aD)$ to
$\Fdot(\aD,\aC)$ induces a weak equivalence of cyclotomic spectra
\[
WTHH^{\Gamma}(\aC\sto\aD)\to |WTHH^{\Gamma}(\Fdot(\aD,\aC))|.
\]
\end{prop}

We have the notion of a 
\indexterm{closed}{simplicially enriched Waldhausen
subcategory!closed}\index{closed (Waldhausen subcategory)}
simplicially enriched Waldhausen
subcategory, which is a simplicially enriched Waldhausen subcategory
$\aA\subset \aB$ where $\aA_{0}$ is a closed Waldhausen subcategory of
$\aB_{0}$ (i.e., every object of $\aB$ weakly equivalent to an object
of $\aA$ is in $\aA$).  When $\aB$ is an enhanced simplicially enriched
Waldhausen category and $\aA\subset \aB$ is a closed simplicially
enriched Waldhausen subcategory, then $\aA$ is also enhanced
simplicially enriched.  The discussion above then generalizes to show
that 
\[
\Sdot[m]\Sdot[n](\aC\sto \aD) \to \Sdot[m]\Fdot[n](\aD,\aC).
\]
induces an equivalence (and in particular DK-equivalence) on
non-connective enrichments.  It follows that 
\[
WTHH(\aA\sto \aB)\to |WTHH(\Fdot(\aB,\aA))|
\]
is also a weak equivalence of cyclotomic spectra.

%%%%%%%%%%%%%%%%%%%%%%%%%%%%%%%%%%%%%%%%
\section{The Localization Theorem}\label{sec:loc}

The Localization Theorem, called by Waldhausen the ``Fibration
Theorem'', provides the most important instance of the Cofiber
Theorem.  Roughly speaking, this theorem states that algebraic
$K$-theory takes quotient sequences of triangulated categories
to cofiber sequences of spectra.  In this section, we prove versions
of this theorem for $THH$ and $TC$.  In the case of the non-connective
enrichment, we obtain a localization sequence equivalent to the one in
\cite{BlumbergMandellTHHLoc}; in the case of the connective
enrichment, we obtain a localization sequence generalizing the one in
\cite{HM3} (q.v.~Chapter~\ref{tw-loc}).  

For the setup for the Localization Theorem, we take an enhanced
simplicially enriched Waldhausen category $\aA$ together with an
additional subcategory of weak equivalences $v\aA_{0}$ that contains
its usual weak equivalences $w\aA_{0}$.  We assume that $v\aA_{0}$
satisfies the \term{two-out-of-three property}, meaning that for composable
maps $f$ and $g$, if any two of $f$, $g$, and $g\circ f$ are in
$v\aA_{0}$, then so is the third.  We also assume that $v\aA_{0}$
satisfies the \term{Extension Axiom}
\cite[\S1.2]{WaldhausenKT}, meaning that given a map of cofibration
sequences
\[
\xymatrix@-1pc{%
x\ar@{ >->}[r]\ar[d]_{v}&y\ar[r]\ar[d]&y/x\ar[d]^{v}\\
x'\ar@{ >->}[r]&y'\ar[r]&y'/x'
}
\]
with the outer maps $x\to x'$ and $x/y\to x'/y'$ in $v\aA_{0}$, then
the inner map $y\to y'$ is in $v\aA_{0}$.  Finally, recalling that as
an enhanced simplicially enriched Waldhausen category, $\aA$ admits 
tensors with contractible simplicial sets, we say that $v\aA_{0}$ is 
\term{compatible with cylinders} when for any map $x\to x'$ in
$v\aA_{0}$, the map
\[
x\to x' \cup_{x} (x\otimes \Delta[1])
\]
is a cofibration in $\aA_{0}$, i.e., its quotient is in $\aA$.  The category
of \specialterm{$v$-acyclics}{v-acyclics} $\aA_{0}^{v}$ consists of the full subcategory
of objects $v$-equivalent to the trivial object $*$.  Under these
hypotheses, $\aA_{0}^{v}$ forms a
closed Waldhausen subcategory of $\aA$.  Moreover, $\aA_{0}^{v}$ 
is \noindexterm{closed under extensions and cofibers} in
$\aA_{0}$, meaning that for a cofibration sequence in $\aA_{0}$
\[
\xymatrix@-1pc{%
x\ar@{ >->}[r]&y\ar[r]&y/x,
}
\]
if $x$ and either of $y$ or $y/x$ is in $\aA_{0}^{v}$, then so is the
other.  Letting $\aA^{v}$ be the full simplicially enriched
subcategory of $\aA$ consisting of the objects in $\aA^{v}_{0}$, then
$\aA^{v}$ forms an enhanced simplicially enriched Waldhausen category
with the inclusion functor $\aA^{v}\to \aA$ enhanced exact.  We can
now state the Localization Theorem.

\begin{thm}[Localization Theorem]\label{thmloc}
With hypotheses and notation as in the previous paragraph, the
following commutative squares of cyclotomic spectra are homotopy
(co)cartesian.
\[
\xymatrix{%
WTHH^{\Gamma}(\aA^{v})\ar[r]\ar[d]&WTHH^{\Gamma}(\aA^{v}|v)\ar[d]
&WTHH(\aA^{v})\ar[r]\ar[d]&WTHH(\aA^{v}|v)\ar[d]\\
WTHH^{\Gamma}(\aA)\ar[r]&WTHH^{\Gamma}(\aA|v)
&WTHH(\aA)\ar[r]&WTHH(\aA|v)
}
\]
Moreover, in each square, the upper right entry is null homotopic
through cyclotomic maps.  Thus, we have cofiber sequences of
cyclotomic spectra,
\begin{gather*}
WTHH^{\Gamma}(\aA^{v})\to WTHH^{\Gamma}(\aA)\to
WTHH^{\Gamma}(\aA|v)\to \Sigma WTHH^{\Gamma}(\aA^{v})\\
WTHH(\aA^{v})\to WTHH(\aA)\to WTHH(\aA|v)\to \Sigma WTHH(\aA^{v}).
\end{gather*}
\end{thm}

Although formally similar in statement and proof, the two localization
sequences above are very different in practice.  In the case when
$\aA$ is pretriangulated (which by Corollary~\ref{cor:pretriang} 
just means in this context that every object is weakly equivalent to a
suspension), the Localization Theorem of \cite{BlumbergMandellTHHLoc}
(Theorem~\ref{thmgenone} above) identifies the relative term $WTHH(\aA|v)$ in the second sequence
above as the $THH$ of the triangulated quotient
$\pi_{0}\aA^{S}/\pi_{0}(\aA^{v})^{S}$ (for any 
spectrally enriched model of this quotient).  

In the special case when $\aA$ is the category of finite cell EKMM
$R$-modules for the $S$-algebra $R=HA$ for a discrete valuation ring
$A$ or $R=ku$ is connective $K$-theory, we take the $v$-equivalences
$v\aA_{0}$ to be the $R[\beta^{-1}]$-equivalences, the maps that
induce isomorphisms on homotopy groups after inverting $\beta$, where
$\beta$ is a uniformizer for $A$ (when $R=HA$) or is the Bott-element
(when $R=ku$).  Then Theorem~\ref{thmspecsphere} (proved in
Section~\ref{secspheretheorem}) combined with
Theorem~\ref{thmnoSnoncon} identify both $WTHH^{\Gamma}(\aA)$ and
$WTHH(\aA)$ as $THH(R)$.  In the non-connective case, we then have
that $WTHH(\aA|v)$ is equivalent to $THH(R[\beta^{-1}])$. 
Calculations show $WTHH(\aA^{v})$ cannot be equivalent to
$THH(R/\beta)$.  On the other hand, we will prove a d\'evissage theorem
in Part~4 that identifies $WTHH^{\Gamma}(\aA^{v})$ as $THH(R/\beta)$
and calculations show that $WTHH^{\Gamma}(\aA|v)$ cannot be equivalent
to $THH(R[\beta^{-1}])$.

Returning to Theorem~\ref{thmloc}, it follows that the analogous
squares in the ``tilde'' models $\wt{WTHH}^{\Gamma}$ and $\wt{WTC}^{\Gamma}$
are homotopy (co)cartesian as well, and we get cofiber sequences on
$\wt{WTHH}^{\Gamma}$ and $\wt{WTC}^{\Gamma}$.  By naturality, the maps in the
squares and in the cofiber sequences commute with the cyclotomic
trace.  For convenient reference, we state this explicitly in the
following theorem.

\begin{thm}
Under the hypotheses of Theorem~\ref{thmloc}, the following diagram of
cofiber sequences commutes.
\[
\xymatrix@-1pc{%
K(\aA_{0}^{v})\ar[r]\ar[d]_{\trc}&K(\aA_{0})\ar[r]\ar[d]_{\trc}
&K(\aA_{0}|v)\ar[r]\ar[d]_{\trc}
&\Sigma K(\aA_{0}^{v})\ar[d]_{\trc}\\
\wt{WTC}^{\Gamma}(\aA^{v})\ar[r]\ar[d]
&\wt{WTC}^{\Gamma}(\aA)\ar[r]\ar[d]
&\wt{WTC}^{\Gamma}(\aA|v)\ar[r]\ar[d]
&\Sigma \wt{WTC}^{\Gamma}(\aA^{v})\ar[d]\\
\wt{WTHH}^{\Gamma}(\aA^{v})\ar[r]
&\wt{WTHH}^{\Gamma}(\aA)\ar[r]
&\wt{WTHH}^{\Gamma}(\aA|v)\ar[r]
&\Sigma \wt{WTHH}^{\Gamma}(\aA^{v})
}
\] 
\end{thm}

We begin the proof of Theorem~\ref{thmloc} by noting that the category
of $v$-acyclics completely characterizes the $v$-equivalences $v\aA_{0}$.

\begin{prop}
Under the hypotheses of Theorem~\ref{thmloc}, a map $f\colon x\to y$ is in
$v\aA_{0}$ if and only if the homotopy cofiber 
\[
Cf=y \cup_{x}(x\otimes \Delta[1])\cup_{x}*
\]
is in $\aA_{0}^{v}$.
\end{prop}

\begin{proof}
Let $Mf=y\cup_{x}(x\otimes \Delta [1])$ so that $Cf=Mf/x$.  The map $Mf\to y$
is a weak equivalence (and so  in particular a $v$-equivalence) and
the composite map $x\to Mf\to y$ is $f$ and so the inclusion of  $x$
in $ Mf$ is in $v\aA_{0}$ if and only if $f$ is. Consider the
commutative diagram of cofiber sequences
\[
\xymatrix@-1pc{%
x\ar@{ >->}[r]^{=}\ar[d]_{=}&x\ar[r]\ar[d]&\relax*\ar[d]\\
x\ar@{ >->}[r]&Mf\ar[r]&Cf.
}
\]
By the Gluing Axiom, $Cf$ is in $\aA_{0}^{v}$ when $x\to Mf$ is in
$v\aA_{0}$. By the Extension Axiom $x\to Mf$ is in $v\aA_{0}$ when
$Cf$ is in $\aA_{0}^{v}$.
\end{proof}

Let $\bar v\aA_{0}=v\aA_{0}\cap \co\aA_{0}$ denote the subcategory of
$\aA_{0}$ consisting of the maps that are both cofibrations and
$v$-equivalences.  The previous proposition implies that $\bar
v\aA_{0}$ consists of those cofibrations whose quotients are
$v$-acyclic.  It follows that $\Fdot(\aA,\aA^{v})=\bar v\subdot\aA$,
and applying Corollary~\ref{corcocart} and Proposition~\ref{propcofiber}, we get
homotopy (co)cartesian squares
\[
\xymatrix@C-1pc{%
WTHH^{\Gamma}(\aA^{v})\ar[r]\ar[d]
&\relax|WTHH^{\Gamma}(\bar v\subdot\aA^{v})|\ar[d]
&WTHH(\aA^{v})\ar[r]\ar[d]
&\relax|WTHH(\bar v\subdot\aA^{v})|\ar[d]\\
WTHH^{\Gamma}(\aA)\ar[r]
&\relax|WTHH^{\Gamma}(\bar v\subdot\aA)|
&WTHH(\aA)\ar[r]
&\relax|WTHH(\bar v\subdot\aA)|.
}
\]
We now have what we need to prove Theorem~\ref{thmloc}.

\begin{proof}[Proof of Theorem~\ref{thmloc}]
To obtain the homotopy (co)cartesian squares,
we just need to see that the maps
\[
WTHH^{\Gamma}(\bar v\subdot\aA)\to WTHH^{\Gamma}(v^{M}\subdot\aA)
\quad \text{and}\quad
WTHH(\bar v\subdot\aA)\to WTHH(v^{M}\subdot\aA)
\]
are weak equivalences.  The inclusion of $|\bar v_{p}\Sdot[q]\aA|$ in
$|v_{p}^{M}\Sdot[q]\aA|$ is a DK-embedding and an easy mapping
cylinder argument shows that it is a DK-equivalence.

It follows that $WTHH^{\Gamma}(\aA^{v}|v)$ and
$WTHH^{\Gamma}(\aA^{v}|v)$ are weakly equivalent as cyclotomic spectra
to the trivial spectrum, and to produce a null homotopy through
cyclotomic maps is not much more work.  The simplicial object
$v^{M}\subdot\aA^{v}$ has an extra degeneracy which on objects inserts
the trivial map at the start of the chain of maps.  On maps, we use
the unique (constant trivial) homotopy on any subsimplex that has the
new trivial object as one of its vertices.
\end{proof}

%%%%%%%%%%%%%%%%%%%%%%%%%%%%%%%%%%%%%%%%
\section{The Sphere Theorem}\label{secspheretheorem}

In this section, we state versions of Waldhausen's ``Sphere Theorem''
for the $THH$ of Waldhausen categories, which we prove in the next section.
These theorems allow us to deduce the important consistency result
that all the different models for the $THH$ of the finite-cell modules
over an EKMM $S$-algebra or a simplicial ring agree
(Theorem~\ref{thmspecsphere} above).  Before stating a precise
theorem, we need two definitions,

\begin{defn}\label{defn:stable}
Let $\aC$ be a simplicially tensored Waldhausen category.  We say that
$\aC$ is \term{stable} when:
\begin{enumerate}
\item Every object of $\aC$ is weakly equivalent to a suspension, and
\item For all objects $x$ and $y$ in $\aC$, the suspension map
$\aC(x,y)\to \aC(\Sigma x,\Sigma y)$ is a weak equivalence.
\end{enumerate}
We say that $\aC$ is \term{almost stable} when it satisfies just
condition~(ii). 
\end{defn}

As observed in Corollary~\ref{cor:pretriang}, the first condition
implies that the non-connective spectral category $\aC^{S}$ is
pretriangulated, and its homotopy category $\pi_{0}\aC^{S}$ is
triangulated.  The second condition implies that the homotopy category
$\pi_{0}\aC^{S}$ coincides with the homotopy category $\pi_{0}\aC$ and
also that the connective spectral enrichment $\aC^{\Gamma}(x,y)$
is the connective cover of the non-connective spectral enrichment
$\aC^{S}(x,y)$ (Proposition~\ref{prop:conncover}).  Combined with the
fact that the mapping simplicial sets $\aC(x,y)$ are Kan complexes
(and that weak equivalences in $\aC$ are homotopy equivalences in the
obvious sense), this puts all the basic tools and techniques of homotopy
theory and stable homotopy theory at our disposal.

In the stable case the hypotheses we need for the Sphere Theorem
greatly simplify and so we will explore that case first.  In addition
to the stability assumptions above, we need to assume that $\aC$ is
generated by connective objects in the following sense.

\begin{defn}\label{defconnclass}
Let $\aC$ be an almost stable simplicially tensored Waldhausen category.
A \term{connective class} $Q$ in $\aC$ is a set of objects of $\aC$
such that for any $a,b$ in $Q$, $\aC^{S}(a,b)$ is connective.  If
$\aC$ is stable, then we say that $Q$ is 
\indexterm{generating}{connective class!generating}%
\index{generating (connective class)}
if the
smallest triangulated subcategory of the triangulated category $\pi_{0}
\aC^{S}$ that contains $Q$ is all of $\pi_{0} \aC^{S}$.
\end{defn}

See Definition~\ref{defn:asgen} for the definition of generating when
$\aC$ is almost stable.
In this terminology, we prove the following
theorem, the $THH$ analogue of Waldhausen's Sphere Theorem for the
stable case.

\begin{thm}[Sphere Theorem, Stable Version]\label{thm:spherestable}
Let $\aC$ be a stable simplicially tensored Waldhausen category and
assume that $\aC$ has a generating connective class $Q$.  Then the canonical
cyclotomic maps are weak equivalences
\[
WTHH^{\Gamma}(\aC) \overto{\sim} WTHH(\aC) 
\overfrom{\sim} 
THH(\aC^{S}) \overfrom{\sim} THH(Q^{S}).
\]
Here $Q^{S}$ denotes the full spectral subcategory of $\aC^{S}$ on
the objects of $Q$.
\end{thm}

We state the following corollary for ease of reference and citation;
it is one case of Theorem~\ref{thmspecsphere}.

\begin{cor}\label{cor:ekmm}
Let $R$ be a connective EKMM $S$-algebra and $\aC_{R}$ the
category of finite cell $R$-modules.  Then the canonical cyclotomic
maps 
\[
WTHH^{\Gamma}(\aC_{R}) \overto{\sim} WTHH(\aC_{R}) 
\overfrom{\sim} 
THH(\aC_{R}^{S}) \overfrom{\sim} THH(R)
\]
are weak equivalences.
Here $THH(R)$ denotes the usual B\"okstedt model of $THH$ of a
particular symmetric ring spectrum (or FSP) equivalent to $R$.
\end{cor}

The corollary follows by taking the connective class $Q$ to be the
singleton set containing the one object $S_{R}$, the ``cell zero
sphere $R$-module'' \cite[III.2]{EKMM}.  Then $THH(Q^{S})$ coincides
with $THH$ of the symmetric ring spectrum 
\[
F=Q^{S}(S_{R},S_{R})=\aC_{R}^{S}(S_{R},S_{R}).
\]
Concretely, this has $n$-th space
\[
F(n)=\aC_{R}(S_{R},S_{R}\sma S^{n}),
\]
and multiplication induced by composition.  We can identify this as
the symmetric ring spectrum (or ``FSP defined on
spheres'') obtained from the FSP
$\mathbf{F}(-)=\aC_{R}(S_{R},S_{R}\sma -)$ by restricting to
spheres $F(n)=\mathbf{F}(S^{n})$.  

Another symmetric ring spectrum derives from the general theory
of~\cite{SchwedeEKMM}; writing $\aM_{S}$ for the category of
$S$-modules, this has spaces $\Phi(n)=\aM_{S}((S_{S}^{-1})^{(n)},R)$ and
multiplication induced by smash product together with the
multiplication on $R$.  Experts know how to compare these symmetric
ring spectra and therefore their $THH$, $TR$, and $TC$ spectra:
Briefly, noting that $S_{R}=R\sma S_{S}$, we construct a third
symmetric ring spectrum $\Phi'$ that lies between them.  $\Phi'$ has
spaces
\begin{align*}
\Phi'(n)
&=\aC_{R}(S_{R}\sma_{S}(S_{S}^{-1} \sma S^{1})^{(n)},S_{R}\sma S^{n})\\
&\iso \aM_{S}((S_{S}^{-1} \sma S^{1})^{(n)},F_{R}(S_{R},S_{R}\sma S^{n}))
\end{align*}
and multiplication induced both by smash product (on the
$(S_{S}^{-1} \sma S^{1})^{(n)}$ factors) and composition (on the
$F_{R}(S_{R},S_{R}\sma S^{n})$ factors).  We have a weak equivalence
of symmetric ring spectra from $F$ to $\Phi'$ given by 
\[
F(n)=
\aC_{R}(S_{R},S_{R}\sma S^{n}) \to
\aC_{R}(S_{R}\sma_{S}(S_{S}^{-1} \sma S^{1})^{(n)},S_{R}\sma S^{n})
=\Phi'(n)
\]
induced by the collapse map $S_{S}^{-1}\sma S^{1}\to S$; the induced
map $F(n)\to \Phi'(n)$ is a weak equivalence of simplicial sets for
all $n$.  We have a
weak equivalence of symmetric ring spectra from $\Phi$ to $\Phi'$ given by
\begin{multline*}
\Phi(n)=
\aM_{S}((S_{S}^{-1})^{(n)},R)\to
\aM_{S}((S_{S}^{-1})^{(n)},F_{R}(S_{R},S_{R}))\\
\to
\aM_{S}((S_{S}^{-1}\sma S^{1})^{(n)},F_{R}(S_{R},S_{R}\sma S^{n}))
\iso \Phi'(n)
\end{multline*}
induced by the unit map $R\to F_{R}(S_{R},S_{R})$ (which arises from
the extra $R$ action on $S_{R}=R\sma S_{S}$); again, this is a weak
equivalence of simplicial sets for all $n$.  For convenience, we
state these remarks as a proposition.

\begin{prop}\label{prop:corEKMM}
The symmetric ring spectrum in Corollary~\ref{cor:ekmm} is weakly
equivalent to the symmetric ring spectrum obtained from the EKMM
$S$-algebra $R$ by~\cite{SchwedeEKMM}.
\end{prop}

For the other half of Theorem~\ref{thmspecsphere}, we need to treat
the almost stable case.  This requires introducing the following
subcategories of $\aC$ associated to a connective class $Q$.

\begin{notn}\label{notn:SigmaQ}
Let $\aC$ be an almost stable  simplicially tensored Waldhausen
category and let $Q$ be a connective class.  Write $\aQ$ for the
smallest closed Waldhausen category of $\aC$ containing $Q$.  
For $n\geq 0$, let $\Sigma^{n}\aQ$ denote the full subcategory of $\aC$
containing all objects weakly equivalent to $\Sigma^{n}x$ for $x$
in $\aQ$.  Let $\Sigma^{-n}\aQ$ denote the 
full subcategory of $\aC$ containing all $x$ such that
$\Sigma^{n}x$ is in $\aQ$.
\end{notn}

We note that the subcategories $\Sigma^{n}\aQ$ are themselves
connective classes and closed Waldhausen subcategories.

\begin{prop}\label{prop:wccc}
Let $\aC$ be an almost stable  simplicially tensored Waldhausen
category and let $Q$ be a connective class.  Then 
$\Sigma^{n}\aQ$ is a connective class  and closed Waldhausen
subcategory  for all $n\in \bZ$. 
\end{prop}

\begin{proof}
We begin by showing that $\aQ$ is a connective class; stability
hypothesis~(ii) then shows that $\Sigma^{n}\aQ$ is a connective class
for all $n$.
Let $Q_{0}$ be the collection of objects of $\aC$ weakly equivalent to
finite coproducts of objects in $Q$, and inductively let $Q_{n}$ be the
collection of objects of $\aC$ that are weakly equivalent to finite
coproducts of homotopy pushouts $y\cup_{x}(x\otimes \Delta[1])\cup_{x}z$ where
$x,y,z\in Q_{n-1}$ and $y\cup_{x}(x\otimes \Delta[1])\cup_{x}*\in Q_{n-1}$.
If we regard $\bigcup Q_{n}$ as the full subcategory of $\aC$ of objects in $Q_{n}$ for
some $n$, it is then clear that $\aQ=\bigcup Q_{n}$ is the smallest closed
Waldhausen subcategory of $\aC$ containing $Q$.  To show that $\aQ$ is
a connective class, it suffices to show that for $x,y$ in $Q_{n}$,
$\aC^{S}(x,y)$ is connective, which we do by induction.  We know that
$x$ is weakly equivalent to a finite coproduct of homotopy pushouts of
objects in $Q_{n-1}$ along maps whose homotopy cofiber is also in
$Q_{n-1}$.  Looking at the long exact sequence of homotopy
groups from the fibration sequence in
Proposition~\ref{propfibercofiber}, we then see that $\aC(x,z)$ is
connective for all $z$ in $Q_{n-1}$. Using the same fact about $y$ and
the long exact sequence of homotopy groups from the cofibration
sequence in Proposition~\ref{propfibercofiber}, we see that $\aC(x,y)$
is connective.

By definition $\Sigma^{0}\aQ=\aQ$ is a closed Waldhausen subcategory
and it follows that $\Sigma^{n}\aQ$ is a closed Waldhausen subcategory
for $n<0$ since suspension preserves homotopy pushouts.  Let $n>0$ and
suppose $f\colon x\to y$ is a cofibration in $\aC$ such that $x$, $y$, and
$y/x\htp Cf$ are all in $\Sigma^{n}\aQ$.  Then we can find $x'$ and
$y'$ in $\aQ$ and weak equivalences $\Sigma^{n}x'\to x$ and
$\Sigma^{n}y'\to y$.  By stability hypothesis~(ii) and the fact that
the mapping spaces in $\aC$ are Kan complexes, we can find a
map $f'\colon x'\to y'$ such that the diagram
\[
\xymatrix{%
\Sigma^{n}x'\ar[r]^{\Sigma^{n}f'}\ar[d]_{\sim}&\Sigma^{n}y'\ar[d]_{\sim}\\
x\ar[r]_{f}&y
}
\]
commutes up to homotopy.  Choosing a homotopy, we get a weak
equivalence $C\Sigma^{n} f'\to Cf$.  Then $y/x$ is weakly equivalent
to $\Sigma^{n}Cf$ and it follows (again applying
stability hypothesis~(ii)) that $Cf$ is in $\aQ$.  For any map $x\to
z$ with $z$ in $\Sigma^{n}\aQ$, we can choose a compatible map $x'\to
z'$ (for some $z'$ with $\Sigma^{n}z'\htp z$) such that the pushout
$w=z\cup_{x}y$ is weakly equivalent to $\Sigma^{n}$ of the homotopy
pushout $w'=z'\cup_{x'}(x'\otimes \Delta[1])\cup_{x'}y'$.  Since $\aQ$
is a closed Waldhausen subcategory of $\aC$, $w'$ is in $\aQ$, and it
follows that $w$ is in $\Sigma^{n}\aQ$.  This shows that
$\Sigma^{n}\aQ$ is a closed Waldhausen subcategory of $\aC$.
\end{proof}

We use the subcategories $\Sigma^{n}\aQ$ to define what it means for a
connective class to be generating in the almost stable case.  For this,
we need the following technical definitions.

\begin{defn}\label{defn:built}
Given a class $A$ of objects of a Waldhausen category $\aC$, we say
that an object $x$ of $\aC$ is \term{finitely cellularly built from}
$A$ if we can find a sequence of objects $x_{0},x_{1},\dotsc,x_{n}$ of
$\aC$ that fit into pushout squares
\[
\xymatrix@-1pc{%
a_{j}\ar@{ >->}[r]\ar[d]&b_{j}\ar[d]\\
x_{j}\ar@{ >->}[r]&x_{j+1}
}
\]
where $x_{n}$ is weakly equivalent to $x$, $x_{0}=*$, and for each $j$,
$a_{j}$ is in $A$, $b_{j}$ is contractible (weakly equivalent to $*$),
and $a_{j}\to b_{j}$ is a cofibration. 
\end{defn}

The concept of ``finitely cellularly built from'' above differs from other
notions of ``built from'' in other contexts. Note in particular that
an object of $A$ is not necessarily finitely cellularly built from
$A$.  However, suspensions of objects of $A$ are finitely cellularly
built from $A$, for example.

\begin{defn}\label{defn:asgen}
Let $\aC$ be an almost stable simplicially tensored Waldhausen
category and $Q$ a connective class.  We say that $Q$ is
\indexterm{generating}{connective class!generating}%
\index{generating (connective class)}
if every object of $\aC$ is finitely cellularly built from $\bigcup
\Sigma^{n}\aQ$. 
\end{defn}

The following proposition clarifies the relationship between the
notions of generating given in Definitions~\ref{defconnclass}
and~\ref{defn:asgen}.

\begin{prop}\label{prop:generating}
Let $\aC$ be a stable simplicially tensored Waldhausen
category and $Q$ a connective class.  Then $Q$ is generating in the
sense of Definition~\ref{defconnclass} if and only if it is generating
in the sense of Definition~\ref{defn:asgen}.
\end{prop}

\begin{proof}
Since the triangulated subcategory generated by $Q$ contains $\bigcup
\Sigma^{n}\aQ$ (cf.\ the proof of Proposition~\ref{prop:wccc}), one
direction is clear.  We must show that if $Q$ is generating in the
sense of Definition~\ref{defconnclass}, then it is generating in the
sense of Definition~\ref{defn:asgen}.  Let $C_{0}=\bigcup
\Sigma^{n}\aQ$, and inductively let $C_{n}$ be the collection of
objects of $\aC$ that are weakly equivalent to the homotopy cofiber of
a map between objects of $C_{n-1}$.  We note that the $C_{n}$ are
closed under suspension and desuspension and that the objects of
$C_{0}$ are finitely cellularly built from $C_{0}$ (since each is equivalent to
the suspension of an object of $C_{0}$).  Generating in the sense of
Definition~\ref{defconnclass} implies that $\aC=\bigcup C_{n}$, so it
suffices to show by induction that all objects of $C_{n}$ are
finitely cellularly built from $C_{0}$.  Given $f\colon x\to y$ with $x,y$ in
$C_{n-1}$, we need to show that $z=Cf=y\cup_{x}(x\otimes
\Delta[1])\cup_{x}*$ is finitely cellularly built from $C_{0}$.  Replacing $z$
with a weakly equivalent object, we can assume
without loss of generality that $x$ and $y$ are isomorphic rather than
just weakly equivalent to an iterated pushout.  Then we build $z$ by
first building $y$ and then gluing $Cb_{j}=(b_{j}\otimes \Delta
[1])\cup_{b_{j}}*$ along $a'_{j}=b_{j}\cup_{a_{j}}(a_{j}\otimes
\Delta[1])\cup_{a_{j}}*$ where $a_{j}\to b_{j}$ build $x$.  Since
$Cb_{j}$ is contractible and $a'_{j}$ is weakly equivalent to $\Sigma
a_{j}$, this shows that $z$ is finitely cellularly built from $C_{0}$.
\end{proof}

The following theorem now generalizes
Theorem~\ref{thm:spherestable} to the almost stable case.

\begin{thm}[Sphere Theorem]\label{spheretheorem}
Let $\aC$ be an almost stable simplicially tensored Waldhausen category and
assume that $\aC$ has a generating connective class $Q$.
Then the canonical 
cyclotomic maps are weak equivalences
\[
WTHH^{\Gamma}(\aC) \overto{\sim} WTHH(\aC) 
\overfrom{\sim} 
THH(\aC^{S}) \overfrom{\sim} THH(Q^{S}).
\]
where $Q^{S}$ denotes the full spectral subcategory of $\aC^{S}$ on
the objects of $Q$.
\end{thm}

We now have the other half of Theorem~\ref{thmspecsphere} as a
corollary. 

\begin{cor}\label{cor:simpring}
Let $A$ be a simplicial ring, let $\aC_{A}$ be the category of finite
cell $A$-modules and let $\aP_{A}$ be the category of finite cell
$A$-modules built out of finitely generated projective $A$-modules.
Then the canonical cyclotomic maps
\[
\xymatrix{%
WTHH^{\Gamma}(\aC_{A})\ar[d]_{\sim}\ar[r]^{\sim}
&WTHH(\aC_{A})\ar[d]_{\sim}
&THH(\aC_{A}^{S})\ar[l]_{\sim}\ar[d]_{\sim}&THH(A)\ar[l]_{\sim}\\
WTHH^{\Gamma}(\aP_{A})\ar[r]^{\sim}
&WTHH(\aP_{A})&THH(\aP_{A}^{S})\ar[l]_{\sim}
}
\]
are weak equivalences.
\end{cor}

The vertical arrows are weak equivalences
by Theorem~\ref{thmfactor} since every object of $\aP_{A}$
or $\Sdot[n]\aP_{A}$ is a direct summand of an object of $\aC_{A}$ or
$\Sdot[n]\aC_{A}$.  We get the top row from
Theorem~\ref{spheretheorem} taking the connective class $Q$ to be the
singleton set containing the object $A$, which is clearly generating.
The symmetric spectrum $Q^{S}(A,A)=\aC^{S}(A,A)$ is just the usual
symmetric ring spectrum constructed from $A$.

We have one last version of the Sphere Theorem, which is closer in
spirit to Waldhausen's Sphere Theorem for $K$-theory.  It also has the
technical advantage of being stated purely in terms of the connective
enrichments.

\begin{thm}[Sphere Theorem, Alternate Version]\label{thm:altsphere}
Let $\aC$ be an almost stable simplicially tensored Waldhausen category,
let $Q$ be a generating connective class, and let $\aQ$ be the smallest
closed Waldhausen subcategory of $\aC$ containing $Q$. The inclusion
of $\aQ$ into $\aC$ induces 
a weak equivalence $WTHH^{\Gamma}(\aQ)\to WTHH^{\Gamma}(\aC)$.
\end{thm}

The previous theorem is equivalent to Theorem~\ref{spheretheorem}, but
to see this, we need more information about the categories $\Sdot\aQ$
implicit in the statement.  The following proposition has everything
we need for the comparison, plus what we need for the proofs in the
next section.

\begin{prop}\label{prop:sncc}
Let $\aC$ be an almost stable simplicially tensored Waldhausen
category and let $Q$ be a connective class.  Then $\Sdot[n]\aC$ is an
almost stable simplicially tensored Waldhausen category, and
$\Sdot[n]\aQ$ is a closed Waldhausen subcategory and a connective
class; moreover, $\Sdot[n]\Sigma^{m}\aQ=\Sigma^{m}\Sdot[n]\aQ$.  If
$Q$ is generating, then so is $\Sdot[n]\aQ$.
\end{prop}

\begin{proof}
We saw in Proposition~\ref{prop:Sdotinherit} that $\Sdot[n]\aC$ is
simplicially tensored; the fact that the tensor on $\Sdot[n]\aC$ is
objectwise on the diagram and the formula~\eqref{eqsn} for the mapping
spaces of $\Sdot[n]\aC$ prove that $\Sdot[n]\aC$ is almost stable.
Since $\aQ$ is a closed Waldhausen subcategory of $\aC$, $\Sdot[n]\aQ$
is a closed Waldhausen subcategory of $\Sdot[n]\aC$.  Again, the
formula~\eqref{eqsn} shows that the mapping spectra
$\Sdot[n]\aQ^{S}(A,B)$ are connective.  It is clear that
$\Sdot[n]\Sigma^{m}\aQ=\Sigma^{m}\Sdot[n]\aQ$ since both categories
are the functor categories whose objects are the sequences starting
with $*$ of $n$ composable cofibrations in $\aC$ between objects in
$\Sigma^{m}\aQ$ together with choices of quotients which also must be
in $\Sigma^{m}\aQ$.

Now assume that $Q$ is generating; it remains to show that
$\Sdot[n]\aQ$ is generating. For $n=2$, a typical object of $S_{2}\aC$
is of the form $Z=[\xymatrix@-1pc{x\ar@{ >->}[r]&y\ar@{->>}[r]&z }]$
for objects $x,y,z$ in $\aC$.  Replacing $Z$ with a weakly equivalent
object, we can assume without loss of generality that $x$ and $y$ are
isomorphic rather than just weakly equivalent to an iterated pushout
in Definition~\ref{defn:built}.
Clearly the objects
\[
[\xymatrix@-1pc{%
x\ar@{ = }[r]&x\ar@{->>}[r]&\relax *
}]
\qquad \text{and}\qquad
[\xymatrix@-1pc{%
\relax * \ar@{ >->}[r]&y\ar@{ = }[r]&y
}]
\]
can be finitely cellularly built using pushouts of objects of the same form.
We can then build
\[
Z'=[\xymatrix@-1pc{%
x\ar@{ >->}[r]&y\cup_{x}(x\otimes \Delta[1])\ar@{->>}[r]
&y\cup_{x}(x\otimes \Delta[1])\cup_{x}*
}]
\]
by first building
$[\xymatrix@-1pc{%
x\ar@{ >->}[r]&x \amalg y \ar@{->>}[r]&y
}]$
and then using pushouts over maps of the form
\begin{multline*}
[\xymatrix@-1pc{%
\relax * \ar@{ >->}[r]
& b_{j}\cup_{a_{j}}(a_{j}\otimes \Delta[1])\cup_{a_{j}}b_{j} \ar@{ = }[r]
&b_{j}\cup_{a_{j}}(a_{j}\otimes \Delta[1])\cup_{a_{j}}b_{j} 
}]\\
\to
[\xymatrix@-1pc{%
\relax * \ar@{ >->}[r]
& b_{j}\otimes \Delta[1] \ar@{ = }[r]
&b_{j}\otimes \Delta[1]
}]
\end{multline*}
where $a_{j}\to b_{j}$ are the cells building $x$.
Similar observations apply for $n>2$.
\end{proof}

Now Theorems~\ref{spheretheorem} and~\ref{thm:altsphere} are easily
seen to be equivalent by looking at the following diagram.
\[
\xymatrix{%
\Omega |THH(\Sdot\aQ^{\Gamma})|\ar[r]\ar[d]_{\sim}
&\Omega |THH(\Sdot\aC^{\Gamma})|\ar[d]\\
\Omega |THH(\Sdot\aQ^{S})|\ar[r]^{\sim}
&\Omega |THH(\Sdot\aC^{S})|
&THH(\aC^{S})\ar[l]_-{\sim}&THH(Q^{S}).\ar[l]_-{\sim}
}
\]
The lefthand vertical map is a weak equivalence since each map
$\Sdot[n]\aQ^{\Gamma}\to \Sdot[n]\aQ^{S}$ is a weak equivalence (and
in particular DK-equivalence) of spectral categories by
Propositions~\ref{prop:wccc} and~\ref{prop:sncc}, while the bottom horizontal
maps are weak equivalences by Theorem~\ref{thmthick}
(for the first and third maps) and Theorem~\ref{thmnoSnoncon} (for the
middle map).  Theorem~\ref{spheretheorem} then amounts to the
assertion that the righthand vertical map is a weak equivalence while
Theorem~\ref{thm:altsphere} is the assertion that the top horizontal
map is a weak equivalence.

\section{Proof of the Sphere Theorem}\label{secpfsphere}

This section contains the proof of Theorem~\ref{thm:altsphere}.  We
fix the almost stable simplicially tensored Waldhausen category $\aC$
and the generating connective class $Q$, letting $\aQ$ and
$\Sigma^{n}\aQ$ be as in Notation~\ref{notn:SigmaQ}.  Just as in
Waldhausen's argument \cite[\S1.7]{WaldhausenKT} we need to introduce
a Waldhausen category of CW complexes built out of cells based on $Q$.

\begin{defn}\label{defn:QCW}
A $Q$-CW complex\index{Q-CW complex@$Q$-CW complex} 
is a filtered object $X$ in $\aC$ 
\[
\cdots \to x_n \to x_{n+1} \to \cdots
\]
indexed on the integers, 
where the arrows $x_n \to x_{n+1}$ are cofibrations 
for all $n$ and such that the following
conditions hold for some $N\gg 0$:
\begin{enumerate}
\item $x_n = *$ for $n \leq -N$,
\item $x_n = x_{n+1}$ for $n \geq N$, and
\item For all $n$, the quotient $x_{n+1} / x_n$ is an object in $\Sigma^{n+1}\aQ$.
\end{enumerate}
We call $x_{N}$ the \noindexterm{underlying object} of $X$ in $\aC$.  Let
$CW_{Q}\aC$ denote the category whose objects are the $Q$-CW complexes
and whose maps are the maps of the underlying objects in $\aC$.  We say that a
$Q$-CW complex $X$ is connective if $x_n = *$ for $n < 0$, and denote
the full subcategory of connective $Q$-CW complexes by
$CW_{Q}\aC_{[0,\infty)}$.  More generally, for $I$ an interval in
$\bZ$, write $CW_{Q}\aC_{I}$ for the full subcategory of $Q$-CW
complexes $X$ with $x_{n}=*$ whenever $n$ is less than the elements of
$I$ and $x_{n}=x_{n+1}$ whenever $n+1$ is greater than the elements of
$I$.
\end{defn}

We define the mapping spectra in $CW_{Q}\aC^{\Gamma}$ and
$CW_{Q}\aC^{S}$ as the mapping spectra of the underlying objects in
$\aC^{\Gamma}$ and $\aC^{S}$, respectively.  For the
Waldhausen category structure, we use the following definition.

\begin{defn}\label{defn:cellular}
A \term{cellular map} of $Q$-CW complexes $X\to Y$ consists of
compatible maps
$x_{n}\to y_{n}$ for all $n$.  A cellular map is a \term{cellular cofibration} when 
each map $x_n \cup_{x_{n-1}} y_{n-1} \to
y_n$ is a cofibration in $\aC$ and the induced map $x_{n}/x_{n-1}\to
y_{n}/y_{n-1}$ is a cofibration in $\Sigma^{n}\aQ$.
\end{defn}

An easy check of the definitions then proves the following proposition.

\begin{prop}\label{prop:CWwald}
The category of $Q$-CW complexes and cellular maps forms a
Waldhausen category with 
cofibrations the cellular cofibrations of
Definition~\ref{defn:cellular} and weak equivalences the weak
equivalences of the underlying objects in $\aC$.  For $I$ an interval
in $\bZ$, the subcategory $CW_{Q}\aC_{I}$ forms a Waldhausen
subcategory (though not a closed one).
\end{prop}

Since $\Sdot[n]\aQ$ is a connective class, we also have the category
of $\Sdot[n]\aQ$-CW complexes in $\Sdot[n]\aC$.  When we restrict to the
subcategories of cellular maps, both $\Sdot[n](CW_{Q}\aC)$ and
$CW_{\Sdot[n]\aQ}\Sdot[n]\aC$ are subcategories of the category of
functors $\Ar[n]\times \bZ \to \aC$ (where the category $\bZ$ is the
ordered set of integers).  An easy check of the definitions then shows
that these categories coincide.  More generally, for $I$ an interval
in $\bZ$, the cellular maps in $\Sdot[n](CW_{Q}\aC_{I})$ and
$(CW_{\Sdot[n]\aQ}\Sdot[n]\aC)_{I}$ are the same subcategory of
functors $\Ar[n]\times I \to \aC$. Expanding to all maps in
$\Sdot[n](CW_{Q}\aC)$ and $CW_{\Sdot[n]\aQ}\Sdot[n]\aC$, and looking at
the cofibrations and weak equivalences, we get the following
proposition.

\begin{prop}\label{prop:snswitch}
The Waldhausen categories $\Sdot[n](CW_{Q}\aC)$ and
$CW_{\Sdot[n]\aQ}\Sdot[n]\aC$ are canonically isomorphic.
For any interval $I$ in $\bZ$, the Waldhausen categories\break 
$\Sdot[n](CW_{Q}\aC_{I})$ and $(CW_{\Sdot[n]\aQ}\Sdot[n]\aC)_{I}$ are
canonically isomorphic.
\end{prop}

Because we need to restrict to cellular maps to obtain a Waldhausen
category, the category $CW_{Q}\aC$ does not fit into our usual
framework of simplicially enriched Waldhausen categories (as the
familiar example of CW complexes in spaces demonstrates).  Instead,
thinking of $Q$-CW complexes as objects of $\aC$ with extra structure,
we assign mapping spectra by looking at the underlying objects.  We
use the following notation.

\begin{notn}\label{notn:cwgamma}
Let $\Sdot[n](CW_{Q}\aC)^{\Gamma}$ 
denote the spectral category whose objects are the objects of
$\Sdot[n](CW_{Q}\aC)$ and whose mapping spectra are the
mapping spectra of the
underlying objects in $\Sdot[n]\aC^{\Gamma}$.
For $I$ an interval in $\bZ$, we define
$\Sdot[n](CW_{Q}\aC_{I})^{\Gamma}$ 
analogously.
\end{notn}

As an alternate take on this notation, we note that under the
canonical isomorphism of Proposition~\ref{prop:snswitch}, 
we get the identification of spectral categories
\[
\Sdot[n](CW_{Q}\aC_{I})^{\Gamma}=
(CW_{\Sdot[n]\aQ}\Sdot[n]\aC_{I})^{\Gamma}
\]
As a first reduction of Theorem~\ref{thm:altsphere}, we have the
following observation.  In it, the ``forgetful functor'' is the
functor that takes a $\Sdot[n]\aQ$-CW complex to its underlying object
of $\Sdot[n]\aC$.

\begin{prop}\label{prop:cwreplace}
For any $n$, the forgetful functor
$\Sdot[n](CW_{Q}\aC)^{\Gamma}\to \Sdot[n]\aC^{\Gamma}$ is a DK-equivalence.
\end{prop}

\begin{proof}
Using the identification of $\Sdot[n](CW_{Q}\aC)$ as
$CW_{\Sdot[n]\aQ}\Sdot[n]\aC$, it suffices to show that for the
arbitrary almost stable simplicially tensored Waldhausen category
$\aC$ and generating connective class $Q$ the forgetful functor
$CW_{Q}\aC^{\Gamma}\to \aC^{\Gamma}$ is a DK-equivalence.  By
definition of the mapping spectra, it is a DK-embedding, and so we
just need to show that every object of $\aC$ is weakly equivalent to
the underlying object of a $Q$-CW complex.  Since $Q$ is generating,
and $*$ is the underlying object of a $Q$-CW complex, it suffices to
show that if $y$ is the underlying object of a $Q$-CW complex $Y$,
then $x=y\cup_{a}b$ is weakly equivalent to the underlying object of a
$Q$-CW complex whenever $a$ is in $\Sigma^{n}\aQ$, $b$ is contractible, and
$a\to b$ is a cofibration.  Using the cofibration sequence of
Proposition~\ref{propfibercofiber} and stability hypothesis~(ii), we
see that we have homotopy fibration sequences
\[
\aC(a,y_{m})\to \aC(a,y_{m+1})\to \aC(a,y_{m+1}/y_{m})
\]
for all $m$.  Since $y_{m+1}/y_{m}$ is in $\Sigma^{m+1}\aQ$, for
$m\geq n$ we have that  $\pi_{0}\aC(a,y_{m+1}/y_{m})=0$ and every map
from $a$ to $y_{m+1}$ lifts up to homotopy to a map $a\to y_{m}$.
Thus, the map $a\to y$ lifts up to homotopy to a map $a\to y_{n}$.
Let $X$ be the $Q$-CW complex 
\[
X=(\dotsb \to y_{n-1}\to y_{n}\cup_{a} b \to y_{n+1}\cup_{a}b\to \dotsb ),
\]
Then the underlying object of $X$ is weakly equivalent to $x$.
\end{proof}

It follows from Proposition~\ref{prop:cwreplace} that the map 
\[
THH(\Sdot[n] (CW_{Q}\aC)^{\Gamma})\to
THH(\Sdot[n] \aC^{\Gamma})
\]
is a weak equivalence. The next step is to compare the subcategory of
connective objects.  The cone and suspension functor on $\aC$ extend
to cone and suspension functors of $Q$-CW complexes in the usual way:
for a $Q$-CW complex $X$, let $CX$ be the $Q$-CW complex with $n$-th
object $x_{n}\cup_{x_{n-1}}Cx_{n-1}$.  The inclusion of $X$ in $CX$ is
a cellular cofibration and $\Sigma X$ is its quotient.  The Additivity
Theorem and Corollary~\ref{coradditivity} generalize to the context of
$THH(\Sdot (CW_{Q}\aC)^{\Gamma})$ to show that the self-map of
$|THH(\Sdot (CW_{Q}\aC)^{\Gamma})|$ induced by $C$ coincides (in the
stable category) with the sum of the identity and the map induced by
$\Sigma $.  Since $C$ induces the trivial map, it follows that
$\Sigma$ induces the map $-\id$, and in particular is a weak
equivalence.  The analogous observations apply to $\aC_{[0,\infty)}$,
showing that suspension induces a weak equivalence on $|THH(\Sdot
(CW_{Q}\aC)^{\Gamma})|$ and on
$|THH(\Sdot(CW_{Q}\aC_{[0,\infty)})^{\Gamma})|$.  Taking the homotopy
colimit of the maps induced by suspension, we see that the inclusions
\begin{gather*}
|THH(\Sdot (CW_{Q}\aC)^{\Gamma})|
\to \hocolim_{\Sigma}|THH(\Sdot (CW_{Q}\aC)^{\Gamma})|\\
|THH(\Sdot (CW_{Q}\aC_{[0,\infty)})^{\Gamma})|
\to \hocolim_{\Sigma}|THH(\Sdot (CW_{Q}\aC_{[0,\infty)})^{\Gamma})|
\end{gather*}
are weak equivalences.  We use this observation in the proof of the
following proposition. 

\begin{prop}\label{prop:connreplace}
The inclusion 
\[
|THH(\Sdot (CW_{Q}\aC_{[0,\infty)})^{\Gamma})|\to
|THH(\Sdot (CW_{Q}\aC)^{\Gamma})|
\]
is a weak equivalence.
\end{prop}

\begin{proof}
By the preceding observations, it suffices to prove that the map
\[
\hocolim_{\Sigma} THH(\Sdot[n] (CW_{Q}\aC_{[0,\infty)})^{\Gamma})
\to \hocolim_{\Sigma} THH(\Sdot[n] (CW_{Q}\aC)^{\Gamma})
\]
is a weak equivalence for each $n$.  Again using the fact that $\aC$
and $Q$ are arbitrary, it suffices to consider the case $n=1$.  Let
$CW^{\Sigma}_{Q}\aC^{\Gamma}$ be the spectrally enriched category
where an object is an ordered 
pair $(X,m)$ where $X$ is a $Q$-CW complex and $m$ is a non-negative
integer; for mapping spectra, we let 
\[
CW^{\Sigma}_{Q}\aC^{\Gamma}((X,m),(Y,n))=\lcolim_{k\geq \max(m,n)}
\aC^{\Gamma}(\Sigma^{k-m}X,\Sigma^{k-n}Y).
\]
(Composition is induced levelwise in the colimit system after taking
$k$ large enough.)  Let $CW^{\Sigma}_{Q}\aC_{[0,\infty)}$ be the full
subcategory of $CW^{\Sigma}_{Q}\aC^{\Gamma}$ consisting of the objects
$(X,m)$ with $X$ connective.  By Proposition~\ref{propthhcolim}, the
canonical maps 
\begin{gather*}
\hocolim_{\Sigma} THH(CW_{Q}\aC^{\Gamma})
\to THH(CW^{\Sigma}_{Q}\aC^{\Gamma})\\
\hocolim_{\Sigma} THH(CW_{Q}\aC_{[0,\infty)}^{\Gamma})
\to THH(CW^{\Sigma}_{Q}\aC_{[0,\infty)}^{\Gamma})
\end{gather*}
are weak equivalences.  The inclusion of
$CW^{\Sigma}_{Q}\aC_{[0,\infty)}^{\Gamma}$ in
$CW^{\Sigma}_{Q}\aC^{\Gamma}$ is a DK-equivalence, and so also induces
a weak equivalence on $THH$. 
\end{proof}

The previous two propositions show that the map
\[
|THH(\Sdot(CW_{Q}\aC_{[0,\infty)})^{\Gamma})|\to
|THH(\Sdot\aC^{\Gamma})|
\]
is a weak equivalence, reducing the proof of
Theorem~\ref{thm:altsphere} to showing that the map 
\[
|THH(\Sdot\aQ^{\Gamma})|\to |THH(\Sdot(CW_{Q}\aC_{[0,\infty)})^{\Gamma})|
\]
is a weak equivalence.  This is an easy consequence of the following lemma.

\begin{lem}\label{lem:spheremain}
For every $n\geq 1$, the inclusion of $CW_{Q}\aC_{[0,n-1]}$ in
$CW_{Q}\aC_{[0,n]}$ induces a weak equivalence 
\[
|THH(\Sdot (CW_{Q}\aC_{[0,n-1]})^{\Gamma})|\to 
|THH(\Sdot (CW_{Q}\aC_{[0,n]})^{\Gamma})|.
\]
\end{lem}

\begin{proof}[Proof of Theorem~\ref{thm:altsphere} from
Lemma~\ref{lem:spheremain}]
The lemma implies that the maps in the homotopy colimit system
\[
\hocolim_{n}|THH(\Sdot (CW_{Q}\aC_{[0,n]})^{\Gamma})|
\]
are all weak equivalences.  By Proposition~\ref{propthhcolim}, we see
that the canonical 
map from the homotopy colimit to
$|THH(\Sdot(CW_{Q}\aC_{[0,\infty)})^{\Gamma})|$  
is a weak equivalence.  It follows that the map
\[
|THH(\Sdot\aQ^{\Gamma})|=
|THH(\Sdot(CW_{Q}\aC_{[0,0]})^{\Gamma})|
\to |THH(\Sdot(CW_{Q}\aC)^{\Gamma})|
\]
is a weak equivalence.  Composing with the weak equivalence
\[
|THH(\Sdot (CW_{Q}\aC)^{\Gamma})|\to
|THH(\Sdot \aC^{\Gamma})|
\]
above and applying $\Omega$, we see that the map
$WTHH^{\Gamma}(\aQ)\to WTHH^{\Gamma}(\aC)$ is a weak equivalence.
\end{proof}

The remainder of the section is devoted to the proof of
Lemma~\ref{lem:spheremain}.  The argument is somewhat roundabout,
requiring the introduction of the spectral categories
$\Sdot[k](CW_{Q}\aC_{[0,n]})^{S}$, defined analogously to
$\Sdot[k](CW_{Q}\aC_{[0,n]})^{\Gamma}$ in Notation~\ref{notn:cwgamma}
but using the non-connective enrichment.  The proof of
Proposition~\ref{prop:cwreplace} equally well shows that the forgetful
functor $\Sdot[k](CW_{Q}\aC)^{S}\to \Sdot[k]\aC^{S}$ is a
DK-equivalence.  These non-connective enrichments are easier to
understand because  Proposition~\ref{propfibercofiber} implies that
when we DK-embed $\Sdot[k]\aC^{S}$ in a pretriangulated
spectral category, the DK-embedding 
takes cofiber sequences to distinguished triangles in the derived
category.   As a consequence, Theorem~\ref{thmthick} tells
us that the maps 
\[
THH(\Sdot[k]\aQ^{S})\to
THH(\Sdot[k](CW_{Q}\aC_{[0,n-1]})^{S}) \to
THH(\Sdot[k](CW_{Q}\aC_{[0,n]})^{S})
\]
are weak equivalences.
Looking at the diagram
\[
\xymatrix{%
|THH(\Sdot(CW_{Q}\aC_{[0,n-1]})^{\Gamma})|\ar[r]\ar[d]
&|THH(\Sdot(CW_{Q}\aC_{[0,n]})^{\Gamma})|\ar[d]\\
|THH(\Sdot(CW_{Q}\aC_{[0,n-1]})^{S})|\ar[r]^-{\sim}
&|THH(\Sdot(CW_{Q}\aC_{[0,n]})^{S})|,
}
\]
we assume by induction on $n$ that the lefthand map is a weak
equivalence, the base case being the already known case of 
$\Sdot CW_{Q}\aC_{[0,0]}=\Sdot\aQ$.  We then prove that the top map is a weak
equivalence by showing that the righthand map is a weak equivalence.

To save space and eliminate unnecessary symbols, we will now write
$\aC^{n}_{k}$ for $\Sdot[k](CW_{Q}\aC_{[0,n]})$ or
equivalently, $CW_{\Sdot[k]\aQ}\Sdot[k]\aC_{[0,n]}$, and 
$\aC\Gamma^{n}_{k}$ and $\aC S^{n}_{k}$ for the connective and
non-connective spectral enrichments, respectively.
Let $\aE\Gamma^{n}\subdot$ denote the simplicial spectral category where the
objects of $\aE\Gamma^{n}_{k}$ are the objects of
$\aC^{n}_{k}$ 
and for objects $X$ and $Y$, the mapping spectrum is the fiber product
\[
\aE\Gamma^{n}_{k}(X,Y) = \Sdot[k]\aC^{\Gamma}(x_{n-1},y_{n-1})
\times_{\Sdot[k]\aC^{\Gamma}(x_{n-1},y_{n})}
\Sdot[k]\aC^{\Gamma}(x_{n},y_{n})
\]
(which is a homotopy pullback because $x_{n-1}\to x_{n}$ is a
cofibration). 
We have a canonical simplicial spectral functor
$\aE\Gamma^{n}\subdot\to \aC\Gamma^{n}\subdot$
sending $X$ in $\aE\Gamma^{n}_{k}$ to $X$ viewed as an object of
$\aC\Gamma^{n}_{k}$ and using projection on the mapping
spectra. 
We also have canonical simplicial spectral functors 
\[
\aE\Gamma^{n}\subdot\to \aC\Gamma^{n-1}\subdot
\qquad \text{and}\qquad
\aE\Gamma^{n} \subdot\to \Sdot \Sigma^{n}\aQ^{\Gamma}
\]
sending $X$ to its $(n-1)$-skeleton $X_{n-1}$ and to $x_{n}/x_{n-1}$, respectively, and performing
the corresponding maps on mapping spectra.  Using these maps, we can identify
$\aE\Gamma^{n}\subdot$ as the spectral category of extension sequences
$[\xymatrix@-1pc{X_{n-1}\ar@{ >->}[r]&X\ar@{->>}[r]&x_{n}/x_{n-1}}]$ in
$\aC^{n}\subdot$. Although the categories $\aC\Gamma^{n}\subdot$ and
$\aE\Gamma^{n}$ do not exactly fit into the framework of
Section~\ref{secadd}, McCarthy's argument for the Additivity
Theorem works quite generally and formally essentially using little
more than the fact that the mapping spectra are functorial in the maps
in $\Sdot$; the Additivity Theorem generalizes to the current context,
and the argument following 
Corollary~\ref{coradditivity} shows that the maps described above
induce a weak equivalence 
\[
|THH(\aE\Gamma^{n}\subdot)|\overto{\sim}
|THH(\aC\Gamma^{n-1}\subdot)| \times 
|THH(\Sdot\Sigma^{n}\aQ^{\Gamma})|.
\]
We have an analogous simplicial spectral category $\aE S^{n}\subdot$
with the analogous weak equivalence.  The induction hypothesis
and the weak equivalences above then imply the following proposition.

\begin{prop}
The functor $\aE\Gamma^{n}\subdot\to \aE S^{n}\subdot$
induces a weak equivalence 
\[
|THH(\aE\Gamma^{n}\subdot)|\overto{\sim}
|THH(\aE S^{n}\subdot)|.
\]
\end{prop}

Let $(\aC^{n}_{k})^{w}$ denote the objects of $\aC^{n}_{k}$ that are
weakly equivalent to the zero object~$*$, and write
$(\aC\Gamma^{n}_{k})^{w}$, $(\aC S^{n}_{k})^{w}$,
$(\aE\Gamma^{n}_{k})^{w}$, $(\aE S^{n}_{k})^{w}$ for the various
spectral enrichments on this category, the full spectral subcategories
of $\aC\Gamma^{n}_{k}$, $\aC S^{n}_{k}$, $\aE\Gamma^{n}_{k}$, $\aE
S^{n}_{k}$, respectively.  The mapping spectra in
$(\aC\Gamma^{n}_{k})^{w}$ and  $(\aC S^{n}_{k})^{w}$ are all weakly
contractible so $THH$ is also weakly contractible,
\[
THH((\aC\Gamma^{n}_{k})^{w})\htp
THH((\aC S^{n}_{k})^{w})\htp *.
\]
For the $\aE$ categories, we have the following proposition.

\begin{prop}
The canonical spectral functor $(\aE\Gamma^{n}_{k})^{w}\to(\aE
S^{n}_{k})^{w}$ is a  DK-equivalence.
\end{prop}

\begin{proof}
Since the categories have the same object set, it suffices to show
that the map is a DK-embedding, and for this it suffices to show that
the mapping spectra in $(\aE S^{n}_{k})^{w}$ are connective.
We note that for $X$ in $(\aC^{n}_{k})^{w}$, stability hypothesis~(ii)
implies that $x_{n-1}$ is an object of
$\Sdot[k]\Sigma^{n-1}\aQ$ since $x_{n}$ is contractible and
$x_{n}/x_{n-1}$ is an object of $\Sdot[k]\Sigma^{n}\aQ$.  Now given $X$
and $Y$ in $(\aC^{n}_{k})^{w}$, the projection map 
\[
\aE S^{n}_{k}(X,Y) = \Sdot[k]\aC^{S}(x_{n-1},y_{n-1})
\times_{\Sdot[k]\aC^{S}(x_{n-1},y_{n})}
\Sdot[k]\aC^{S}(x_{n},y_{n})
\to \Sdot[k]\aC^{S}(x_{n-1},y_{n-1})
\]
is a weak equivalence.  In particular, $\aE S^{n}_{k}(X,Y)$ is connective.
\end{proof}

We denote by $CTHH(\aC\Gamma^{n}_{k},w)$ the homotopy cofiber of 
the inclusion
\[
THH((\aC\Gamma^{n}_{k})^{w})\to THH(\aC\Gamma^{n}_{k}),
\]
and analogously for $CTHH(\aC S^{n}_{k},w)$,
$CTHH(\aE\Gamma^{n}_{k},w)$, and $CTHH(\aE S^{n}_{k},w)$.  In this
notation, the two previous propositions then imply the following
proposition.

\begin{prop}
The map $|CTHH(\aE\Gamma^{n}\subdot,w)|\to 
|CTHH(\aE S^{n}\subdot,w)|$ is a
weak equivalence.
\end{prop}

Since the inclusions 
\begin{gather*}
THH(\aC\Gamma^{n}_{k})\to
CTHH(\aC\Gamma^{n}_{k},w)
\\
THH(\aC S^{n}_{k})\to 
CTHH(\aC S^{n}_{k},w)
\end{gather*}
are weak equivalences, the following lemma
when combined with the previous proposition 
then completes the proof of
Lemma~\ref{lem:spheremain}.

\begin{lem}\label{lem:spherefinal}
For all $k$, the maps 
\begin{gather*}
CTHH(\aE\Gamma^{n}_{k},w)\to CTHH(\aC\Gamma^{n}_{k},w)\\
CTHH(\aE S^{n}_{k},w)\to CTHH(\aC S^{n}_{k},w)
\end{gather*}
are weak equivalences.
\end{lem}

We prove the case for the connective enrichment in detail, the case
for the non-connective enrichment being similar (but slightly easier).
The statement is analogous to the Localization Theorem~7.2 of
\cite{BlumbergMandellTHHLoc} (reviewed in Chapter~\ref{tw-thhmv} as
Theorem~\ref{thmgentwo}) using the Dennis-Waldhausen Morita Argument
(Section~\ref{secdwm}) except for the fact that the subcategories
above are not pretriangulated.  The following proof goes roughly along
the same lines as well.

For this argument $k$ is both fixed, and so replacing
$\aC\Gamma^{\infty}_{k}$ and $\aE\Gamma^{\infty}_{k}$ by weakly equivalent
spectral categories if necessary, we can assume without loss of
generality that they are pointwise cofibrant and their subcategories 
$\aC\Gamma^{n}_{k}$, $\aE\Gamma^{n}_{k}$, $(\aC\Gamma^{n}_{k})^{w}$,
$(\aE\Gamma^{n}_{k})^{w}$ are pointwise cofibrant.

Define the $\aE\Gamma^{n}_{k}$-bimodule $\fL_{\aE}$ and 
$\aC\Gamma^{n}_{k}$-bimodule $\fL_{\aC}$ by
\begin{gather*}
\fL_{\aE}(X,Y) = 
B(\aE\Gamma^{n}_{k}(-,Y);(\aE\Gamma^{n}_{k})^{w};\aE\Gamma^{n}_{k}(X,-))
\\
\fL_{\aC}(X,Y) = 
B(\aC\Gamma^{n}_{k}(-,Y);(\aC\Gamma^{n}_{k})^{w};\aC\Gamma^{n}_{k}(X,-)),
\end{gather*}
where $B$ denotes the two-sided bar construction
(Definition~\ref{defn:tsbc}). 
We then have maps of $\aE\Gamma^{n}_{k}$- and
$\aC\Gamma^{n}_{k}$-bimodules
\[
\fL_{\aE}\to \aE\Gamma^{n}_{k} \qquad \text{and}\qquad 
\fL_{\aC}\to \aC\Gamma^{n}_{k};
\]
we let $\fM_{\aE}$ and $\fM_{\aC}$ be the homotopy cofibers.  Then
the Dennis-Waldhausen Morita
Argument and specifically Theorem~\ref{thmdwmapp} give us weak equivalences 
\begin{align*}
THH(\aE\Gamma^{n}_{k};\fL_{\aE}) &\htp THH((\aE\Gamma^{n}_{k})^{w})
&THH(\aE\Gamma^{n}_{k};\fM_{\aE}) &\htp CTHH(\aE\Gamma^{n}_{k},w)\\
THH(\aC\Gamma^{n}_{k};\fL_{\aC}) &\htp THH((\aC\Gamma^{n}_{k})^{w})
&THH(\aC\Gamma^{n}_{k};\fM_{\aC}) &\htp CTHH(\aC\Gamma^{n}_{k},w),
\end{align*}
and we can identify the map in Lemma~\ref{lem:spherefinal} as the map
\begin{equation}\label{eq:spherefinal}
THH(\aE\Gamma^{n}_{k};\fM_{\aE})\to 
THH(\aC\Gamma^{n}_{k};\fM_{\aC}).
\end{equation}

As the mapping spectra in $(\aC\Gamma^{n}_{k})^{w}$ are weakly
contractible, 
the spectra $\fL_{\aC}(X,Y)$ are weakly contractible for all $X,Y$,
and it follows that the map of $\aC\Gamma^{n}_{k}$-bimodules
$\aC\Gamma^{n}_{k}\to 
\fM_{\aC}$ is a weak equivalence.  We next move towards understanding the
$\aE\Gamma^{n}_{k}$-bimodules $\fL_{\aE}$.   We write $u$ for the
canonical functor $\aE\Gamma^{n}_{k}\to \aC\Gamma$ and also its restriction
$(\aE\Gamma^{n}_{k})^{w}\to (\aC\Gamma)^{w}$.  We then have a
commutative diagram of $\aE\Gamma^{n}_{k}$-bimodules
\[
\xymatrix{%
\fL_{\aE}\ar[r]\ar[d]&u^{*}\fL_{\aC}\ar[d]\\
\aE\Gamma^{n}_{k}\ar[r]&u^{*}\aC\Gamma^{n}_{k},
}
\]
Letting $\aF$ be the homotopy pullback of the deleted diagram
\[
\aE\Gamma^{n}_{k}\to u^{*}\aC\Gamma^{n}_{k}\from
u^{*}\fL_{\aC},
\]
we get a map of $\aE\Gamma^{n}_{k}$-bimodules $\fL_{\aE}\to \aF$.

\begin{prop}
The map of $\aE\Gamma^{n}_{k}$-bimodules $\fL_{\aE}\to \aF$ is a weak
equivalence. 
\end{prop}

\begin{proof}
Fix $X$ and $Y$ objects in $\aE\Gamma^{n}_{k}$; we need to show that
the map $\fL_{\aE}(X,Y)\to \aF(X,Y)$ is a weak equivalence.  Consider the cofibration sequence 
\[
y_{n-1}\to y_{n} \to y_{n}/y_{n-1}\to \Sigma y_{n-1}
\]
obtained using a homotopy inverse weak equivalence to the collapse
weak equivalence $y_{n}\cup_{y_{n-1}}Cy_{n-1}\to y_{n}/y_{n-1}$.
By definition, $y_{n}/y_{n-1}$ is in $\Sdot[k]\Sigma^{n}\aQ$, and since
$n\geq 1$, there exists an object $p$ in  $\Sdot[k]\Sigma^{n-1}\aQ$
such that $\Sigma p$ is weakly equivalent to $y_{n}/y_{n-1}$.  Then
applying stability hypothesis~(ii), we obtain from the cofibration
sequence above a (homotopy class of) map $p\to y_{n-1}$ and a null
homotopy $Cp\to y_{n}$ such that the induced map $\Sigma p\to
y_{n}/y_{n-1}$ is homotopic to the chosen weak equivalence.  Regarding
$Cp$ as an object of $\aE^{n}_{k}$, it is an object of
$(\aE^{n}_{k})^{w}$ and we have constructed a cellular map $Cp\to Y$.
Consider the following commutative square.
\[
\xymatrix{%
\fL_{\aE}(X,Y)\ar[d]&\fL_{\aE}(X,Cp)\ar[l]_{a}\ar[d]^{c}\\
\aF(X,Y)&\aF(X,Cp)\ar[l]^{b}
}
\]
We complete the proof by arguing that the maps $a$, $b$, and $c$ are
weak equivalences. 

To analyze the map $a$, consider an object $Z$ in
$(\aE\Gamma^{n}_{k})^{w}$.  Since $z_{n}$ is weakly equivalent to $*$
in $\aC$, $\aE\Gamma^{n}_{k}(Z,Y)$ is weakly equivalent to the
homotopy fiber of the map $\Sdot[k]\aC^{\Gamma}(z_{n-1},y_{n-1})$ to
$\Sdot[k]\aC^{\Gamma}(z_{n-1},y_{n})$.  We can use the cofibration
sequence of Proposition~\ref{propfibercofiber} to understand this
homotopy fiber: We have that $\Sdot[k]\aC^{\Gamma}(z_{n-1},
y_{n}/y_{n-1})$ is connected since $z_{n-1}$ is an object of
$\Sdot[k]\Sigma^{n-1}\aQ$ and $y_{n}/y_{n-1}$ is an object of
$\Sdot[k]\Sigma^{n}\aQ$.  It follows that $\aE\Gamma^{n}_{k}(Z,Y)$ is
weakly equivalent to $\Omega \Sdot[k]\aC^{\Gamma}(z_{n-1},
y_{n}/y_{n-1})$.  The same observations apply to $Cp$.  Since by
construction the map $\Sigma p=Cp/p\to y_{n}/y_{n-1}$ is a weak
equivalence, we see by naturality that the map
$\aE\Gamma^{n}_{k}(Z,Cp)\to \aE\Gamma^{n}_{k}(Z,Y)$ is a weak
equivalence.  Since this holds for any $Z$ in
$(\aE\Gamma^{n}_{k})^{w}$, unwinding the definition of $\fL_{\aE}$, we
see that $a$ is a weak equivalence.

For the map $b$, we note that $\aF(X,Y)$ being the homotopy fiber of
the map $\aE\Gamma^{n}_{k}(X,Y)$ to
$\aC\Gamma(X,Y)=\Sdot[k]\aC^{\Gamma}(x_{n},y_{n})$, it is naturally
weakly equivalent to the homotopy fiber of the map
$\Sdot[k]\aC^{\Gamma}(x_{n-1},y_{n-1})$ to
$\Sdot[k]\aC^{\Gamma}(x_{n-1},y_{n})$.  As in the previous case, we
can identify this up to weak equivalence as $\Omega
\Sdot[k]\aC^{\Gamma}(x_{n-1}, y_{n}/y_{n-1})$ since
$\Sdot[k]\aC^{\Gamma}(x_{n-1}, y_{n}/y_{n-1})$ is connected (which can
be proved by induction up the skeletal filtration of $X$ using
Proposition~\ref{propfibercofiber}).  Again, since the map $\Sigma
p=Cp/p\to y_{n}/y_{n-1}$ is a weak equivalence, we see that $b$ is a
weak equivalence.

For the map $c$, since $Cp$ is in $(\aE\Gamma^{n}_{k})^{w}$, the
Two-Sided Bar Lemma~\ref{lemtwobar}  implies that the
natural map $\fL_{\aE}(X,Cp)\to \aE\Gamma^{n}_{k}(X,Cp)$ is a weak
equivalence.  Since $Cp$ is weakly equivalent to $*$ in $\aC$,
$\aC\Gamma(X,Cp)$ is weakly contractible and we see that $b$ is a weak
equivalence.
\end{proof}

The previous proposition lets us understand $\fM_{\aE}$.

\begin{prop}
The map of $\aE\Gamma^{n}_{k}$-bimodules $\fM_{\aE}\to u^{*}\fM_{\aC}$
is a weak equivalence.
\end{prop}

\begin{proof}
Since homotopy fiber squares in spectra are homotopy cocartesian, the
canonical map from the homotopy cofiber of $\aF\to \aE\Gamma^{n}_{k}$
to the homotopy cofiber of $u^{*}\fL_{\aC}\to u^{*}\aC\Gamma^{n}_{k}$
is a weak equivalence.
\end{proof}

We now return to the map~\eqref{eq:spherefinal}.  We see from the
previous proposition that we are in the situation
where Theorem~\ref{thmcompcrit} applies.
Thus, to see that the map
\[
THH(\aE\Gamma^{n}_{k};\fM_{\aE})\to 
THH(\aC\Gamma^{n}_{k};\fM_{\aC})
\]
is a weak equivalence, we just need to check that the map
\[
B(\aC\Gamma^{n}_{k}(-,Y);\aE\Gamma^{n}_{k};\fM_{\aC}(X,-))
\to \aC\Gamma^{n}_{k}(X,Y)
\]
is a weak equivalence for all $X$, $Y$ in $\aC\Gamma^{n}_{k}$, or
equivalently in this case, for all $X$, $Y$ in $\aE\Gamma^{n}_{k}$.
Since the Two-Sided Bar 
Lemma~\ref{lemtwobar} shows that the map
\[
B(\aC\Gamma^{n}_{k}(-,Y);\aE\Gamma^{n}_{k};\aE\Gamma^{n}_{k}(X,-))
\to \aC\Gamma^{n}_{k}(X,Y)
\]
is a weak equivalence and $\fM_{\aC}(X,-)\htp \fM_{\aE}(X,-)$ is
the homotopy cofiber of $\fL_{\aE}(X,-)\to \aE\Gamma^{n}_{k}(X,-)$, it
suffices to show that  
\[
\aG(X,Y)=B(\aC\Gamma^{n}_{k}(-,Y);\aE\Gamma^{n}_{k};\fL_{\aE}(X,-))
\]
is weakly contractible.  But we have 
\begin{multline*}
\aG(X,Y)=
B(\aC\Gamma^{n}_{k}(-,Y);\aE\Gamma^{n}_{k};
B(\aE\Gamma^{n}_{k}(-,-);(\aE\Gamma^{n}_{k})^{w};\aE\Gamma^{n}_{k}(X,-))
)\\
\htp 
B(\aC\Gamma^{n}_{k}(-,Y);(\aE\Gamma^{n}_{k})^{w};\aE\Gamma^{n}_{k}(X,-)).
\end{multline*}
Since $\aC\Gamma^{n}_{k}(Z,Y)$ is weakly contractible for any $Z$ in
$(\aE\Gamma^{n}_{k})^{w}$, it follows that $\aG(X,Y)$ is weakly
contractible.  This completes the proof that~\eqref{eq:spherefinal} is
a weak equivalence and hence the proof of Lemma~\ref{lem:spherefinal},
which in turn completes the proof of Lemma~\ref{lem:spheremain}.

%%%%%%%%%%%%%%%%%%%%%%%%%%%%%%%%%%%%%%%%
%%%%%%%%%%%%%%%%%%%%%%%%%%%%%%%%%%%%%%%%

% tp-set-source: tw-whole.tex
% tex-command: latex
% ultex-add-defs: list

%
% Copyright (C) 2008-12  Andrew J. Blumberg and Michael A. Mandell
%

\chapter{Localization sequences for $THH$ and $TC$}
\label{tw-loc}

In \cite{HM3}, Hesselholt and Madsen introduced a localization
sequence for $THH$ and $TC$ in the context of discrete valuation
rings, producing cofiber sequences
\begin{gather*}
THH(k)\to THH(R)\to THH(R|F)\to \Sigma THH(k)\\
TC(k)\to TC(R)\to TC(R|F)\to \Sigma TC(k),
\end{gather*}
where $R$ denotes a discrete valuation ring, $k$ its residue field,
and $F$ its field of fractions.  Here $THH(R|F)$ and $TC(R|F)$ denote
the $THH$ and $TC$ of a relative theory they construct.  Although
their constructions are slightly different than ours, we prove in
Section~\ref{sechm3} that this sequence arises as the Localization
Theorem~\ref{thmloc} for the connective spectral enrichment
\[
WTHH^{\Gamma}(\aA^{v})\to WTHH^{\Gamma}(\aA)\to
WTHH^{\Gamma}(\aA|v)\to \Sigma WTHH^{\Gamma}(\aA^{v})
\]
where $\aA$ denotes either the category of perfect simplicial modules
over the ring $R$ or the category of finite cell EKMM $HR$-modules
(for the Eilenberg-Mac Lane spectrum $HR$) and $v\aA$ denotes the
subcategory of maps that induce an isomorphism on homotopy groups
after inverting the action of the uniformizer (or, equivalently,
tensoring over $R$ with $F$).  This is in contrast to the localization sequence
obtained from the non-connective spectral enrichment
\[
WTHH(\aA^{v})\to WTHH(\aA)\to WTHH(\aA|v)\to \Sigma WTHH(\aA^{v})
\]
which leads to the localization sequences
\begin{gather*}
THH(R \on k)\to THH(R)\to THH(F)\to \Sigma THH(R \on k)\\
TC(R \on k)\to TC(R)\to TC(F)\to \Sigma TC(R \on k)
\end{gather*}
where $THH(R \on k)$ and $TC(R \on k)$ are as in Theorem~1.1 of
\cite{BlumbergMandellTHHLoc}. 

Hesselholt and Ausoni \cite{AusoniTHH, AusoniK} conjectured that the
above localization sequences generalize from the ``chromatic level 0''
case to ``chromatic level 1'' and specifically that there should be
analogous cofiber sequences
\begin{gather*}
THH(\bZ)\to THH(ku)\to THH(ku|KU)\to \Sigma THH(\bZ)\\
TC(\bZ)\to TC(ku)\to TC(ku|KU)\to \Sigma TC(\bZ)
\end{gather*}
(as well as $p$-local and $p$-complete variants; see
Theorem~\ref{thmkuloc} below).
Here $ku$ denotes complex connective (topological) $K$-theory, $KU$
denotes complex periodic $K$-theory.  In Sections~\ref{secar}
and~\ref{secdevissage}, we prove these sequences arise again from 
the Localization Theorem for the connective enrichment
\[
WTHH^{\Gamma}(\aC^{v})\to WTHH^{\Gamma}(\aC)\to
WTHH^{\Gamma}(\aC|v)\to \Sigma WTHH^{\Gamma}(\aC^{v})
\]
where $\aC$ is the category of finite cell EKMM $ku$-modules and
$v\aC$ the maps that induce isomorphisms on homotopy groups after
inverting the action of the Bott element.  Our argument is general
enough to also produce the localization sequences   
\begin{gather*}
THH(\bW\bF_{p^{n}}[\![u_{1},\ldots, u_{n-1}]\!])\to THH(BP_{n})\to THH(BP_{n}|E_{n})\to \Sigma\cdots %THH(\bW\bF_{p^{n}}[\![u_{1},\ldots, u_{n-1}]\!])
\\
TC(\bW\bF_{p^{n}}[\![u_{1},\ldots, u_{n-1}]\!])\to TC(BP_{n})\to
TC(BP_{n}|E_{n})\to \Sigma\cdots %TC(\bW\bF_{p^{n}}[\![u_{1},\ldots, u_{n-1}]\!])
\end{gather*}
for all $n$
discussed in the introduction of this paper relating the $THH$ and
$TC$ of the Eilenberg-Mac Lane spectra on the Witt rings to the $THH$
and $TC$ of the connective cover $B_{n}$ of the Lubin-Tate spectrum
$E_{n}$ and the corresponding relative construction.

The chapter is organized as follows.  Section~\ref{sechm3} compares
our constructions $WTHH^{\Gamma}$ to the analogous construction of
Hesselholt-Madsen~\cite{HM3}.  Section~\ref{secar} states the main
theorem on localization sequences for $THH(ku)$ and reduces the proof
to a d\'evissage theorem, Theorem~\ref{thm:dev};
Section~\ref{secdevissage} then proves Theorem~\ref{thm:dev}.

\section{The localization sequence for $THH$ of a discrete valuation ring}
\label{sechm3}

In this section, we compare the construction of $THH$ we use here with
the construction used by Hesselholt-Madsen in~\cite{HM3} to prove the
localization sequences in $THH$ and $TC$ for discrete valuation
rings.  The main theorem of this section is then the following.

\begin{thm}\label{thmcomphm3}
Let $R$ be a discrete valuation ring, $k$ its quotient field and $F$
its field of fractions.  Let $\aA$ denote the category of perfect
simplicial $R$-algebras and let $v\aA$ denote the subcategory of those
maps which induce isomorphisms on homotopy groups after inverting a
uniformizer (i.e., after tensoring with $F$).  Then the cofibration
sequence
\[
WTHH^{\Gamma}(\aA^{v})\to WTHH^{\Gamma}(\aA)\to
WTHH^{\Gamma}(\aA|v)\to \Sigma WTHH^{\Gamma}(\aA^{v})
\]
of Theorem~\ref{thmloc} on $THH$ induces on $TC$ a cofbration sequence
naturally weakly equivalent to the cofibration sequence of
\cite[1.5.7]{HM3}, compatibly with the cyclotomic trace.
\end{thm}

Assuming Theorem~\ref{thmwkexact} from Chapter~\ref{tw-gen}, we also
sketch a proof of the following theorem for the EKMM $S$-module models.

\begin{thm}\label{thmcompekmm}
Let $R$ be a discrete valuation ring, $k$ its quotient field and $F$
its field of fractions.  Let $\aA$ denote the category of finite cell
EKMM $HR$-modules and let $v\aA$ denote the subcategory of those
maps which induce isomorphisms on homotopy groups after inverting a
uniformizer (i.e., after tensoring over $R$ with $F$).  Then the cofibration
sequence
\[
WTHH^{\Gamma}(\aA^{v})\to WTHH^{\Gamma}(\aA)\to
WTHH^{\Gamma}(\aA|v)\to \Sigma WTHH^{\Gamma}(\aA^{v})
\]
of Theorem~\ref{thmloc} on $THH$ induces on $TC$ a cofbration sequence
naturally weakly equivalent to the cofibration sequence of
\cite[1.5.7]{HM3}, compatibly with the cyclotomic trace.
\end{thm}

We begin with a quick review of the construction used by
Hesselholt-Madsen~\cite{HM3}.  Let $\aC_{0}$ denote the category of
perfect complexes $R$-modules, i.e., the category of
bounded chain complexes of finitely generated projective $R$-modules
and let $v\aC_{0}$ denote the subcategory of maps that induce isomorphisms
on homology after inverting a uniformizer (or equivalently,
tensoring over $R$ with $F$).  Regarding $\aC_{0}$ and $\aC_{0}^{v}$ as exact
categories, we get connective spectral enrichments $\aC_{0}^{\Gamma}$ and
$(\aC_{0}^{v})^{\Gamma}$.  Hesselholt-Madsen~\cite[p.~27]{HM3} then
produce weak equivalences
\[
THH(w\subdot\Sdot \aC_{0}^{\Gamma})\simeq THH(R),
\qquad 
THH(w\subdot\Sdot (\aC_{0}^{v})^{\Gamma})\simeq THH(k)
\]
and a homotopy cartesian square
\[
\xymatrix@C-1pc{%
THH(w\subdot\Sdot (\aC_{0}^{v})^{\Gamma})\ar[r]\ar[d]&
THH(v\subdot\Sdot (\aC_{0}^{v})^{\Gamma})\ar[d]\\
THH(w\subdot\Sdot \aC_{0}^{\Gamma})\ar[r]&
THH(v\subdot\Sdot \aC_{0}^{\Gamma})
}
\]
with the upper left hand entry (canonically) contractible.  Their
$THH$ cofibration sequence is then
\[
THH(w\subdot\Sdot (\aC_{0}^{v})^{\Gamma})\to
THH(w\subdot\Sdot \aC_{0}^{\Gamma})\to 
THH(v\subdot\Sdot \aC_{0}^{\Gamma})\to
\Sigma THH(w\subdot\Sdot (\aC_{0}^{v})^{\Gamma}).
\]
Since the simplicial categories $\Sdot[n]\aC_{0}$ and $\Sdot[n]\aC_{0}^{v}$
are discrete, the canonical inclusions
\[
w_{m}\Sdot[n]\aC_{0}^{v}\to w_{m}^{M}\Sdot[n]\aC_{0}^{v}
w_{m}\Sdot[n]\aC_{0}\to w_{m}^{M}\Sdot[n]\aC_{0}
v_{m}\Sdot[n]\aC_{0}\to v_{m}^{M}\Sdot[n]\aC_{0}
\]
are isomorphisms, and so we can identify the $THH$ cofibration
sequence of \cite{HM3} as the cofibration sequence
\[
WTHH^{\Gamma}(\aC_{0}^{v})\to WTHH^{\Gamma}(\aC_{0})\to
WTHH^{\Gamma}(\aC_{0}|v)\to \Sigma WTHH^{\Gamma}(\aC_{0}^{v}).
\]
The proof of Theorem~\ref{thmcomphm3} then consists of essentially two
parts: First reconciling the use of the category of complexes of
$R$-modules ($\aC_{0}$) with the use of the category of simplicial
$R$-modules ($\aA_{0}$), and second in the construction of connective
enrichment, reconciling the use mapping spaces ($\aA^{\Gamma}(x,y)$)
and mapping sets ($\aA_{0}^{\Gamma}(x,y)$).

To treat the case of $w\subdot\Sdot \aC_{0}^{v}$, $w\subdot\Sdot
\aC_{0}$, and $v\subdot\Sdot\aC_{0}$ on equal footing, we will work in
the following context.  Let $\abA$ be an abelian category, let $\aB_{0}$
be a full subcategory of the category of bounded below (in the
homological grading) complexes of $\abA$-modules, and let $v\aB_{0}$ be a
subcategory of $\aB_{0}$ containing all the quasi-isomorphisms,
satisfying the Gluing Axiom for the degreewise split monomorphisms,
and satisfying the two-out-of-three property.
We also assume that $\aB_{0}$ contains $0$, is closed 
under suspension and is closed under quotients and extensions by
degreewise split monomorphisms, i.e., if 
\[
0\to a\to b\to c\to 0
\]
is a short exact sequences of chain complexes in $\abA$ with $a\to b$
degreewise split, if $a$ is $\aB_{0}$ and either $b$ or $c$ is in
$\aB_{0}$ then so is the other.  Let $\aA_{0}$ be the subcategory of
strictly connective complexes in $\aB_{0}$.  Then the Dold-Kan
correspondence allows us to view $\aA_{0}$ as a full subcategory of
the category of simplicial objects in $\abA$, extending it to a
simplicially enriched category $\aA$.  We regard $\aB_{0}$ as a
Waldhausen category with cofibrations the degreewise split
cofibrations and weak equivalences the quasi-isomorphisms; then
$\aA_{0}$ is a Waldhausen subcategory (though not closed) and $\aA$ is
a simplicially tensored Waldhausen category.  We prove the following
lemmas.

\begin{lem}\label{lemhm31}
Under the hypotheses of the preceding paragraph, 
the inclusion of $THH(v\subdot\Sdot \aA^{\Gamma}_{0})$ in $THH(v\subdot \Sdot
\aB^{\Gamma}_{0})$ is a weak equivalence.
\end{lem}

\begin{lem}\label{lemhm32}
Under the hypotheses of the preceding paragraph, 
the inclusion of $THH(v\subdot\Sdot \aA^{\Gamma}_{0})$ in $THH(\aA|v)$
is a weak equivalence. 
\end{lem}

These two lemmas then immediately imply Theorem~\ref{thmcomphm3}.

\begin{proof}[Proof of Lemma~\ref{lemhm31}]
Writing $\Sigma$ for suspension and $C$ for cone, we have a cofiber
sequence of enriched exact functors 
\[
\Id \to C\to \Sigma
\]
on each $v_{m}\aB_{0}$, and so it follows from
Corollary~\ref{coradditivity} that 
\[
\Id\vee\Sigma, C\colon
  WTHH^{\Gamma}(v\subdot \aB_{0})\to  WTHH^{\Gamma}(v\subdot \aB_{0})
\]
induce the same map in the stable category.  On the other hand, using
a simplicial contraction, it is easy to see that $C$ induces the
trivial map. Thus,
\[
\Sigma \colon THH(v\subdot\Sdot\aB_{0})\to THH(v\subdot \Sdot \aB_{0})
\]
is a weak equivalence.  Similarly, 
\[
\Sigma \colon THH(v\subdot\Sdot\aA_{0})\to THH(v\subdot \Sdot \aA_{0})
\]
is a weak equivalence.  Since the canonical map
\[
\colim_{\Sigma}\aA_{0}\to \colim_{\Sigma}\aB_{0}
\]
is an isomorphism, the lemma now follows from
Proposition~\ref{propthhcolim}. 
\end{proof}

\begin{proof}[Proof of Lemma~\ref{lemhm32}]
Let $\bar v\aA_{0}$ denote the subcategory of $v\aA_{0}$ consisting of
those maps that are also degreewise split monomorphisms.  Then by
\cite[1.3.9]{HM3} and Proposition~\ref{propmoorecof}, it suffices to
show that the inclusion 
\[
THH(\bar v\subdot\Sdot \aA^{\Gamma}_{0})\to
THH(\bar v\subdot\Sdot \aA^{\Gamma})
\]
is a weak equivalence.  Since $\Sdot[n]\aB_{0}$ and $v\Sdot[n]\aB_{0}$
satisfy the same hypotheses as $\aB_{0}$, without loss of
generality, it suffices to show that the inclusion
\[
THH(\bar v\subdot\aA^{\Gamma}_{0})\to
THH(\bar v\subdot\aA^{\Gamma})
\]
is a weak equivalence.  By Proposition~\ref{propthhreal}, it suffices
to show that each degeneracy map
\[
THH(\bar v\subdot\aA^{\Gamma}_{0})\to
THH(\bar v\subdot\aA^{\Gamma}_{n})
\]
is a weak equivalence, which we do using an argument similar to the
proof of Theorem~\ref{thmdkdiag}.

Let $s\colon \aA_{0}\to \aA_{n}$ denote the iterated degeneracy and
let $d\colon \aA_{n}\to \aA_{0}$ denote the iterated last face map.
The composite functor $d\circ s$ is the identity and so induces the
identity map 
\[
THH(\bar v\subdot\aA^{\Gamma}_{0})\to
THH(\bar v\subdot\aA^{\Gamma}_{n}).
\]
We show that the composite $s\circ d$ is homotopic to the identity
map.  We have a map of simplicial sets 
\[
c\colon \Delta[n] \times \Delta[1]\to \Delta[n]
\]
that is a null homotopy from the identity map to the inclusion of the
last vertex.  We can use this to construct an exact functor
$c\colon \aA_{n}\to \aA_{n}$ as follows.  Regarding an element of
$f\in \aA_{n}(x,y)$ as a map 
$\tilde f\colon x\otimes \Delta[n]\to y$ in
$\aA_{0}$, we let $c(f)\in \aA_{n}(x\otimes \Delta[1],y\otimes
\Delta[1])$ be the element represented by the map 
\[
(x\otimes \Delta[n]) \otimes \Delta[1]
\iso x\otimes (\Delta [n]\times \Delta[1])
\to y\otimes \Delta[1]
\]
in $\aA_{0}$ induced by $\tilde f$, $c$, and the diagonal map on
$\Delta[1]$.  This then extends to a simplicial spectral functor 
\[
c\subdot\colon \bar v\subdot \aA_{n}^{\Gamma}\to \bar v\subdot \aA_{n}^{\Gamma}.
\]
We construct two simplicial homotopies $H_{0},H_{1}$ of simplicial
spectral functors 
using the two inclusions $\partial_{0},\partial_{1}$ of $\Delta[0]$ in
$\Delta[1]$: On objects, 
\[
x_{1}\to \cdots \to x_{n}
\]
in $\Ob v_{n}\aA$ is sent to 
\[
x_{1}\to \cdots \to x_{i}\to 
x_{i}\otimes \Delta[1]\to \cdots \to x_{n}\otimes \Delta[1]
\]
in $\Ob v_{n}\aA$ where the map $x_{i}\to x_{i}\otimes \Delta[1]$ is
$\partial_{0}$ for $H_{0}$ and $\partial_{1}$ for $H_{1}$.  On morphisms,
$H_{0}$ sends
\[
\xymatrix@-1pc{%
x_{1}\ar[r]\ar[d]_{f_{1}}&\relax\cdots\ar[r]&x_{n}\ar[d]_{f_{n}}\\
y_{1}\ar[r]&\relax\cdots\ar[r]&y_{n}
}
\]
to
\[
\xymatrix@-1pc{%
x_{1}\ar[r]\ar[d]_{f_{1}}&\relax\cdots\ar[r]
&x_{i}\ar[r]^-{\partial_{0}}\ar[d]_{f_{i}}
&x_{i}\otimes \Delta[1]\ar[r]\ar[d]_{c(f_{i})}
&\relax\cdots\ar[r]
&x_{n}\otimes \Delta[1]\ar[d]_{c(f_{n})}
\\
y_{1}\ar[r]&\relax\cdots\ar[r]&y_{i}\ar[r]
&y_{i}\otimes \Delta[1]\ar[r]
&\relax\cdots\ar[r]
&y_{n}\otimes \Delta[1]
}
\]
and $H_{1}$ sends it to
\[
\xymatrix@-1pc{%
x_{1}\ar[r]\ar[d]_{s(d(f_{1}))}&\relax\cdots\ar[r]
&x_{i}\ar[r]^-{\partial_{1}}\ar[d]_{s(d(f_{i}))}
&x_{i}\otimes \Delta[1]\ar[r]\ar[d]_{c(f_{i})}
&\relax\cdots\ar[r]
&x_{n}\otimes \Delta[1]\ar[d]_{c(f_{n})}
\\
y_{1}\ar[r]&\relax\cdots\ar[r]&y_{i}\ar[r]
&y_{i}\otimes \Delta[1]\ar[r]
&\relax\cdots\ar[r]
&y_{n}\otimes \Delta[1].
}
\]
Then $H_{0}$ is a simplicial homotopy of
spectral functors from the identity to $c\subdot$ and $H_{1}$ is a
simplicial homotopy of spectral functors from $s\circ d$ to $c\subdot$.
\end{proof}

We now move on to the proof of Theorem~\ref{thmcompekmm}.  Let $\aA$
denote the category of perfect simplicial $R$-modules and now let
$\aC$ denote the category of finite cell EKMM $HR$-modules.  Having
proved Theorem~\ref{thmcomphm3}, for the proof of
Theorem~\ref{thmcompekmm}, we just need to produce compatible zigzags
of weak equivalences 
\begin{gather*}
WTHH^{\Gamma}(\aA^{v})\simeq WTHH^{\Gamma}(\aC^{v})\\
WTHH^{\Gamma}(\aA)\simeq WTHH^{\Gamma}(\aC)\\
WTHH^{\Gamma}(\aA|v)\simeq WTHH^{\Gamma}(\aC|v).
\end{gather*}
Let $\aM$ denote the full subcategory of EKMM $HR$-modules that are
compact in the derived category and whose underlying spectra satisfy a
cardinality bound (for any limit cardinal large enough that
$\aC\subset \aM$).  The inclusion of $\aC$ in $\aM$ then induces
compatible weak equivalences 
\begin{gather*}
WTHH^{\Gamma}(\aC^{v})\overto{\simeq} WTHH^{\Gamma}(\aM^{v})\\
WTHH^{\Gamma}(\aC)\overto{\simeq} WTHH^{\Gamma}(\aM)\\
WTHH^{\Gamma}(\aC|v)\overto{\simeq} WTHH^{\Gamma}(\aM|v),
\end{gather*}
We will in fact compare $THH$ of the $\aA$ categories with $THH$ of
the $\aM$ categories.  We use the functor denoted $\bM$ in \cite[\S
I.7]{MM} to construct a simplicially enriched functor $\aA\to \aM$ as
follows.

Implicitly we are working with the standard model of $HR$ as a
commutative EKMM $S$-algebra, constructed as follows.  The usual
Eilenberg-Mac Lane spectrum $HR$ has as its $n$-th space 
\[
R\otimes \tilde \bZ[S^{n}]
\]
as in Example~\ref{remexact}; this spectrum is canonically a
commutative ring orthogonal spectrum and the canonical commutative
EKMM $S$-algebra $\bM$ of this (i.e., $S\sma_{\aL}(-)$ applied to its
Lewis-May spectrification).  As $\bM$ is a lax monoidal functor, for
any (discrete) $R$-module $M$, the standard EKMM $S$-module $HM$
is $\bM$ of the spectrum
\[
M\otimes \tilde \bZ[S^{n}]
\]
and is canonically an $HA$-module.  For a simplicial $R$-module $M$,
geometric realization commutes with $\bM$, and we obtain a simplicial
functor $M$ from simplicial $A$-modules to EKMM $HA$-modules.

Because the construction $(-)\otimes \tilde \bZ[S^{(-)}]$ does not
preserve coproducts or pushouts, $M\colon \aA\to \aM$ is not an exact
functor.  But it does preserve coproducts up to weak equivalence and
homotopy pushouts, so it is a \term{weakly exact} functor.  It is also 
\indexterm{based}{based (functor)} in that it sends $0$ to $*$ (after
perhaps modifying it by an isomorphism).  Theorem~\ref{thmwkexact}
and the work of Section~\ref{sec:weaklyexact} below then produces
compatible zigzags of maps of cyclotomic spectra
\begin{gather*}
WTHH^{\Gamma}(\aA^{v})\to WTHH^{\Gamma}(\aM^{v})\\
WTHH^{\Gamma}(\aA)\to WTHH^{\Gamma}(\aM)\\
WTHH^{\Gamma}(\aA|v)\to WTHH^{\Gamma}(\aM|v).
\end{gather*}
Since $M$ induces DK-equivalences
\begin{gather*}
\Sdot[n]\aA^{v}\to \Sdot[n]\aM^{v}\\
\Sdot[n]\aA\to \Sdot[n]\aM\\
v^{M}_{m}\Sdot[n]\aA\to v^{M}_{m}\Sdot[n]\aM,
\end{gather*}
the zigzags above consist of weak equivalences.  This completes the
sketch proof of Theorem~\ref{thmcompekmm}.

\section{The localization sequence for $THH(ku)$}
\label{secar}

The main result of this chapter is the following theorem conjectured
by Hesselholt and Ausoni-Rognes.

\begin{thm}\label{thmkuloc}
The transfer maps and the
canonical maps fit into cofiber sequences of cyclotomic spectra
\begin{gather*}
THH(\mathbb{Z}\phat)\to THH(\ell\phat)\to WTHH^{\Gamma}(\ell\phat |
L\phat)\to \Sigma THH(\mathbb{Z}\phat)\\ 
THH(\mathbb{Z}_{(p)})\to THH(\ell)\to WTHH^{\Gamma}(\ell|L)\to \Sigma THH(\mathbb{Z}_{(p)})\\
THH(\mathbb{Z})\to THH(ku)\to WTHH^{\Gamma}(ku|KU) \to \Sigma THH(\mathbb{Z})
\end{gather*}
inducing cofiber sequences
\begin{gather*}
TC(\mathbb{Z}\phat)\to TC(\ell\phat)\to WTC^{\Gamma}(\ell\phat |
L\phat)\to \Sigma TC(\mathbb{Z}\phat)\\ 
TC(\mathbb{Z}_{(p)})\to TC(\ell)\to WTC^{\Gamma}(\ell|L)\to \Sigma TC(\mathbb{Z}_{(p)})\\
TC(\mathbb{Z})\to TC(ku)\to WTC^{\Gamma}(ku|KU) \to \Sigma TC(\mathbb{Z})
\end{gather*}
which are compatible via the cyclotomic trace with the corresponding
cofiber sequences in 
algebraic $K$-theory constructed in \cite{BlumbergMandell}.
\end{thm}

Here $WTHH^{\Gamma}(ku|KU)$ denotes the connective $THH$ of the
category of finite cell $ku$-modules with the spectral enrichment
induced by the canonical mapping spaces in $ku$ but weak equivalences
the $KU$-equivalences.  That is,  
\[
WTHH^{\Gamma}(ku|KU) = 
WTHH^{\Gamma}(\aC_{ku}|v) = 
\Omega|THH(v\subdot^{M} \Sdot \aC^{\Gamma}_{ku})|,
\]
where $\aC_{ku}$ is the category of finite cell EKMM $ku$-modules
(as in Example~\ref{exofinterest}.(i)) and $v \aC_{ku}$ is the
collection of maps $M \to N$ such that $M \sma_{ku} KU \to N \sma_{ku}
KU$ is an equivalence, or equivalently, those maps that induce an
isomorphism on homotopy groups after inverting the action of the Bott
element.

The proof of this theorem follows the same general outline as the
proof of the corresponding result in algebraic
$K$-theory \cite{BlumbergMandell}.  In particular, the localization
theorem follows from a ``d\'{e}vissage'' theorem for finitely
generated finite stage Postnikov towers.  We now give the definitions
necessary to state this theorem.  Throughout, we work with EKMM
$S$-algebras and $S$-modules.

For an $S$-algebra $R$, let $\aP_{R}$ denote the full subcategory of
left $R$-modules that are of the homotopy type of cell $R$-modules and
have only finitely many non-zero homotopy groups, all of which are
finitely generated over $\pi_{0}R$.  We give $\aP_{R}$ the structure
of a simplicially tensored Waldhausen category as follows.  For the
simplicial structure, we use the
usual simplicial enrichment obtained by regarding the category of
$R$-modules as a simplicial model category.  For the Waldhausen
category structure,  we take
the weak equivalences to be the usual weak equivalences and the
cofibrations to be the Hurewicz cofibrations, i.e., the maps
satisfying the homotopy extension property in the category of
$R$-modules.  As we described in \cite[\S 1]{BlumbergMandell}, this
gives $\aP_R$ the structure of a Waldhausen category, and the
pushout-product axiom on the tensors follows from \cite[X.2.3]{EKMM}.
(Techniques to make a version of $\aP_R$ that is a small category 
are discussed in \cite[1.7]{BlumbergMandell}.)

Restricting to the subcategory of the category of $S$-algebras with
morphisms the maps $R\to R'$ for which $\pi_{0}R'$ is finitely
generated as a left $\pi_{0}R$-module, we can regard
$WTHH^{\Gamma}(\aP_{(-)})$ as a contravariant functor to the
homotopy category of cyclotomic spectra.  We can now state the
D\'{e}vissage Theorem. 

\begin{thm}[D\'evissage Theorem]\label{thm:dev}
Let $R$ be a connective $S$-algebra with $\pi_{0}R$ left Noetherian.
Then there is a natural isomorphism in the homotopy category of
cyclotomic spectra $THH(\aE^{fg}_{\pi_{0}R})\to
WTHH^{\Gamma}(\aP_{R})$, where $\aE^{fg}_{\pi_{0}R}$ denotes the exact
category of finitely generated left $\pi_{0}R$-modules.  Moreover,
this isomorphism and the induced isomorphism (in the stable category)
on $TC$ are compatible via the cyclotomic trace with the analogous
isomorphism (in the stable category) on algebraic $K$-theory
$K'(\pi_{0}R)\to K'(R)$ in
the D\'evissage Theorem of~\cite{BlumbergMandell}.
\end{thm}

We prove Theorem~\ref{thm:dev} in the next section and use the rest of
this section to prove Theorem~\ref{thmkuloc} from
Theorem~\ref{thm:dev}.  Let $R$ be one of $ku$, $\ell$, or
$\ell\phat$, and let $\beta$ denote the appropriate Bott element in
$\pi_{*}R$ in degree $2$ or $2p-2$.  Then $R[\beta^{-1}]$ is $KU$,
$L$, or $L\phat$ respectively.  For convenience, let $Z$ denote
$\pi_{0}R$; so $Z=\bZ$, $\bZ_{(p)}$, or $\bZ\phat$ in the respective
cases.  As above we write $\aC_{A}$ for the simplicially
tensored Waldhausen category of finite cell $A$-modules (where $A=HZ$,
$R$, or $R[\beta^{-1}]$).  On $\aC_{R}$ we have the
additional weak equivalences $v\aC_{R}$, the maps that induce an
isomorphism on homotopy groups after inverting the action of the Bott
element.  Since $v\aC_{R}$ contains the usual weak equivalences
$w\aC_{R}$, the hypothesis of the Localization Theorem
(Theorem~\ref{thmloc}) applies and we get a cofibration sequence of
cyclotomic spectra
\[
WTHH^{\Gamma}(\aC_{R}^{v})\to
WTHH^{\Gamma}(\aC_{R})\to
WTHH^{\Gamma}(\aC_{R}|v)\to
\Sigma WTHH^{\Gamma}(\aC_{R}^{v}),
\]
compatible with the analogous sequence in $K$-theory via the
cyclotomic trace.  Corollary~\ref{cor:ekmm} identifies
$WTHH^{\Gamma}(\aC_{R})$ with $THH(R)$, compatibly with the cyclotomic
trace. The inclusion of the $v$-acyclics $\aC_{R}^{v}$ into
the simplicially tensored Waldhausen category $\aP_{R}$ described
above is a tensored exact functor and a DK-equivalence.  Thus,
Theorem~\ref{thm:dev} identifies $\aC_{R}^{v}$ as  $THH(Z)$,
compatibly with the identification of $K(\aC_{R}^{v})$ with $K(HZ)$.

This completes most of the proof of Theorem~\ref{thmkuloc}; it just
remains to identify the map 
\[
THH(Z)\htp WTHH^{\Gamma}(\aC_{R}^{v})\to 
WTHH^{\Gamma}(\aC_{R})\htp THH(R)
\]
in terms of
the transfer map $THH(HZ)\to THH(R)$.  First, we review this transfer
map.  In our current context with $R=ku$, $\ell$, or
$\ell\phat$, the Eilenberg-Mac Lane $R$-module $HZ$ is weakly
equivalent to a finite cell $R$-module.  If we choose a model for $HZ$
as a cofibrant associative $R$-algebra, then finite cell $HZ$-modules
are cell $R$-modules and homotopy equivalent to finite cell
$R$-modules.  Let $\aM^{c}_{R}$ be the simplicially tensored
Waldhausen category whose objects are the $R$-modules that are
homotopy equivalent to finite cell $R$-modules with the usual
simplicial sets of maps, with the usual weak equivalences,  and with
cofibrations the Hurewicz cofibrations 
(using the technique of \cite[1.7]{BlumbergMandell} to make a version
that is a small category).  Then $\aP_{R}$ is a closed Waldhausen
subcategory of $\aM^{c}_{R}$; moreover, the inclusion of $\aC_{R}$ in
$\aM^{c}_{R}$ is tensored exact and a DK-equivalence, and so induces
an equivalence on all versions of $THH$.  We also have the analogous category
$\aM^{c}_{HZ}$ for $HZ$, which coincides with $\aP_{HZ}$. The forgetful
functor from $HZ$-modules to 
$R$-modules is a tensored exact functor $\aM^{c}_{HZ}\to
\aM^{c}_{R}$.  
The transfer map $THH(HZ)\to
THH(R)$ is by definition the map 
\[
\tau_{HZ}^{R} \colon THH(HZ)\to THH(\aM^{c}_{R})
\overfrom{\sim}THH(R),
\]
where the map on the right is
induced by the inclusion of $S_{HZ}$ in $\aM^{c}_{R}$ and the map of
endomorphism spectra  
\[
\aC^{S}_{HZ}(S_{HZ},S_{HZ}) = 
(\aM^{c}_{HZ})^{S}(S_{HZ},S_{HZ}) \to 
(\aM^{c}_{R})^{S}(S_{HZ},S_{HZ}).
\]
(We understand $THH$ of the EKMM $S$-algebra $HZ$ as $THH$ of the
symmetric ring spectrum 
$\aC^{S}_{HZ}(S_{HZ},S_{HZ})$; cf.\ Corollary~\ref{cor:ekmm} and the
remarks that follow it.)

Since the transfer map coincides with the map
\[
THH(HZ)\overto{\sim} THH(\aM^{c}_{HZ})
\to THH(\aM^{c}_{R})
\overfrom{\sim}THH(R),
\]
applying Corollary~\ref{cor:ekmm} and naturality, we can also identify
it as the map 
\[
THH(HZ)\simeq WTHH^{\Gamma}(\aM^{c}_{HZ})
\to WTHH^{\Gamma}(\aM^{c}_{R})
\simeq THH(R).
\]
Using the naturality of the isomorphism in
Theorem~\ref{thm:dev}, we obtain the following commutative diagram of
maps in the homotopy category of cyclotomic spectra. 
\[
\xymatrix{%
THH(Z)\ar[d]_{=}\ar[r]^-{\sim}
&WTHH^{\Gamma}(\aP_{HZ})\ar[d]\ar[r]^-{=}
&WTHH^{\Gamma}(\aM^{c}_{HZ})\ar[d]
&THH(HZ)\ar[l]_-{\sim}\ar@{..>}[d]^{\tau_{HZ}^{R}}
\\
THH(Z)\ar[r]^-{\sim}
&WTHH^{\Gamma}(\aP_{R})\ar[r]
&WTHH^{\Gamma}(\aM^{c}_{R})
&THH(R)\ar[l]_-{\sim}\ar[dl]_-{\sim}\\
&WTHH^{\Gamma}(\aC^{v}_{R})\ar[u]^-{\sim}\ar[r]
&WTHH^{\Gamma}(\aC_{R})\ar[u]^-{\sim}
}
\]
It will be obvious from the proof of Theorem~\ref{thm:dev} in the next
section that the isomorphism $THH(Z)\simeq THH(HZ)$ in the top row of
the diagram is the standard one, and this identifies the map
$THH(Z)\to THH(R)$ as the transfer map.  This completes the proof of
Theorem~\ref{thmkuloc}. 

\section{Proof of  the D\'{e}vissage Theorem}\label{sec:dev}
\label{secdevissage}

This section is devoted to the proof of the D\'evissage Theorem,
Theorem~\ref{thm:dev}.  The argument parallels the analogous
d\'evissage theorem in \cite{BlumbergMandell}, which we review along
the way.

We fix the connective $S$-algebra $R$, writing
$\aP$ for $\aP_{R}$ . Let $\aP_{m}^{n}$ denote the full subcategory of
$\aP$ consisting of those $R$-modules whose homotopy groups $\pi_{q}$
are zero for $q>n$ or $q<m$.  In this notation, we permit $m=-\infty$
and/or $n=\infty$, so $\aP=\aP_{-\infty}^{\infty}$.  The categories
$\aP_{m}^{n}$ are closed Waldhausen subcategories of $\aP_{R}$.  The
following theorem proved below parallels \cite[1.2]{BlumbergMandell}.

\begin{thm}\label{thmbigthm}
The inclusion $\aP_{0}^{0}\to \aP$ induces a weak equivalence
\[
WTHH^{\Gamma}(\aP_{0}^{0}) \to WTHH^{\Gamma}(\aP).
\]
\end{thm} 

The point of the previous theorem is that $\pi_{0}$ provides an exact
functor from $\aP_{0}^{0}$ to the exact category of finitely generated
left $\pi_{0}R$-modules $\aE^{fg}_{R}$.  Theorem~1.3 of 
\cite{BlumbergMandell} proves that this functor induces a weak
equivalence of $K$-theory.  Since the simplicial mapping sets for
$\aE^{fg}_{R}$ are discrete, $\pi_{0}$ is also a simplicially enriched
functor $\aP_{0}^{0}\to\aE^{fg}_{R}$. It is in fact a DK-equivalence
and induces a DK-equivalence $\Sdot[n]\aP_{0}^{0}\to
\Sdot[n]\aE^{fg}_{R}$ for all $n$.  This proves the following theorem,
which parallels \cite[1.3]{BlumbergMandell}.

\begin{thm}
The functor $\pi_{0}\colon \aP_{0}^{0}\to \aE^{fg}_{R}$ induces a weak
equivalence 
\[
WTHH^{\Gamma}(\aP_{0}^{0})\to
WTHH^{\Gamma}(\aE^{fg}_{R})=THH(\aE^{fg}_{R}).
\]
\end{thm}

Theorem~\ref{thm:dev} is an immediate consequence of the previous two
theorems, with the natural isomorphism coming from the natural zigzag of weak
equivalences of cyclotomic spectra
\[
THH(\aE^{fg}_{R})=WTHH^{\Gamma}(\aE^{fg}_{R})\from
WTHH^{\Gamma}(\aP_{0}^{0})\to WTHH^{\Gamma}(\aP).
\]
Thus, it remains to prove Theorem~\ref{thmbigthm}.

The proof of Theorem~\ref{thmbigthm} follows the same outline as the
parallel theorem \cite[1.2]{BlumbergMandell}.  As in the argument
there, we have the following two easy observations.  

\begin{prop}
The inclusion $\aP_{0}^{\infty}\to \aP$
induces an equivalence
\[
WTHH^{\Gamma}(\aP_{0}^{\infty}) \to WTHH^{\Gamma}(\aP).
\]
\end{prop}

\begin{prop}
The cyclotomic spectrum $WTHH^{\Gamma}(\aP_{0}^{\infty})$ is weakly
equivalent to the 
telescope of the sequence of maps
\[
WTHH^{\Gamma}(\aP_{0}^{0})\to \dotsb \to WTHH^{\Gamma}(\aP_{0}^{n})\to
WTHH^{\Gamma}(\aP_{0}^{n+1})\to \dotsb.
\]
\end{prop}

As in \cite{BlumbergMandell}, the proof of Theorem~\ref{thmbigthm}
will then be completed by showing that the maps 
\[
WTHH^{\Gamma}(\aP_{0}^{n})\to 
WTHH^{\Gamma}(\aP_{0}^{n+1})
\]
are weak equivalences for all $n\geq 0$.  Applying
Proposition~\ref{propcofiber} and Theorem~\ref{thmcofiber}, this is
equivalent to proving the following lemma.

\begin{lem}\label{lem:devmain}
$WTHH^{\Gamma}(\aP_{0}^{n}\to \aP_{0}^{n+1})\htp \Omega |THH(\Sdot
\Fdot(\aP_{0}^{n+1},\aP_{0}^{n})^{\Gamma})|$ is weakly contractible.
\end{lem}

In \cite{BlumbergMandell} the proof of the parallel (unnumbered) lemma
consisted of several steps, each of which compared (multi)simplicial
sets; the following diagram outlines the comparisons as stated there.
\[
\xymatrix@C-3pc{%
w\subdot \Sdot \Fdot(\aP_{0}^{n+1},\aP_{0}^{n})\ar[dr]
&&u\subdot\Sdot \aP_{0}^{n+1}\ar[dl]\ar[rr]
&\hspace{3pc}
&u\subdot\Spdot \aP_{0}^{n+1}\ar@{..>}[rr]\ar@{..>}[d]
&&u\subdot M\subdot Z\\
&w\subdot u\subdot \Sdot \aP_{0}^{n+1}
&&&u\subdot \Ffdot[\ssdot-1]\aP^{n+1}_{0}\ar@{..>}[urr]
&\relax\hspace{3pc}
&u\subdot \Ffdot[\ssdot-1]\aP^{n+1}_{n+1}\ar@{..>}[u]\ar@{..>}[ll]
}
\]
We review these constructions as needed below. Here the solid arrows
are simplicial maps of diagonal simplicial sets and the dotted arrows
are maps that are simplicial only in one of the simplicial directions.
We correct a minor error in \cite{BlumbergMandell} below.  There we
claimed that the dotted arrow in the top row was a map of bisimplicial
sets; it is not.  The diagram for the corrected argument looks like this; it
commutes up to simplicial homotopy.
\[
\xymatrix@C-4pc{%
w\subdot \Sdot \Fdot(\aP_{0}^{n+1},\aP_{0}^{n})\ar[dr]
&&u\subdot\Sdot \aP_{0}^{n+1}\ar[dl]\ar[rr]
&\hspace{4.5pc}
&u\subdot\Spdot \aP_{0}^{n+1}\ar[rr]\ar@{..>}[d]
&\hspace{4.5pc}
&u\subdot S^{f}M\subdot Z
&\relax\hspace{4pc}
&u\subdot M\subdot Z\ar[ll]\\
&w\subdot u\subdot \Sdot \aP_{0}^{n+1}
&&&u\subdot \Ffdot[\ssdot-1]\aP^{n+1}_{0}\ar@{..>}[urrrr]
&&&&u\subdot \Ffdot[\ssdot-1]\aP^{n+1}_{n+1}
  \ar@{..>}[llll]\ar@{..>}[u]
}
\]
In the current context of $THH$, the line of reasoning and the diagram
simplifies slightly; we use the following diagram of
spectrally enriched functors, which commutes up to natural isomorphism.
\begin{equation}\label{eq:devmain}
\begin{gathered}
\hbox to 0pt{\hss
\xymatrix@C-1.1pc{%
(\Sdot[p] \Fdot[q](\aP_{0}^{n+1},\aP_{0}^{n}))^{\Gamma}\ar[r]^-{\sim}
&((u_{q}\Spdot[p])^{M}\aP_{0}^{n+1})^{\Gamma}\ar[r]\ar@{..>}[d]_-{\sim}
&(u_{q}S^{f}M_{p}Z)^{\Gamma}
&(u_{q}M_{p}Z)^{\Gamma}\ar[l]_-{\sim}\\
&((u_{q}\Ffdot[p-1])^{M}\aP_{0}^{n+1})^{\Gamma}\ar@{..>}[urr]
&&((u_{q}\Ffdot[p-1])^{M}\aP_{n+1}^{n+1})^{\Gamma}\ar@{..>}[u]_-{\sim}\ar@{..>}[ll]
}\hss}
\end{gathered}
\end{equation}
All of the spectral categories fit into simplicial spectral categories
(in the $q$ direction) and the ones on the top row fit into
bisimplicial spectral categories (in $p,q$).  The solid arrows are the
spectrally enriched functors that respect the bisimplicial structure;
the dotted arrows respect the simplicial structure in the $q$
direction.  The arrows marked ``$\sim$'' are DK-equivalences, as shown in
Propositions~\ref{prop:devdk1}, \ref{prop:devdk2},
\ref{prop:devdk3}, and \ref{prop:devdk4}.  The goal is to show that the composite functor $(\Sdot[p] \Fdot[q](\aP_{0}^{n+1},\aP_{0}^{n}))^{\Gamma}\to
(u_{q}S^{f}M_{p}Z)^{\Gamma}$ induces a weak equivalence 
\[
|THH((\Sdot \Fdot (\aP_{0}^{n+1},\aP_{0}^{n}))^{\Gamma})|\to
|THH((u\subdot S^{f}M\subdot Z)^{\Gamma})|
\htp
|THH((u\subdot M\subdot Z)^{\Gamma})|
\]
and then prove Lemma~\ref{lem:devmain} by showing that
$|THH((u\subdot M\subdot Z)^{\Gamma})|$ is contractible
(Proposition~\ref{prop:uMcont}). 

We now begin to review the categories and maps in
diagram~\eqref{eq:devmain}.  We use the following notation.

\begin{defn}
\index{uP@$u\aP$}\index{fP@$f\aP$}
Let $u\aP$ denote the subcategory of $\aP$ consisting of those maps
that induce an isomorphism on $\pi_{n+1}$ and an injection on
$\pi_{n}$.  Let $f\aP$ denote the subcategory of $\aP$ consisting of
those maps that induce an epimorphism on $\pi_{0}$.
\end{defn}

We write $u\subdot \aP$ for the nerve categories: An object of
$u_{q}\aP$ is a sequence of $q$ composable maps in $u\aP$ and a map in
$u_{q}\aP$ is a commuting diagram (of maps in $\aP$).  For consistency
with \cite[3.7]{BlumbergMandell}, we let $\Ffdot[p]\aP$ denote the nerve
category $f_{p}\aP$: An object is a sequence of $p$ composable maps
in $f\aP$ and a map is a commuting diagram (of maps in $\aP$).  We
extend the definition of $u\subdot$ in the obvious way to functor
categories: In diagram~\eqref{eq:devmain}, the category
$u_{q}\Sdot[p]\aP_{0}^{n+1}$ has as objects the sequences of
$q$-composable maps
\[
A_{0}\overto{\alpha_{1}} A_{1}\overto{\alpha_{2}} \dotsb \overto{\alpha_{q}} A_{q}
\]
between objects $A_{i}$ in $\Sdot[p]\aP_{0}^{n+1}$ where each
$\alpha_{i}$ is (objectwise) in $u\aP_{0}^{n+1}$; a map from
$\{\alpha_{i}\}$ to $\{\alpha'_{i}\}$ consists of a map $\phi_{i}\colon A_{i}\to
A'_{i}$ in $\Spdot[p]\aP_{0}^{n+1}$ for each $i$, making the diagram
\[
\xymatrix{%
A_{0}\ar[r]^{\alpha_{1}}\ar[d]_{\phi_{0}}
&A_{1}\ar[r]^{\alpha_{2}}\ar[d]_{\phi_{1}}
&\relax\dotsb \ar[r]^{\alpha_{q}}
&A_{q}\ar[d]^{\phi_{p}}\\
A'_{0}\ar[r]_{\alpha'_{1}}
&A'_{1}\ar[r]_{\alpha'_{2}}
&\relax\dotsb \ar[r]_{\alpha'_{q}}&A'_{q}
}
\]
in $\Spdot[p]\aP_{0}^{n+1}$ commute.  We define the
categories $u_{q}\Ffdot[p-1]\aP_{0}^{n+1}$ and
$u_{q}\Ffdot[p-1]\aP_{n+1}^{n+1}$ analogously in terms of composable
maps and diagrams in $\Ffdot[p-1]\aP_{0}^{n+1}$ and
$\Ffdot[p-1]\aP_{n+1}^{n+1}$.  We obtain the spectrally enriched
categories $((u_{q}\Spdot[p])^{M}\aP_{0}^{n+1})^{\Gamma}$,
$((u_{q}\Ffdot[p-1])^{M}\aP_{0}^{n+1})^{\Gamma}$, and
$((u_{q}\Ffdot[p-1])^{M}\aP_{n+1}^{n+1})^{\Gamma}$ using the Moore Tot
mapping spaces (Construction~\ref{cons:moore}) and the connective
spectral enrichment.   
The usual face and degeneracy maps in the nerve construction makes
$((u\subdot\Spdot )^{M}\aP_{0}^{n+1})^{\Gamma}$ into a bisimplicial
spectral category and make $((u\subdot
\Ffdot[p-1])^{M}\aP_{0}^{n+1})^{\Gamma}$ and $((u\subdot
\Ffdot[p-1])^{M}\aP_{n+1}^{n+1})^{\Gamma}$ into simplicial spectral
categories for each $p>0$. 

Next we review the canonical inclusion
\[
\Fdot[q](\aP_{0}^{n+1},\aP_{0}^{n})\to u_{q}\aP_{0}^{n+1}.
\]
We recall that an object of $\Fdot[q](\aP_{0}^{n+1},\aP_{0}^{n})$
consists of a sequence of $q$ composable cofibrations in
$\aP_{0}^{n+1}$
\[
\xymatrix@C-1pc{%
x_{0}\ar@{ >->}[r]&x_{1}\ar@{ >->}[r]&\dotsb\ar@{ >->}[r]&x_{q}
}
\]
such that each quotient $x_{i+1}/x_{i}$ is in $\aP_{0}^{n}$.  We
note that for a cofibration $j\colon a\to b$ in $\aP$ between objects of
$\aP_{0}^{n+1}$, the quotient $b/a$ is in $\aP_{0}^{n+1}$ if and only
if $j$ induces an isomorphism on $\pi_{n+1}$ and an injection on
$\pi_{n}$, that is, if and only if $j$ is in $u\aP$.  It follows that
$\Fdot[q](\aP_{0}^{n+1},\aP_{0}^{n})$ is the full subcategory of
$u_{q}\aP_{0}^{n+1}$ consisting of those objects whose structure maps
are cofibrations.  We then obtain the functors
\[
\Sdot[p] F_{q}(\aP_{0}^{n+1},\aP_{0}^{n})\to
u_{q}\Spdot[p]\aP_{0}^{n+1}
\]
as the corresponding inclusions of full subcategories.  When we look
at mapping spaces and use the Moore enrichment, we obtain a DK-embedding
\[
\Sdot[p] F_{q}(\aP_{0}^{n+1},\aP_{0}^{n})\to
(u_{q}\Spdot[p])^{M}\aP_{0}^{n+1}.
\]
This map is a DK-equivalence since the usual cylinder argument
replacing a map with a cofibration converts any diagram in
$u_{q}\Spdot[p]\aP_{0}^{n+1}$ to a weakly equivalent diagram in
$\Sdot[p]\Fdot[q](\aP_{0}^{n+1},\aP_{0}^{n})$.  Passing to the
connective spectral enrichments, we obtain the following proposition.

\begin{prop}\label{prop:devdk1}
The spectrally enriched functor 
\[
\Sdot[p]
F_{q}(\aP_{0}^{n+1},\aP_{0}^{n})^{\Gamma}\to 
((u_{q}\Spdot[p])^{M}\aP_{0}^{n+1})^{\Gamma}
\]
is a DK-equivalence.
\end{prop}

Next we review the functor $\Spdot[p]\aP_{0}^{n+1}\to
\Ffdot[p-1]\aP_{0}^{n+1}$ of \cite[3.8]{BlumbergMandell}.  First note
that for an object $A=\{a_{i,j}\}$ in
$\Spdot[p]\aP_{0}^{n+1}$, the map $a_{i,p}\to a_{j,p}$ is the cofiber
of the map $a_{i,j}\to a_{i,p}$ and so we have a long exact sequence
of homotopy groups
\[
0\to \pi_{n+1}a_{i,j}\to \dotsb \to \pi_{0}a_{i,j}\to
\pi_{0}a_{i,p}\to \pi_{0}a_{j,p}\to 0.
\]
In particular, the map $a_{i,p}\to a_{j,p}$ is surjective on
$\pi_{0}$, that is, is a map in $f\aP_{0}^{n+1}$.  We therefore obtain
a functor $\Spdot[p]\aP_{0}^{n+1}\to \Ffdot[p-1]\aP_{0}^{n+1}$
by sending each object of $\Spdot[p]\aP_{0}^{n+1}$ to the object of
$\Ffdot[p-1]\aP_{0}^{n+1}$ defined by the sequence
\[
a_{0,p}\to a_{1,p}\to \dotsb \to a_{p-1,p}.
\]
In fact we have the following proposition.

\begin{prop}\label{prop:devdk2}
The spectrally enriched functor 
\[
((u_{q}\Spdot[p])^{M}\aP_{0}^{n+1})^{\Gamma}
\to ((u_{q}\Ffdot[p-1])^{M}\aP_{0}^{n+1})^{\Gamma}
\]
is a DK-equivalence.
\end{prop}

\begin{proof}
Although $\Spdot[p]\aP_{0}^{n+1}$ is defined in terms of homotopy
cocartesian squares,  it could
equally well be defined in terms of homotopy cartesian squares since
for EKMM $R$-modules a square is homotopy cartesian if and only if it
is homotopy cocartesian.  The description of the mapping space of
$\Sdot[p]\aP$ in~\eqref{eqsn} has an analogue in this context: The
canonical map from
$\SpMdot[p]\aP$ to the iterated homotopy pullback
\[
\aP(a_{0,p},b_{0,p})\times^{h}_{\aP(a_{0,p},b_{1,p})}
\dotsb
\times^{h}_{\aP(a_{p-2,p},b_{p-1,p})}
\aP(a_{p-1,p},b_{p-1,p}).
\]
is a weak equivalence.  This extends to $(u_{q}\Spdot[p])^{M}\aP$ and
from this it is easy to deduce that we have a 
DK-embedding.  It is a DK-equivalence because every object of
$u_{q}\Ffdot[p-1]\aP_{0}^{n+1}$ is weakly equivalent to the image of
an object in $u_{q}\Spdot[p]\aP_{0}^{n+1}$, filling out the diagram by
taking homotopy fibers. 
\end{proof}

The inclusion of $\aP_{n+1}^{n+1}$ as a subcategory of $\aP_{0}^{n+1}$
induces a spectrally enriched functor
$((u_{q}\Ffdot[p-1])^{M}\aP_{n+1}^{n+1})^{\Gamma}\to
((u_{q}\Ffdot[p-1])^{M}\aP_{0}^{n+1})^{\Gamma}$, which assembles to a
simplicial spectrally enriched functor in the $q$ direction.  Although
not a DK-equivalence at any level, the simplicial spectrally enriched
functor does induce a weak equivalence on $THH$.

\begin{prop}\label{prop:post}
The inclusion 
$((u\subdot\Ffdot[p-1])^{M}\aP_{n+1}^{n+1})^{\Gamma}\to
((u\subdot\Ffdot[p-1])^{M}\aP_{0}^{n+1})^{\Gamma}$
induces a weak equivalence 
\[
|THH(((u \subdot \Ffdot[p-1])^{M} \aP_{n+1}^{n+1})^{\Gamma})| \to
|THH(((u \subdot \Ffdot[p-1])^{M}\aP_{0}^{n+1})^{\Gamma})|. 
\]
\end{prop}

\begin{proof}
Consider the bisimplicial spectral category $V\dsubdot^{\Gamma}$
defined as follows: in bidegree $r,s$, $V_{r,s}^{\Gamma}$ is the full
spectral subcategory of $((u_{r+s+1} \Ffdot[p-1])^{M}
\aP_{0}^{n+1})^{\Gamma}$ with objects the sequences of sequences of
the form
\[
a_{0}\to \dotsb \to a_{r}\to b_{0}\to \dotsb b_{s}
\]
such that the objects $a_i$ are in $\Ffdot[p-1] \aP_{n+1}^{n+1}$.
Dropping the objects $\{a_i\}$ and the objects $\{b_i\}$ respectively
induce bisimplicial spectrally enriched functors 
\[
((u_r \Ffdot[p-1])^{M} \aP_{n+1}^{n+1})^{\Gamma}
\from
V_{r,s}^{\Gamma}
\to
((u_s \Ffdot[p-1])^{M} \aP_{0}^{n+1})^{\Gamma},
\]
where we regard the targets as constant bisimplicial objects in the
appropriate direction.  Since the (connective) spectrum of maps from
an object $x$ of $\aP_{n+1}^{n+1}$ to an object $y$ of $\aP^{n+1}_{0}$
is homotopy discrete with $\pi_{0}=\Hom_{\bZ}(\pi_{n+1}x,\pi_{n+1}y)$,
we see that the map $V_{r,s}^{\Gamma} \to ((u_s \Ffdot[p-1])^{M}
\aP_{0}^{n+1})^{\Gamma}$ is a DK-embedding.  Furthermore, it is clear
that this functor is essentially surjective (choosing an $n$-connected
cover of $b_{0}$), and so is a DK-equivalence.

The usual arguments show that the map
$V_{r,s}^{\Gamma} \to
((u_r \Ffdot[p-1])^{M} \aP_{n+1}^{n+1})^{\Gamma}$ is a simplicial
homotopy equivalence in the $s$-direction, using the homotopy inverse
induced by
\[
(a_{0}\to \dotsb \to a_{r}) \qquad \mapsto\qquad 
(a_{0}\to \dotsb \to a_{r} = a_{r} = \dotsb = a_{r}). 
\]
Using this homotopy inverse, the composite map on (diagonal)
simplicial spectral categories
\[
((u_r \Ffdot[p-1])^{M} \aP_{n+1}^{n+1})^{\Gamma} \to
V_{r,r}^{\Gamma}\to
((u_r \Ffdot[p-1])^{M} \aP_{0}^{n+1})^{\Gamma}
\]
is induced by
\[
(a_{0}\to \dotsb \to a_{r}) \qquad \mapsto\qquad 
(a_{r} = a_{r} = \dotsb = a_{r}),
\]
and is easily seen to be simplicially homotopic to the inclusion map.
\end{proof}

For the categories $u_{q}M_{p}Z$, we copy the following definition
from \cite[3.9]{BlumbergMandell}.

\begin{defn}
\index{MpZ@$M_{p}Z$}
\index{uMpZ@$uM_{p}Z$}
Let $Z=\pi_{0}R$.  Let $M_{p}Z$ be the category whose objects are
sequences of $p-1$ composable maps of finitely generated left $Z$-modules
$x_{0}\to \dotsb \to x_{p-1}$ and whose morphisms are commutative
diagrams.  Let $uM_{p}Z$ be the subcategory of $M_{p}Z$ consisting of
all objects but only those maps $x\to y$ that are isomorphisms
$x_{i}\to y_{i}$ for all $0\leq i\leq p-1$.
\end{defn}

We understand $M_{0}Z$ to be the trivial category consisting of a
single object (the empty sequence of maps) with only the identity map.
As above, we let $u_{q}M_{p}Z$ denote the nerve category, which has as
its objects the composable sequences of $q$ maps in $uM_{p}Z$ (i.e.,
isomorphisms in $M_{p}Z$) and maps the commutative diagrams of maps in
$M_{p}Z$.  We regard $u_{q}M_{p}Z$ as simplicially enriched with
discrete mapping spaces and we obtain a connective spectral enrichment
$(u_{q}M_{p}Z)^{\Gamma}$ using objectwise direct sum of finitely
generated left $Z$-modules.

As above, $(u\subdot M_{p}Z)^{\Gamma}$ assembles into a simplicial
spectral category using the usual face and degeneracy maps for the
nerve.  We make $(u\subdot M\subdot Z)^{\Gamma}$ into a
bisimplicial spectral category as follows: For $0\leq i\leq p-1$, on
$x_{0}\to \dotsb \to x_{p-1}$, the face map $\partial_{i}\colon
u_{q}M_{p}Z\to u_{q}M_{p-1}Z$ is defined by dropping $x_{i}$ (and
composing) and the degeneracy map $s_{i}\colon M_{p-1}Z\to M_{p}Z$ is
defined by repeating $x_{i}$ (with the identity map).  The face map
$\partial_{p}\colon M_{p}Z\to M_{p-1}Z$ sends $x_{0}\to \dotsb \to
x_{p-1}$ to $k_{0}\to \dotsb \to k_{p-2}$, where $k_{i}\subset x_{i}$ is the kernel
of the composite map $x_{i}\to x_{p-1}$.  The last degeneracy
$s_{p-1}\colon M_{p-1}Z\to M_{p}Z$ puts $0$ in as the last object in
the sequence.  The fundamental property of $(u\subdot M\subdot
Z)^{\Gamma}$ that we need is the following.

\begin{prop}\label{prop:uMcont}
For each $q$, $|THH((u_{q}M\subdot Z)^{\Gamma})|$ is contractible.
\end{prop}

\begin{proof}
The argument at the end of Section~3 of \cite{BlumbergMandell}
constructs a simplicial contraction on the simplicial spectral
category $(u_{q}M\subdot Z)^{\Gamma}$.  This simplicial contraction
induces a simplicial contraction on the simplicial spectrum
$THH((u_{q}M\subdot Z)^{\Gamma})$ and geometric realization converts
this to a contraction of $|THH((u_{q}M\subdot Z)^{\Gamma})|$.
\end{proof}

Applying $\pi_{n+1}$, we get a functor $u\Ffdot[p-1] \aP_{0}^{n+1}\to
uM_{p-1}$ and spectrally enriched functors 
\[
((u_{q}\Ffdot[p-1])^{M}\aP_{0}^{n+1})^{\Gamma}\to
(u_{q}M_{p}Z)^{\Gamma}
\quad \text{and}\quad 
((u_{q}\Ffdot[p-1])^{M}\aP_{n+1}^{n+1})^{\Gamma}\to
(u_{q}M_{p}Z)^{\Gamma}.
\]
Looking at the mapping spaces and mapping spectra, the following
proposition is clear.

\begin{prop}\label{prop:devdk3}
The spectrally enriched functor
\[
((u_{q}\Ffdot[p-1])^{M}\aP_{n+1}^{n+1})^{\Gamma}\to
(u_{q}M_{p}Z)^{\Gamma}
\]
is a DK-equivalence.
\end{prop}

In \cite[\S3]{BlumbergMandell}, we claimed that the functors
$u\Spdot[p]\aP_{0}^{n+1}\to uM_{p} Z$ respected the simplicial
structure in the $p$ direction, which is untrue.  To fix this, we
introduce the category $uS^{f}M\subdot Z$.

\begin{defn}\label{def:usfM}
\index{SfMpZ@$S^{f}M_{p}Z$}
\index{uSfMpZ@$uS^{f}M_{p}Z$}
Let $S^{f}M_{p}Z$ be the category whose objects are functors
$A=a_{-,-}$ from $\Ar[p]$ to the category of finitely generated left
$Z$-modules such that:
\begin{enumerate}
\item $a_{i,i}=0$, and
\item $a_{i,j}\to a_{i,k}$ is an isomorphism onto the kernel of the
map $a_{i,k}\to a_{j,k}$
\end{enumerate}
for all $i\leq j\leq k$.  A map in $S^{f}M_{p}Z$ is a commutative
diagram.  The subcategory $uS^{f}M_{p}Z$ consists of those maps in
$S^{f}M_{p}$ that are isomorphisms.  
\end{defn}

We make $uS^{f}M\subdot Z$ a simplicial category using the usual face
and degeneracy operations on $\Ar[\ssdot]$. 
Basically $S^{f}M\subdot Z$ is the fibration version of the $\Sdot$
construction for the co-Waldhausen category (category with
\emph{fibrations} and weak equivalences) structure we get on the
category of finitely generated left $Z$-modules by taking the
fibrations to be all maps and the weak equivalences to be the
isomorphisms.  We have a forgetful functor $uS^{f}M_{p}Z\to uM_{p}Z$
which takes $A=\{a_{i,j}\}$ to the sequence
\[
a_{0,p}\to \dotsb \to a_{p-1,p}.
\]
This functor is an equivalence of categories, with the inverse functor
$uM_{p}Z\to uS^{f}M_{p}Z$ filling out the $\Ar[p]$ diagram from the
sequence with the kernels of the maps.  These functors then assemble
into a simplicial functor $uM\subdot Z\to uS^{f}M\subdot Z$.

Now $\pi_{n+1}$ defines a simplicial functor $u\Spdot\aP_{0}^{n+1}\to
uS^{f}M\subdot Z$.  The following theorem fixes the argument in
\cite{BlumbergMandell} by replacing Theorem~3.10.

\begin{thm}\label{thm:fixbm}
The simplicial functors $u\Spdot\aP_{0}^{n+1}\to
uS^{f}M\subdot Z\from uM\subdot Z$ induce weak equivalences on
nerves.
\end{thm}

\begin{proof}
Fix $p$.  Since $uM_{p}Z\to uS^{f}M_{p}Z$ is an equivalence of
categories, it induces a weak equivalence on nerves.  The proof of
Theorem~3.10 in \cite[\S4]{BlumbergMandell} correctly proves that the
functor $u\Spdot[p]\aP_{0}^{n+1}\to uM_{p}Z$ induces a weak equivalence
on nerves, and the composite functor 
\[
u\Spdot[p]\aP_{0}^{n+1}\to uM_{p}Z\to uS^{f}M_{p}Z
\]
is naturally isomorphic to the functor $u\Spdot[p]\aP_{0}^{n+1}\to
uS^{f}M_{p}Z$ in the statement, so that functor also induces a
weak equivalence on nerves.
\end{proof}

We regard the categories $u_{q}S^{f}M_{p}$ as
simplicially enriched with discrete mapping spaces and we obtain a
connective spectral enrichment $(u_{q}S^{f}M_{p})^{\Gamma}$ using
objectwise direct sum.  Since the functor $u_{q}M_{p}Z\to
u_{q}S^{f}M_{p}Z$ is an equivalence of categories, we get a
DK-equivalence on the connective spectral enrichments.

\begin{prop}\label{prop:devdk4}
The spectral functor $(u_{q}M_{p}Z)^{\Gamma}\to
(u_{q}S^{f}M_{p}Z)^{\Gamma}$ is a DK-equivalence.
\end{prop}

Finally, we have everything in place to prove Lemma~\ref{lem:devmain}.

\begin{proof}[Proof of Lemma~\ref{lem:devmain}]
Propositions~\ref{prop:devdk1}, \ref{prop:devdk2}, \ref{prop:post},
\ref{prop:devdk3}, and~\ref{prop:devdk4} imply that the bisimplicial map
\[
THH(\Sdot[p]\Fdot[q](\aP_{0}^{n+1},\aP_{0}^{n})^{\Gamma})\to
THH((u_{q}S^{f}M_{p} Z)^{\Gamma})
\]
is a weak equivalence for each fixed $p,q$.
Propositions~\ref{prop:uMcont} and~\ref{prop:devdk4} then imply that
\[
|THH(\Sdot \Fdot(\aP_{0}^{n+1},\aP_{0}^{n})^{\Gamma})|\htp
|THH((u\subdot S^{f}M\subdot Z)^{\Gamma})| \htp
|THH((uM\subdot Z)^{\Gamma})|
\]
is contractible. 
\end{proof}

%%%%%%%%%%%%%%%%%%%%%%%%%%%%%%%%%%%%%%%%
%%%%%%%%%%%%%%%%%%%%%%%%%%%%%%%%%%%%%%%%

% tp-set-source: tw-whole.tex
% tex-command: latex
% ultex-add-defs: list

%
% Copyright (C) 2008-14  Andrew J. Blumberg and Michael A. Mandell
%

\chapter{Generalization to {W}aldhausen categories with factorization}
\label{tw-gen}

In previous sections, we imposed stringent hypotheses on our
categories and functors.  In this chapter, we relax these hypotheses
and extend the theory.  We begin in the first section by generalizing
the maps we consider.  Often, a functor between Waldhausen categories
preserves the structure only ``up to homotopy''; in previous
work~\cite{BlumbergMandell}, we developed a theory of ``weakly exact''
functors to describe the associated functoriality of algebraic
$K$-theory.  In the first section (Section~\ref{sec:weaklyexact}), we
study the functoriality of $WTHH$ in weakly exact functors and show
that a weakly exact functor of simplicial Waldhausen categories
induces a zig-zag of spectra.  We use this additional generality in
the following sections to establish that our hypotheses introduced
previously are generic, in the following sense.  In
Section~\ref{futuresec}, we show that any reasonable Waldhausen
category (an ``HCLF Waldhausen category''; see Definition~\ref{def:hclfw})
is connected by a weakly exact DK-equivalence to an enhanced 
simplicially enriched Waldhausen category.  In the last section
(Section~\ref{sec:speccompare}), we show that if we start with a
suitable spectral category $\aC$, the two evident constructions of
$THH$ (namely, $THH$ applied to the category viewed as a ring spectrum
with many objects and $WTHH$ applied to the Waldhausen category of
finite-cell modules) are connected by a natural zig-zag of weak
equivalences.

The combination of Sections~\ref{sec:weaklyexact} and~\ref{futuresec}
allow us to regard $THH$ and $TC$ as functors from the homotopy
category of HCLF Waldhausen categories (and weakly exact functors) to the
homotopy categories of cyclotomic spectra and spectra, respectively;
see Theorem~\ref{thmwrcsv}.(i).  Ideally, we would like to say
something about coherence, perhaps showing that
$THH$ and $TC$ are $\infty$-functors to spectra.  The difficulty
arises in Section~\ref{sec:weaklyexact}, where a weakly exact
simplicially enriched functor of Waldhausen categories only yields a
zigzag of spectra instead of a map of spectra.  A composable sequence
of such functors yields a subdivided simplex of maps, a
multi-dimensional zigzag.  While it is clear that there should be a
clean and straightforward $\infty$-category interpretation of this
kind of generalized functor, we know of no good theory to plug it
into, and this monograph does not seem like the appropriate place to
develop one.

The work in Section~\ref{futuresec}, on the other hand, does have a
relatively straightforward to interpretation as an $\infty$-functor in
the context of quasi-categories (using the homotopy coherent nerve).
The functor $\aC\mapsto \tilde\aC$ from HCLF Waldhausen categories to
simplicially enriched Waldhausen categories is easily seen to be the
composite of a lax pseudofunctor with an op-lax pseudofunctor, where
we regard both categories of Waldhausen categories as strict
2-categories with 2-morphisms the natural weak equivalences.  A
(strictly unital) lax or op-lax pseudo-functor of strict 2-categories
then induces a map on homotopy coherent nerves of the topologically
enriched categories obtained by geometric realization of the nerve of
the morphism categories.  After strictifying the units, we
get a zigzag of $\infty$-functors.  We leave the details to a future
paper.

%%%%%%%%%%%%%%%%%%%%%%%%%%%%%%%%%%%%%%%%
\section{Weakly exact functors}\label{sec:weaklyexact}

In this section, we still consider functors that preserve the
simplicial enrichment, but now we drop the hypothesis that the functor is
exact, and substitute the up to weak equivalence version of this
hypothesis that the functor is ``weakly exact''
\cite[\S2]{BlumbergMandellUW}.  For Waldhausen categories that admit
functorial factorization of weak cofibrations (FFWC), a weakly exact functor is the
minimum structure necessary to induce a map on $K$-theory.  
The purpose of this section is to explain the proof of the following
theorem, which provides the corresponding result in our setting.

\begin{thm}\label{thmwkexact}
Let $\aC$ and $\aD$ be simplicially enriched Waldhausen categories and
assume that the underlying Waldhausen category of $\aD$
% is DK-compatible and  % Need this?
admits FFWC.  Let $\phi \colon \aC\to \aD$ be a simplicially
enriched functor that restricts to a based weakly exact functor on the
underlying Waldhausen categories, then it induces a map 
\[
WTHH^{\Gamma}(\aC)\to
WTHH^{\Gamma}(\aD)
\]
in the
homotopy category of cyclotomic spectra.  This map is compatible with
the cyclotomic trace in that the following diagram commutes in the
stable category.
\[
\xymatrix{%
K\aC\ar[d]\ar[r]^-{\trc}&WTC^{\Gamma}(\aC)\ar[d]\ar[r]&WTHH(\aC)\ar[d]\\
K\aD\ar[r]_-{\trc}&WTC^{\Gamma}\aD\ar[r]&WTHH^{\Gamma}(\aD)
}
\]
\end{thm}

In the case of enhanced simplicially enriched Waldhausen categories,
we have the following version of the previous theorem.

\begin{thm}\label{thmwkexactnc}
Let $\aA$ and $\aB$ be enhanced simplicially enriched Waldhausen
categories with ambient simplicially tensored Waldhausen categories
$\aC$ and $\aD$ respectively.  If $\phi \colon \aC\to \aD$ is a simplicially
enriched functor that sends $\aA$ into $\aB$ and restricts to a based weakly
exact functor on the underlying Waldhausen categories, then it induces
a map in the homotopy category of cyclotomic spectra 
\[
WTHH(\aA) \to WTHH(\aB)
\]
making the following diagram commute in the homotopy
category of cyclotomic spectra.
\[
\xymatrix{%
WTHH^{\Gamma}(\aA)\ar[r]\ar[d]&WTHH^{\Gamma}(\aB)\ar[d]\\
WTHH(\aA)\ar[r]&WTHH(\aB)
}
\]
\end{thm}

We also have the following theorem for natural weak equivalences
between enriched weakly exact functors.

\begin{thm}\label{thmwknat}
Let $\phi$ and $\phi'$ be as in Theorem~\ref{thmwkexact} or
Theorem~\ref{thmwkexactnc} above.  If there is a natural weak
equivalence from $\phi$ to $\phi'$, then the induced maps from
$WTHH^{\Gamma}(\aC)$ to $WTHH^{\Gamma}(\aD)$ agree in the homotopy
category of cyclotomic spectra and (for Theorem~\ref{thmwkexactnc})
the induced maps from $WTHH(\aA)$ to $WTHH(\aB)$
agree in the homotopy category of cyclotomic spectra.
\end{thm}

The proof of these theorems requires the $\SpMdot$ construction from
Section~\ref{sec:spMdot}; a weakly exact functor is precisely a
functor that is compatible with that construction.  We begin with the
definition of weakly exact functor.

\begin{defn}[{\cite[2.1]{BlumbergMandellUW}}]
Let $\aC_{0}$ and $\aD_{0}$ be Waldhausen categories.  A functor $\phi
\colon \aC_{0}\to \aD_{0}$ is \term{weakly exact} if the initial map
$*\to \phi(*)$ in $\aD_{0}$ is a weak equivalence and $\phi$ preserves
weak equivalences, weak cofibrations, and homotopy cocartesian
squares.  We say that a weakly exact functor $\phi$ is
\indexterm{based}{based (functor)} if
the initial map $*\to \phi(*)$ is the identity.
\end{defn}

It follows that a functor that preserves weak equivalences will
preserve weak cofibrations and homotopy cocartesian squares if and
only if it takes cofibrations to weak cofibrations and takes pushouts
along cofibrations to homotopy cocartesian squares.

Let 
\begin{align*}
W'THH^{\Gamma}\aC&=\Omega |THH(\SpMdot \aC^{\Gamma})|\\
\wt{W'THH}^{\Gamma}\aC(n)&=|THH((w\subdot\Spndot)^{M}\aC^{\Gamma})|.
\end{align*}
If $\aA$ in an enhanced simplicially enriched Waldhausen category, let
\begin{align*}
W'THH\aA&=\Omega |THH(\SpMdot \aA^{S})|\\
\wt{W'THH}\aA(n)&=|THH((w\subdot\Spndot)^{M}\aA^{S})|.
\end{align*}
Proposition~\ref{prop:zeroinclude} now implies the following theorem.

\begin{thm}
Let $\aC$ be a simplicially enriched Waldhausen category that admits
FFWC.  The maps of cyclotomic spectra
\[
WTHH^{\Gamma}(\aC)\to W'THH^{\Gamma}(\aC)
\quad \text{and}\quad
\wt{WTHH}^{\Gamma}(\aC)\to \wt{W'THH}^{\Gamma}(\aC)
\]
are weak equivalences.
If $\aA$ is an enhanced simplicially enriched Waldhausen category,
then the maps of cyclotomic spectra
\[
WTHH(\aA)\to W'THH(\aA)
\quad \text{and}\quad
\wt{WTHH}(\aA)\to \wt{W'THH}(\aA)
\]
are weak equivalences.
\end{thm}

Functoriality of $THH$ in weakly exact functors requires one more
twist.  Because an exact functor $\aC_{0}\to \aD_{0}$ preserves
coproducts, an enriched exact functor induces a functor on spectral
enrichments.  For a weakly exact functor $\phi$, the map 
\[
\sigma_{(n)} \colon \phi (c_{1})\vee \dotsb \vee \phi (c_{n})\to
\phi (c_{1}\vee \dotsb \vee c_{n})
\]
is generally not an isomorphism, though it is always a weak
equivalence.  To fix this problem,  we use a zigzag with the following
construction.

\begin{cons}
For simplicially enriched Waldhausen categories $\aC$ and $\aD$ and a
functor $\phi \colon \aC\to \aD$ that is simplicially
enriched and based weakly exact, let
$\phi^{*}(\SpMdot\aC)^{\Gamma}$ be the (simplicial)
topological $\Gamma$-category whose objects are the objects of
$\SpMdot\aC$ and whose $\Gamma$-space of maps
$\phi_{*}(\SpMdot\aC)^{\Gamma}_{q}(a,b)$ (for each fixed
$\ssdot=0,1,2,\ldots$) consists of maps
\[
f_{0}\in \SpMdot\aC(a, \myop\bigvee_{q}b),
\qquad
f_{1}\in \SpMdot\aD(\phi(a), \myop\bigvee_{q}\phi(b)),
\]
a non-negative real number $s$, and a homotopy $f_{0,1}$ of length $s$
from
$\phi(f_{0})$ to $\sigma_{(q)} \circ f_{1}$, which we topologize as a subset of
\[
\SpMdot\aC(a,\myop\bigvee_{q}b)\times
\SpMdot\aD(\phi(a),\myop\bigvee_{q}\phi(b)) \times \bR
\times \SpMdot\aD(\phi(a),\phi(\myop\bigvee_{q}b))^{I}
\]
as in Construction~\ref{consvone}.  Composition works as follows:
Given $(f_{0},f_{1},s,f_{0,1})\colon a\to b$ in level $q$ and
$(g_{0},g_{1},t,g_{0,1})\colon b\to c$ in level $r$, the composition is
$(g_{0}\circ f_{0},g_{1}\circ f_{1},s+t,h_{0,1})$ where $g_{0}\circ
f_{0}$ and $g_{1}\circ f_{1}$ are the compositions in
$(\SpMdot\aC)^{\Gamma}$ and $(\SpMdot\aD)^{\Gamma}$, and $h_{0,1}$ is
the homotopy that does 
\[
\phi (\myop\bigvee_{q} g_{0})\circ f_{0,1}
\]
on $[0,s]\subset [0,s+t]$ (composing as in
Definition~\ref{defn:vnerve}) and does
\[
\sigma_{(r)}\circ (\myop\bigvee_{q} g_{0,1})\circ f_{1}
\]
on the length $t$ part $[s,s+t]$ of $[0,s+t]$.  We define
$\phi_{*}(w\subdot\SpnMdot\aC)^{\Gamma}$ similarly.
\end{cons}

We have canonical (simplicial) spectral functors 
\[
\phi_{*}(\SpMdot\aC)^{\Gamma}\to \SpMdot\aC^{\Gamma}
\qquad \text{and}\qquad
\phi_{*}(\SpMdot\aC)^{\Gamma}\to \SpMdot\aD^{\Gamma}
\]
induced by projecting on to the relevant factors in the product
\[
\phi_{*}(\SpMdot\aC)^{\Gamma}_{q}(a,b)\subset 
\SpMdot\aC^{\Gamma}_{q}(a,b)\times
\SpMdot\aD^{\Gamma}_{q}(\phi(a),\phi(b)) \times \bR
\times \SpMdot\aD(\phi(a),\phi(\myop\bigvee_{q}b))^{I}.
\]
Because the Moore construction for
$\phi_{*}(\SpMdot\aC)^{\Gamma}_{q}(a,b)$ is homotopy equivalent to the
homotopy pullback and the maps $\sigma_{(n)}$ above are weak
equivalences, the projection 
\[
\phi_{*}(\SpMdot\aC)^{\Gamma}_{q}(a,b)\to
\SpMdot\aC^{\Gamma}_{q}(a,b)
\]
is always a weak homotopy equivalence.  We now have the following theorem.

\begin{thm}\label{thm:swefunct}
Let $\aC$ and $\aD$ be simplicially enriched Waldhausen categories and
let $\phi \colon \aC\to \aD$ be a functor that is simplicially
enriched and based weakly exact.  Then we have zigzags of
(simplicial or multisimplicial) spectrally enriched functors, with the leftward
arrow a DK-equivalence. 
\begin{gather*}
\SpMdot\aC^{\Gamma}\overfrom{\sim} 
\phi_{*}(\SpMdot\aC)^{\Gamma}_{\aC}\to
\SpMdot\aD^{\Gamma}\\
w\subdot\SpnMdot\aC^{\Gamma}\overfrom{\sim} 
\phi_{*}(w\subdot\SpnMdot\aC)^{\Gamma}_{\aC}\to
w\subdot\SpnMdot\aD^{\Gamma}
\end{gather*}
\end{thm}

Returning to $THH$, we define
\begin{gather*}
W'THH^{\Gamma}(\phi_{*}\aC)=\Omega |THH(\phi_{*}(\SpMdot\aC)^{\Gamma}|\\
\wt{W'THH}^{\Gamma}(\phi_{*}\aC)(n)=|THH(\phi_{*}(\SpnMdot\aC)^{\Gamma}|.
\end{gather*}
The fact that $\phi$ is based gives simplicial suspension maps,
imparting the structure of a symmetric spectrum (in the category of
cyclotomic spectra).  The symmetric spectrum structure is compatible
with the functors in the previous theorem, as summarized in the next
result. 

\begin{thm}\label{thmwkzig}
Let $\phi \colon \aC\to \aD$ be a simplicially enriched functor that
restricts to a based weakly exact functor $\aC_{0}\to \aD_{0}$.  Then
we have the following maps of cyclotomic spectra
\[
\xymatrix@C+1pc@R-1.25pc{%
WTHH^{\Gamma}(\aC)\ar[d]&
\wt{WTHH}^{\Gamma}(\aC)\ar[d]\\
W'THH^{\Gamma}(\aC)\ar@{<-}[d]_-{\sim}&
\wt{W'THH}^{\Gamma}(\aC)\ar@{<-}[d]_-{\sim}\\
W'THH^{\Gamma}(\phi_{*}\aC)\ar[d]&
\wt{W'THH}^{\Gamma}(\phi_{*}\aC)\ar[d]\\
W'THH^{\Gamma}(\aD)\ar@{<-}[d]&
\wt{W'THH}^{\Gamma}(\aD)\ar@{<-}[d]\\
WTHH^{\Gamma}(\aD)&
\wt{WTHH}^{\Gamma}(\aD)
}
\]
and the upward maps marked ``$\sim$'' are weak equivalences.  
If $\aD_0$ admits FFWC, then all upward maps are weak equivalences. 
\end{thm}

For Theorem~\ref{thmwkexact}, we have the cyclotomic trace induced by
the inclusion of objects, producing the commutative diagram
\[
\xymatrix@R-1.25pc{%
K(\aC)\ar[r]\ar[dr]\ar[ddr]\ar[ddd]&
\wt{WTC}^{\Gamma}(\aC)\ar[d]\ar[r]
&\wt{WTHH}^{\Gamma}(\aC)\ar[d]\\
&\wt{W'TC}^{\Gamma}(\aC)\ar@{<-}[d]_-{\sim}\ar[r]
&\wt{W'THH}^{\Gamma}(\aC)\ar@{<-}[d]_-{\sim}\\
&\wt{W'TC}^{\Gamma}(\phi_{*}\aC)\ar[d]\ar[r]
&\wt{W'THH}^{\Gamma}(\phi_{*}\aC)\ar[d]\\
K'(\aD)\ar[r]
&\wt{W'THH}^{\Gamma}(\aD)\ar@{<-}[d]_-{\sim}\ar[r]
&\wt{W'THH}^{\Gamma}(\aD)\ar@{<-}[d]_-{\sim}\\
K(\aD)\ar[r]\ar[u]^-{\sim}
&\wt{WTC}^{\Gamma}(\aD)\ar[r]
&\wt{WTHH}^{\Gamma}(\aD)
}
\]
where here $K'(\aD)$ denotes $K$-theory constructed from the $\Spdot$
construction.  
This completes the proof the Theorem~\ref{thmwkexact}.

Theorem~\ref{thmwkexactnc} is entirely similar, using the map 
\[
\tau_{n}\colon \Sigma^{n} \phi(c)\to \phi(\Sigma^{n}c)
\]
in place of $\sigma_{(n)}$ above.  We can see that $\tau_{1}$ is a
weak equivalence by writing the suspension as a homotopy pushout, and
$\tau_{n}$ is a weak equivalence since $\tau_{n}=\tau_{1}\circ
\cdots \circ \tau_{1}$.  Given enhanced simplicially enriched Waldhausen
categories $\aA$ and $\aB$ with ambient simplicially tensored Waldhausen categories
$\aC$ and $\aD$, respectively, and
$\phi \colon \aC\to \aD$ a functor that is simplicially
enriched, based weakly exact, and restricts to a functor $\aA\to \aB$,
we then define $\phi_{*}\aA^{S}$ as the spectrally enriched category
whose set of objects is the same as $\aA$ and whose spectrum of maps
$\phi_{*}\aA^{S}(a,b)$ is defined by letting
\[
\phi_{*}\aA^{S}(a,b)(n)\subset |\aA^{S}(a,b)(n)|\times |\aB^{S}(\phi(a),\phi(b))(n)|\times \bR\times |\aB(\phi(a),\phi(\Sigma^{n}b))|^{I}
\]
be the subspace of elements $(f_{0},f_{1},s,f_{0,1})$ where $f_{0,1}$
is a length $s\geq 0$ homotopy from $\phi(f_{0})$ to $\tau_{n}\circ
f_{1}$.  For a diagram $D$, we define $\phi_{*}(D^{M}\aA)^{S}$ likewise.
This gives us the non-connective analogue of
Theorem~\ref{thm:swefunct}. 

\begin{thm}\label{thm:swefunctS}
Let $\aA$ and $\aB$ be enhanced simplicially enriched Waldhausen
categories with ambient simplicially tensored Waldhausen categories
$\aC$ and $\aD$, respectively, and
let $\phi \colon \aC\to \aD$ be a functor that is simplicially
enriched, based weakly exact, and restricts to a functor $\aA\to \aB$.
Then we have a zigzag of spectrally 
enriched functors, with the leftward arrow a DK-equivalence.
\[
\aA^{S} \overfrom{\sim} \phi_{*}\aA^{S}
\to \aB^{S}
\]
\end{thm}

Writing
\begin{align*}
W'THH(\phi_{*}\aA)&=\Omega |THH(\phi_{*}\SpMdot \aA^{S})|\\
\wt{W'THH}(\phi_{*}\aA)(n)&=|THH(\phi_{*}(w\Spndot)^{M}\aA^{S})|
\end{align*}
we obtain non-connective analogue of Theorem~\ref{thmwkzig}.

\begin{thm}
Let $\aA$ and $\aB$ be enhanced simplicially enriched Waldhausen
categories with ambient simplicially tensored Waldhausen categories
$\aC$ and $\aD$, respectively, and
let $\phi \colon \aC\to \aD$ be a functor that is simplicially
enriched, based weakly exact, and restricts to a functor $\aA\to \aB$.
Then we have the following maps of cyclotomic spectra
\[
\xymatrix@C+1pc@R-1.25pc{%
WTHH(\aA)\ar[d]&
\wt{WTHH}(\aA)\ar[d]\\
W'THH(\aA)\ar@{<-}[d]_-{\sim}&
\wt{W'THH}(\aA)\ar@{<-}[d]_-{\sim}\\
W'THH(\phi_{*}\aA)\ar[d]&
\wt{W'THH}(\phi_{*}\aA)\ar[d]\\
W'THH(\aB)\ar@{<-}_-{\sim}[d]&
\wt{W'THH}(\aB)\ar@{<-}_-{\sim}[d]\\
WTHH(\aB)&
\wt{WTHH}(\aB)
}
\]
and the upward maps are weak equivalences.  
\end{thm}

Finally for Theorem~\ref{thmwknat}, choosing a natural weak
equivalence from $\phi$ to $\phi'$, we obtain a simplicially enriched
and weakly exact functor $\Phi$ from $\aC$ to $w_{1}\aD$.  We obtain
the zigzag 
\[
THH(\SpMdot\aC^{\Gamma})\overfrom{\sim}
THH(\Phi_{*}(\SpMdot\aC)^{\Gamma})\to 
THH(((\Spdot w_{1})^{M}\aD)^{\Gamma})
\overfrom{\sim}
THH(\SpMdot\aD^{\Gamma}),
\]
and a similar zigzag in the non-connective case (when it applies).

%%%%%%%%%%%%%%%%%%%%%%%%%%%%%%%%%%%%%%%%
\section{Embedding in simplicially tensored {W}aldhausen
  categories}\label{futuresec}

In previous sections we worked under the stringent compatibility
hypotheses in our definition of a simplicially enriched Waldhausen
category.  In this section, we show how to produce a DK-compatible
simplicially enriched Waldhausen category from a Waldhausen
category satisfying a certain technical hypothesis.

\begin{defn}\label{def:hclfw}
An \term{HCLF Waldhausen category} is a Waldhausen category that
admits a homotopy calculus of left fractions (HCLF) as defined in
\cite[6.1.(ii)]{DKHammock}.
\end{defn}

We prove the following theorem.

\begin{thm}\label{thmwrcsv}
Let $\aC$ be an HCLF Waldhausen category.  Then there exists a
DK-compatible simplicially enriched Waldhausen category $\wrcsv$ and a
based weakly exact functor $\wri\colon \aC\to \wrcsv$ that is a
DK-equivalence (on simplicial localizations).  Moreover:
\begin{enumerate}
\item  $WTHH^{\Gamma}(\wrcsv)$ is a functor from the category of
HCLF Waldhausen categories and weakly exact maps to the homotopy
category of cyclotomic spectra.
\item As a map in the stable category, $K(\aC)\to K(\wrcsv)$ is
natural in exact functors of $\aC$.
\item As a map in the stable category, the cyclotomic trace $K(\wrcsv)\to
WTHH^{\Gamma}(\wrcsv)$ is natural in weakly exact functors of $\aC$.
\item $\wrcsv$ admits FFWC.
\item If $\aC$ is a DK-compatible simplicially enriched Waldhausen
category then $\wri$ is naturally 
weakly equivalent to a simplicially enriched functor $\wri'$, which is
also based weakly exact. 
\item If $\aC$ can be given the structure of an enhanced simplicially
enriched Waldhausen category, then $\wri'$ induces DK-equivalence
$\Sdot[n]\aC\to \SpMdot[n]\wrcsv$ for all $n$ and so induces a weak
equivalence on $WTHH^{\Gamma}$.
\end{enumerate}
\end{thm}

In the context of part~(v), Theorem~\ref{thm:swefunct} gives a
zigzag of spectrally enriched functors relating $\aC^{\Gamma}$ and
$\wrcsv{}^{\Gamma}$, all of which are DK-equivalences in this case.

As we showed in \cite[\S5, App. A]{BlumbergMandellUW}, a Waldhausen
category that admits factorization (every map factors as a cofibration
followed by a weak equivalence) and any closed Waldhausen subcategory
of such a category in particular admits a homotopy calculus of left
fractions.  In this context, we can also produce an enhanced exact
Waldhausen category.

\begin{thm}\label{thmswrcsv}
Let $\aC$ be a Waldhausen category that admits factorization and let
$\aA$ be a closed Waldhausen  subcategory.  Let $\swrcsv$ be the full subcategory
of $\wrcsv$ of objects weakly equivalent to objects from $\aA$.  Then
$\wrcsv$ is a simplicially tensored Waldhausen category, $\swrcsv$ is a
closed Waldhausen subcategory, and the induced based weakly
exact functor $\swri\colon \aA\to \swrcsv$ is a  DK-equivalence.
Moreover:
\begin{enumerate}
\item  $WTHH(\swrcsv)$ is a functor from the category of
pairs (Waldhausen category, closed Waldhausen subcategory) and weakly
exact maps to the homotopy category of 
cyclotomic spectra.
\item There is a based weakly exact functor $\swrj\colon \swrcsv\to \wrcsv[\aA]$
such that $\swrj\circ \swri$ is naturally weakly equivalent to
$\wri$. (In particular, $\swrj$ is a DK-equivalence.)
\item The map  of 
cyclotomic spectra $WTHH^{\Gamma}(\swrcsv)\to
WTHH^{\Gamma}(\wrcsv[\aA])$ induced by $\swrj$ is a weak equivalence and
natural in the homotopy category of cyclotomic spectra.
\end{enumerate}
\end{thm}

In the context of the previous theorem, when $\aC$ is a simplicially
tensored Waldhausen category, $\aA$ is an enhanced simplicially
enriched Waldhausen category, and part~(v) of
Theorem~\ref{thmwrcsv} gives us a based weakly exact simplicially
enriched functor $\swri'\colon \aA\to \swrcsv$, weakly equivalent to
$\swri$; namely, $\swri'$ is the restriction to $\aA$ of $\wri'\colon \aC\to
\wrcsv$.  Theorem~\ref{thm:swefunctS} then produces a zigzag of
spectrally enriched functors between $\aA^{S}$ and $\swrcsv{}^{S}$,
all of which are DK-equivalences in this case.

The proof of the previous theorems works by embedding $\aC$ in a
simplicial model category in which all objects are fibrant.  We do
this using a variant of a presheaf construction in To\"en and Vezzosi
\cite{ToenVezzosi} to define the $K$-theory of a simplicial category.
In the following discussion, let $L\aC$ denote the simplicial category
obtained as the Dwyer-Kan hammock simplicial localization of $\aC$
with respect to the weak equivalences in the given Waldhausen
structure.

\begin{defn}\label{defcosheaves}
Let $\csv[\aC]$ denote the category of simplicial functors from $L\aC$ to
based simplicial sets taking values in a fixed but sufficiently large
cardinal depending on $\aC$.  We regard $\csv[\aC]$ as a
simplicial model category using the injective model structure
\cite{Heller}, where cofibrations and weak equivalences are defined
objectwise and fibrations are defined by the right-lifting property
with respect to the acyclic cofibrations; in this model structure, all
objects are cofibrant.  The opposite category $\csvop[\aC]$ then has the
opposite simplicial model structure and all objects are fibrant.
\end{defn}

Since the cofibrations in $\csv[\aC]$ are the injections, it is clear
that $\csv[\aC]$ satisfies the pushout-product axiom, which is one of
the equivalent forms of Quillen's SM7; in other words, $\csv[\aC]$ is
a simplicial model category.  It follows that $\csvop[\aC]$ is
likewise a simplicial model category.  Heller \cite[\S4]{Heller} shows
that the injective model structure has functorial factorizations, and
in particular, we have a fibrant replacement functor in $\csv[\aC]$.
In $\csvop[\aC]$, this gives functorial factorization and a cofibrant
approximation functor.  It will be useful for us to have these as
simplicial functors and to preserve the zero object $*$.  We prove the
following lemma at the end of the section.

\begin{lem}\label{lem:cosheaves}
The category $\csv$ admits simplicial endo-functors $P^{c}$ and
$I^{f}$ such  that $P^{c}$ is a cofibrant approximation functor for
the projective model structure, $I^{f}$ is a fibrant approximation
functor for the injective model structure, and $P^{c}(*)=*=I^{f}(*)$.
\end{lem}

The full subcategory of cofibrant objects in $\csvop[\aC]$ inherits
the structure of a Waldhausen category.

\begin{defn}\label{defcosheaves2}
Let $\wrcsv[\aC]$ be the full subcategory of $\csvop[\aC]$ consisting of 
cofibrant objects weakly equivalent to the opposite of a
corepresentable in the image of $\aC$, i.e., weakly equivalent to a
functor of the form $L\aC(x,-)$, where $x$ is an object of $\aC$.
When $\aA$ is a closed Waldhausen subcategory of $\aC$,
let $\swrcsv$ be the full subcategory of $\csvop[\aC]$
consisting of cofibrant objects weakly equivalent to the opposite of a
corepresentable of an object $\aA$.  
\end{defn}

As observed in Example~\ref{exofinterest}, $\wrcsv$ becomes a
DK-compatible simplicially enriched Waldhausen category when given the
Waldhausen structure induced by the model structure.  The Yoneda
embedding 
\[
Y_{\aC}\colon x \mapsto L\aC(x,-)
\]
gives us a functor $Y_{\aC}$ from $\aC$ to $\csvop[\aC]$ that we can compose
with $I^{f}$ to obtain a functor $\aC\to \wrcsv$.
We showed in
\cite[6.2]{BlumbergMandellUW} that under the hypothesis of homotopy
calculus of left fractions, the simplicial localization mapping spaces
take homotopy cocartesian squares to homotopy cartesian squares, and
hence to homotopy cocartesian squares in $\csvop[\aC]$. It follows
that $I^{f}Y_{\aC}$ is a weakly exact functor $\aC$ to $\wrcsv$ and a
DK-equivalence.  It is not, however, a based weakly exact functor as 
the zero object of $\aC$ is generally not a zero
object in $L\aC$.  On the other hand, $L\aC(*,-)\to L\aC(x,-)$ is an
objectwise injection (as it is split by the map $L\aC(x,-)\to
L\aC(*,-)$), and so the based functor
\[
Y'_{\aC}\colon x \mapsto L\aC(x,-)/L\aC(*,-)
\]
is weakly equivalent to $Y_{\aC}$ and hence is a based weakly exact
functor and DK-equivalence.  This proves the following proposition.

\begin{prop}\label{propcowald}
Let $\aC$ be a Waldhausen category that admits a homotopy calculus of
left fractions.  Then the functor $\wri=I^{f}Y'_{\aC}\colon
\aC\to\wrcsv$  is a based weakly exact functor and a 
DK-equivalence. 
\end{prop}

When $\aC$ is a simplicially enriched  Waldhausen
category, $\wri$ is a simplicial functor $L\aC\to \wrcsv$ but
generally not a simplicial functor $\aC\subdot\to \wrcsv$.
We can regard the functor
\[
x\mapsto \diag L\aC\subdot(x,-)/L\aC\subdot(*,-)
\]
as a simplicial functor from $\aC$ to $\csvop[\aC]$.  Composing with
$I^{f}$, we get a simplicial functor $\wri'\colon \aC\to \wrcsv$.  The
inclusion of $L\aC_{0}$ in $L\aC\subdot$ induces a natural
transformation $\wri\to \wri'$, which is a natural weak equivalence
when $\aC$ is DK-compatible (by definition).  This proves the
following proposition. 

\begin{prop}
If $\aC$ is a DK-compatible simplicially enriched Waldhausen category,
then $\wri$ is weakly equivalent to a simplicial functor, which is also
a  based weakly exact DK-equivalence.
\end{prop}

When $\aC$ is a DK-compatible simplicially enriched Waldhausen
category, just as in Proposition~\ref{propfunct}, looking at the
formula for mapping spectra in $\Sdot[n]\aC$ and $\SpMdot[n]\wrcsv$,
we see that $\wri'$ induces a DK-embedding $\Sdot[n]\aC\to
\SpMdot[n]\wrcsv$.  If we assume the hypothesis of part~(vi), then
$\aC$ admits tensors with $\Delta[1]$, and for weak cofibration $x\to
y$, the map $x\vee y \to (x\otimes \Delta[1])\cup_{x\otimes \{1\}}y$
is a cofibration, i.e., $\aC$ has functorial mapping cylinders for
weak cofibrations in the terminology of \cite[2.6]{BlumbergMandellUW}.
Since in any simplicially enriched Waldhausen category, weak
equivalences are closed under retracts, we can apply
\cite[6.1]{BlumbergMandellUW} to characterize the weak cofibrations in
$\aC$ as precisely those maps whose images in $\wrcsv$ are weak
cofibrations.  Moreover, tensors with generalized intervals exist in
$\aC$, and arguing as in the proof of Proposition~\ref{propfuncthree},
we see that every object of $\Spdot[n]\wrcsv$ is weakly equivalent to
the image of an object of $\Spdot[n]\aC$, i.e., that the DK-embedding
is a DK-equivalence.  The induced map (from Theorem~\ref{thmwkexact})
\[
WTHH^{\Gamma}(\aC)\to 
WTHH^{\Gamma}(\wrcsv)
\]
is then a weak equivalence.

Now drop the assumption that $\aC$ is simplicially enriched, and
assume instead that $\aC$ admits factorization.  Then Waldhausen
\cite[p.~357]{WaldhausenKT} shows that we can form homotopy colimits
in $\aC$ over diagrams in finite partially ordered sets as iterated
pushouts over cofibrations.  Since any finite simplicial set is weakly
equivalent to the nerve of a finite partially ordered set, it follows
that for any weakly corepresentable $C$ and any finite simplicial set
$X$, the simplicial functor $C^{X}$ is also weakly corepresentable.
This proves the following proposition.

\begin{prop}
If $\aC$ admits factorization then $\wrcsv$ is a simplicially tensored
Waldhausen category.  
\end{prop}

We also have the corresponding proposition for closed Waldhausen
subcategories. 

\begin{prop}
If $\aA$ is a closed Waldhausen subcategory of
$\aC$, then $\swrcsv\subset \wrcsv$ is an enhanced simplicial
Waldhausen category and 
$\swri\colon \aA\to \swrcsv$ is a based weakly exact functor and a
DK-equivalence on simplicial localizations. 
\end{prop}

We obtain the functor $\swrj\colon \swrcsv\to \wrcsv[\aA]$ as the
restriction to $\swrcsv$ of the functor
$I^{f}_{\aA}\circ R_{\aC}^{\aA}$, where $R_{\aC}^{\aA}$ denotes the
functor $\csvop[\aC]\to \csvop[\aA]$ obtained by restricting an $L\aC$
diagram to $L\aA$ and $I^{f}_{\aA}$ denotes the
endo-functor $I^{f}$ in $\wrcsv[\aA]$.  Writing $Y'_{\aA}$ and
$Y'_{\aC}$ for the modified Yoneda embeddings as above, then 
\[
\swrj\circ \swri=I^{f}_{\aA}R_{\aC}^{\aA}I^{f}_{\aC}Y'_{\aC}
\qquad \text{and}\qquad
\wri=I^{f}_{\aA}Y'_{\aA}.
\]
Under the hypothesis of homotopy calculus of left fractions, the
natural map $Y'_{\aA}\to R_{\aC}^{\aA}Y'_{\aC}$ in $\csv[\aA]$ is a weak
equivalence; 
combining this with the canonical weak equivalence  $\Id \to
I^{f}_{\aC}$ in $\csv[\aC]$ and reversing arrows to work in
$\csvop[\aA]$ gives natural weak equivalences 
\[
\swrj\circ \swri
=
I^{f}_{\aA}R_{\aC}^{\aA}I^{f}_{\aC}Y'_{\aC}
\to
I^{f}_{\aA}R_{\aC}^{\aA}Y'_{\aC}
\to
I^{f}_{\aA}Y'_{\aA}
=
\wri
\]
in $\wrcsv[\aA]$.

The previous observations, propositions, and definitions cover all of
the statements in Theorems~\ref{thmwrcsv} and~\ref{thmswrcsv} except
for the naturality statements.   The next result begins the study of
naturality. 

\begin{thm}\label{thmfunctoriality}
Let $\aC$ and $\aC'$ be Waldhausen categories that admit homotopy
calculi of left fractions, and let $\phi \colon \aC\to\aC'$ be a
weakly exact functor.  Then there exists a simplicial functor
$\wrcsv[\phi]\colon \wrcsv\to \wrcsv'$ that restricts to a based
weakly exact functor of the underlying Waldhausen categories and makes
the diagram of functors
\[
\xymatrix{%
\aC\ar[r]\ar[d]_{\phi}&\wrcsv\ar[d]^{\wrcsv[\phi]}\\
\aC'\ar[r]&\wrcsv'
}
\]
commute up to a zigzag of natural weak equivalences.  

If $\aA$ and
$\aA'$ are closed Waldhausen subcategories of $\aC$ and $\aC'$
(respectively) and $\phi$ restricts to a functor from $\aA$ to $\aA'$,
then the functor $\wrcsv[\phi]$ restricts to a functor
$\swrcsv[\phi]\colon \swrcsv\to \swrcsv'$ making the diagram of
functors 
\[
\xymatrix{%
\aA\ar[r]\ar[d]_{\phi}&\swrcsv\ar[d]^{\swrcsv[\phi]}\\
\aA'\ar[r]&\swrcsv'
}
\]
commute up to a zigzag of natural weak equivalences.  
\end{thm}

We prove this theorem below, but first state the following corollary.

\begin{cor}\label{cor:invariance}
Let $\aC$ and $\aC'$ be Waldhausen categories that admit homotopy
calculi of left fractions, and let $\phi \colon \aC \to \aC'$ be a
weakly exact functor.  If $\phi$ induces a DK-equivalence on
passage to simplicial localizations, then the functor
$\wrcsv[\phi]\colon \wrcsv\to \wrcsv'$ is a DK-equivalence.  Moreover,
$\wrcsv[\phi]$ and  (when appropriate) $\swrcsv[\phi]$ 
induce an equivalence of cyclotomic spectra on $WTHH^{\Gamma}$ and 
$WTHH$, respectively. 
\end{cor}

The proof of Theorem~\ref{thmfunctoriality} combines the simplicially
enriched cofibrant and fibrant approximation functors with left Kan
extension. Fix the functor $\phi \colon \aC \to \aC'$.  Left Kan
extension gives rise to a functor $\Lan_\phi \colon \csv[\aC] \to
\csv[\aC']$ and we let $\wrcsv[\phi]\colon \wrcsv \to \wrcsv'$ be the
composite functor  
\[
\xymatrix{
\csv[\aC] \ar[r]^{P^{c}} & \csv[\aC] \ar[r]^{\Lan_\phi} & \csv[\aC']
\ar[r]^{I^{f}} & \csv[\aC'].
}
\]
By construction $\wrcsv[\phi]$ preserves weak equivalences and is equipped
with a zig-zag of natural weak equivalences connecting $\wri \circ \wrcsv[\phi]$
to $\phi \circ \wri$.  This completes the proof of
Theorem~\ref{thmfunctoriality} .

Most of Corollary~\ref{cor:invariance} follows immediately from
Theorem~\ref{thmfunctoriality}.  To see that $\wrcsv[\phi]$ induces a
weak equivalence on $WTHH^{\Gamma}$, we need to see that the induced
functor $\Sdot[n]\wrcsv \to \SpMdot[n]\wrcsv'$ is a DK-equivalence.
The argument for Proposition~\ref{propfuncthree} adapts to the current
context to complete the proof.

The proof of the naturality statements in Theorems~\ref{thmwrcsv}
and~\ref{thmswrcsv} now follow from an easy check that functors
$\wrcsv[\phi]$ compose as expected up to a zigzag of natural weak
equivalences.  Somewhat more work shows that this construction
actually preserves composition up to coherent homotopy; we defer this
to a future paper.

Finally, we need to prove Lemma~\ref{lem:cosheaves}.  The specifics of
the simplicial category $L\aC$ play no role: the lemma holds for the
category of simplicial functors from any small simplicial
category $\aD$ to based simplicial sets, and we argue in this context.
We prove the following lemma, of which Lemma~\ref{lem:cosheaves} is a
special case.

\begin{lem}\label{lem:funcfact}
Let $\aD$ be a small simplicial category and let $\simpD$ denote the
category of simplicial functors from $\aD$ to based simplicial sets.
Then the projective and injective model structures both admit
factorization functors that are simplicial functors and that send the
identity on $*$ to the factorization $*=*=*$.
\end{lem}

The most basic case is when $\aD$ is the trivial category and $\simpD$
is the category of based simplicial sets.  Let $C$ denote the
set of generating cofibrations ($\partial \Delta[n]\to \Delta[n]$,
$n=0,1,2,\ldots$) and let
$A$ denote the set of generating acyclic cofibrations
($\Lambda_{j}[n]\to \Delta[n]$, $n=0,1,2,\ldots$).  Then the usual
construction of 
the factorization functors uses the small objects argument as
follows.  Given $f\colon x\to y$, the factorization of $f$ as an
acyclic cofibration $x\to x'$ followed by a fibration $x'\to y$ is
constructed as $x'=\colim x'_{n}$, where $x'_{0}=x$ and inductively
$x'_{n+1}$ is constructed as the pushout 
\[
x'_{n+1}=x'_{n}\cup_{\coprod a}(\coprod b) 
\]
where the coproduct is over commutative diagrams
\[
\xymatrix@-1pc{%
a\ar[d]\ar[r]&b\ar[d]\\
x'_{n}\ar[r]&y
}
\]
with $i\colon a\to b$ ranging over the elements of $A$.  The version we need
for Lemma~\ref{lem:funcfact} instead uses the based simplicial set of
maps in place of the set of maps above:  We construct $x'_{n}$
inductively as the pushout
\[
x'_{n+1}=x'_{n}\cup_{\coprod a\sma D_{i}}(\coprod b\sma D_{i})
\]
where the coproduct is over the elements $i\colon a\to b$ in $A$ and 
\[
D_{i}=\simps(a,x'_{n})\times_{\simps(a,y)}\simps(b,y)
\]
is the based simplicial set of commutative diagrams of the form
\[
\xymatrix@-1pc{%
a\ar[d]\ar[r]&b\ar[d]\\
x'_{n}\ar[r]&y.
}
\]
The induced map $x'_{n}\to x'_{n+1}$ and the colimit map $x\to x'$ is
an injection and weak equivalence and the map $x'\to y$ is a
fibration.  Moreover, this functor is clearly a simplicial functor
into the appropriate diagram category.  The analogous construction
using $C$ instead of $A$ constructs the other factorization.  When
applied to the identity map on the trivial based simplicial set $*$,
each $D_{i}$ is the trivial based simplicial set $*$, and so we get
that each map $*\iso x_{n}\to x'_{n+1}$ and $x_{n+1}\to y=*$ is an
isomorphism.  Thus, (replacing the factorization functors with
naturally isomorphic functors if necessary), we have that the
factorization of $*=*$ is $*=*=*$.

A slight modification of the factorization functors in Heller
\cite{Heller} constructs the factorizations in the general case.  Let
$\simpDob$ denote the simplicial category $\prod_{\Ob\aD}\simps$, and
(following the notation in \cite{Heller}), let $J^{*}$ denote the
forgetful functor from $\simpD$ to $\simpDob$ that remembers just the
objects in the diagram (values of the functor) and forgets the maps.
Let $J_{P}$ be its left adjoint; since we are working in based
simplicial sets,  $J_{P}X$ is the simplicial functor
\[
J_{P}X = \bigvee_{c\in\Ob\aD} X(c)\sma \aD(c,-)_{+}.
\]
Likewise, let $J^{I}$ be the right adjoint of $J^{*}$, 
\[
J^{I}X = \prod_{c\in\Ob\aD} X(c)^{\aD(-,c)},
\]
where $X(c)^{\aD(-,c)}$ denotes the based simplicial set of unbased
simplicial maps from $\aD(-,c)$ to $X(c)$.  We note that for any $X$,
$J_{P}X$ is cofibrant in the projective model structure and more
generally, $J_{P}$ sends (objectwise) cofibrations and acyclic
cofibrations in $\simpDob$ to cofibrations and acyclic cofibrations in
the projective model structure on $\simpD$.  Likewise $J^{F}$ sends
(objectwise) fibrations and acyclic fibrations to fibrations and
acyclic fibrations in the injective model structure on $\simpD$.

The factorization functors for the projective model structure are
constructed as follows.  For $f\colon X\to Y$, let $Z_{0}=X$ and
construct $Z_{n+1}$ inductively as follows.  First factor
$J^{*}Z_{n}\to J^{*}Y$ objectwise 
\[
J^{*}Z_{n}\to W_{n}\to J^{*}Y
\]
using the simplicial factorization functor (for the appropriate
factorization) on based simplicial sets constructed above, and let
$Z_{n+1}$ be the pushout
\[
Z_{n+1}=Z_{n}\cup_{J_{P}J^{*}Z_{n}}J_{P}W_{n},
\]
with the factorization $Z_{n+1}\to Y$ induced by the map
$J_{P}W_{n}\to Y$.  Letting $Z=\colim Z_{n}$, we get a factorization
$X\to Z\to Y$, with the map $X\to Z$ a cofibration or acyclic
cofibration (as appropriate) in the projective model structure.  We
note that the underlying map in $\simpDob$ from $J^{*}Z_{n}$ to
$J^{*}Z_{n+1}$ factors through $W_{n}$.  It follows that we can
identify $J^{*}Z$ as $\colim W_{n}$ and the underlying map $J^{*}Z\to
J^{*}Y$ in $\simpDob$ as the colimit of the maps $W_{n}\to J^{*}Y$.
Since by construction, these maps are objectwise acyclic fibrations or
fibrations of simplicial sets, the map $j^{*}Z \to J^{*}Y$ is an
objectwise acyclic fibration or fibration as required.  We note that
when $X=*=Y$, by construction each $W_{n}$ is $*$ and $J_{P}W_{n}$ is
isomorphic to $*$, and so we end up with both factorizations of $*=*$
as $*=*=*$.

The factorization functors on the injective model structure are
precisely dual.  We start with $Z_{0}=Y$, and inductively construct
$Z_{n+1}$ as follows.  Using the appropriate objectwise factorization
functor, we factor  $J^{*}X\to J^{*}Z_{n}$ in $\simpDob$ as
\[
J^{*}X\to W_{n}\to J^{*}Z_{n},
\]
and we define $Z_{n+1}$ as the pullback
\[
Z_{n+1}=Z_{n}\times_{J^{I}J^{*}Z_{n}}J^{I}W_{n}.
\]
We let $Z=\lim Z_{n}$ and get a factorization $X\to Z\to Y$ with $Z\to
Y$ by construction a fibration or acyclic fibration (as appropriate)
in the injective model structure.  Again looking at the underlying map
in $\simpDob$, we see that the map $X\to Z$ is an objectwise acyclic
cofibration or cofibration as appropriate.  Again, the factorization
of $*=*$ becomes $*=*=*$.  This completes the proof of
Lemma~\ref{lem:funcfact}. 

%%%%%%%%%%%%%%%%%%%%%%%%%%%%%%%%%%%%%%%%
\section{Spectral categories and Waldhausen categories}\label{sec:speccompare}

The work of the previous section showed how to associate a spectral
category to any well-behaved Waldhausen category.  On the other hand,
given a spectral category $\aC$, we can produce a simplicially tensored
Waldhausen category by passage to the Waldhausen category $\fincell$
of ``finite cell right $\aC$-modules'' described below.  In this section we
show that when $\aC$ is pretriangulated (Definition~\ref{deftriang}),
the spectral category associated to $\fincell$ in 
Definition~\ref{defcosheaves2} recovers the original spectral category
$\aC$ up to DK-equivalence.

As a general principal, it does not matter which modern category
of spectra we use as a model when discussing small spectral
categories.  The monoidal Quillen equivalences relating the
various categories of diagram spectra and EKMM $S$-modules
\cite{MM,MMSS,SchwedeEKMM} allow us to convert a spectral category on
any of these models to one on any other.  In particular, the following
theorem is an easy consequence of the work of \cite{SSMonoidalEq}
(extended by the techniques of \cite{EKMM} for dealing with
non-cofibrant units that arise there). 

\begin{thm}\label{jthm:compare}
Fix a set of objects $O$.  For $\aS$ a modern category of spectra from
\cite{MMSS} or \cite{EKMM}, let $\aS{O}$-Cat denote the category of
$\aS$-enriched categories with object set $O$ and functors that are
the identity on the object set $O$.  Then:
\begin{enumerate}
\item The category $\aS{O}$-Cat forms a closed model category where
the weak equivalences and fibrations are the functors that induce a
weak equivalence or positive fibration, respectively, on mapping spectra.
\item The monoidal Quillen equivalences from
\cite{MM,MMSS,SchwedeEKMM} induce Quillen equivalences between the
categories  $\aS{O}$-Cat for the various $\aS$.
\end{enumerate}
\end{thm}

Because of this theorem, without loss of generality, we can assume
that our spectral category $\aC$ comes enriched in EKMM $S$-modules,
which have the technical advantage that every object is fibrant.
On the other hand, since our goal is to compare with the
non-connective enrichment of a simplicially tensored Waldhausen
category, our comparison must be between spectral categories enriched
in symmetric spectra.  Again, we use the previous theorem.  Spectral
categories enriched in EKMM $S$-modules are always fibrant in the
model structure of the previous theorem, so the associated spectral
category enriched in symmetric spectra has (the same object set and)
mapping spectra $\Phi \aC(x,y)$, where $\Phi $ is the lax symmetric monoidal
right adjoint functor from EKMM $S$-modules to symmetric spectra
defined in \cite{SchwedeEKMM}.  Specifically, for an EKMM $S$-module
$X$, 
\[
\Phi X(n)=\aM_{S}((S_{S}^{-1})^{(n)},X).
\]
Here $\aM_{S}$ denotes the mapping spaces (in simplicial sets) for the
category of EKMM $S$-modules and $S_{S}^{-1}$ denotes the canonical
cell $(-1)$-sphere $S$-module \cite[III.2]{EKMM}; $\Phi X$ is always a
positive $\Omega$-spectrum and when $X$ is a mapping spectrum,
$\Phi X$ often turns out to be an $\Omega$-spectrum (for example, this
happens for $X=\fincell(x,y)$ where $\fincell$ is the spectral category
defined below).  The lax monoidal natural transformation is induced by
\begin{multline*}
\Phi X(m) \sma \Phi Y(n) =
\aM_{S}((S_{S}^{-1})^{(m)},X) \sma \aM_{S}((S_{S}^{-1})^{(n)},Y)\\
\to
\aM_{S}((S_{S}^{-1})^{(m+n)},X \sma Y)
= \Phi (X\sma Y)(m+n)
\end{multline*}
and the map $S\to \Phi S$ is induced by the map 
$S^{0}\to \aM_{S}(S,S)$ sending the non-base point to the identity
element.  

\begin{notn}
For $\aC$ a spectral category in EKMM $S$-modules, write $\Phi \aC$
for the associated spectral category in symmetric spectra described above.
\end{notn}

Now given $\aC$ a spectral category in EKMM $S$-modules we associate a
Waldhausen category to $\aC$ as follows.  Let $\Mod$ denote the
category of (right) $\aC$-modules, the category of enriched functors from
$\aC^{\op}$ to the category of EKMM $S$-modules.  We make $\Mod$ into a
model category with the projective model structure: The weak
equivalences and fibrations are the objectwise weak equivalences and
fibrations.  The cofibrations in this model structure are the retracts
of relative cell inclusions, where a cell is of the form
\[
\aC(-,x)\sma S_{S}^{q}\sma S^{n-1}_{+}\to \aC(-,x)\sma S_{S}^{q}\sma D^{n}_{+}
\]
for some object $x$ in $\aC$, $q\in \bZ$, $n\geq 0$, where $S^{n-1}\to
D^{n}$ is the standard 
$n$-cell in topology.  We then have a subcategory of finite
cell $\aC$-modules, having objects the $\aC$-modules built from $*$ by
attaching finitely many cells.  If we insist on using canonical
pushouts in building these complexes (or restrict to a skeleton), then
the resulting subcategory we get is small.

\begin{notn}
For $\aC$ a small spectral category in EKMM $S$-modules, let
$\fincell$ be the small subcategory of $\Mod$ of finite cell 
$\aC$-modules.  
\end{notn}

We have a spectrally enriched functor $\aC\to \fincell$ sending $x$ to
$\aC(-,x)\sma S^{0}_{S}$.  By the Yoneda lemma
\[
\fincell(\aC(-,x)\sma S^{0}_{S},\aC(-,y)\sma S^{0}_{S})\iso
F_{S}(S^{0}_{S},\aC(x,y)\sma S^{0}_{S})
\]
(where $F_{S}$ denotes the function $S$-module) 
and the map 
\[
\aC(x,y)\to F_{S}(S^{0}_{S},\aC(x,y)\sma S^{0}_{S})
\]
is a weak equivalence.  The following theorem is now clear from the
construction of $\fincell$.

\begin{thm}
For $\aC$ a small spectral category in EKMM $S$-modules,
the spectrally enriched functor $\aC\to \fincell$ is a DK-embedding,
and $\pi_{0}\fincell$ is the thick subcategory of $\pi_{0}\Mod$
generated by the image of $\aC$. In particular, $\aC\to\fincell$ is a
DK-equivalence if and only if $\aC$ is pretriangulated.
\end{thm}

Since $\fincell$ is a subcategory of cofibrant objects in a simplicial
model category with all objects fibrant, it fits into the context of
Example~\ref{exofinterest}, and is canonically a simplicially enriched
Waldhausen category.  In fact, it is easy to see that the tensor in
$\Mod$ of an object of $\fincell$ with a finite simplicial set is
isomorphic to an object of $\fincell$, so $\fincell$ is a simplicially
tensored Waldhausen category.   The following is the main theorem of
this section; combined with the previous theorem, it gives the zigzag
of DK-equivalence of spectral categories $\Phi \aC\htp \wrfincell^{S}$
when $\aC$ is pretriangulated.

\begin{thm}\label{thm:compare}
For $\aC$ a small spectral category in EKMM
$S$-modules, there are zigzags of DK-equivalences of spectral categories (in
symmetric spectra in simplicial sets)
\[
\Phi \fincell \htp \fincell^{S} \htp \wrfincell^{S},
\]
where $\wrfincell$ denotes the simplicially tensored Waldhausen
category constructed from $\fincell$ by Definition~\ref{defcosheaves2}.
\end{thm}

The zigzag of DK-equivalences $\fincell^{S} \htp \wrfincell^{S}$ is
the one obtained from applying Theorem~\ref{thm:swefunctS} to the
simplicially enriched based weakly exact functor $\wri'\colon
\fincell\to \wrfincell$ in part~(v) of Theorem~\ref{thmwrcsv}.  That
leaves us with constructing the zigzag of DK-equivalences $\Phi
\fincell \htp \fincell^{S}$, which is just the generalization of
Proposition~\ref{prop:corEKMM} to rings with many objects.  The proof
is essentially identical: Let $\Phi'\fincell$ denote the spectral
category (in symmetric spectra in simplicial sets) with the same
objects as $\fincell$, but with mapping spectra $\Phi'\fincell(x,y)$
defined by
\[
\Phi'\fincell(x,y)(n)=\aM_{S}((S^{-1}_{S}\sma
S^{1})^{(n)},F_{\fincell}(x,\Sigma^{n}y)),
\]
where we have written $F_{\fincell}$ for the mapping spectrum in
$\fincell$ to avoid confusion with the mapping space (simplicial
set) $\fincell(x,y)$.  For $y=x$, we have the unit $S\to \Phi'\fincell(x,x)$ induced
by the unit for $\fincell(x,x)$ and the canonical isomorphism 
\[
\aM_{S}((S^{-1}_{S}\sma S^{1})^{(0)},
F_{\fincell}(x,\Sigma^{0}y))
= \aM_{S}(S,F_{\fincell}(x,y)) \iso \fincell(x,y).
\]
Composition is induced by the smash product map
\begin{multline*}
\aM_{S}((S^{-1}_{S}\sma S^{1})^{(m)},F_{\fincell}(y,\Sigma^{m}z))\sma 
\aM_{S}((S^{-1}_{S}\sma S^{1})^{(n)},F_{\fincell}(x,\Sigma^{n}y))
\\
\to 
\aM_{S}((S^{-1}_{S}\sma S^{1})^{(m+n)},
F_{\fincell}(y,\Sigma^{m}z) \sma
F_{\fincell}(x,\Sigma^{n}y))
\end{multline*}
and the composition map 
\begin{multline*}
F_{\fincell}(y,\Sigma^{m}z) \sma
F_{\fincell}(x,\Sigma^{n}y)
\to F_{\fincell}(\Sigma^{n}y,\Sigma^{m+n}z) \sma
F_{\fincell}(x,\Sigma^{n}y)\\
\to F_{\fincell}(x,\Sigma^{m+n}z)
\end{multline*}
analogous to the one in Definition~\ref{defenrich}.
We then have spectral functors
\[
\Phi\fincell \to \Phi'\fincell \from \fincell^{S}
\]
defined as follows.  The functor $\Phi \fincell\to \Phi'\fincell$ is
the map
\[
\aM_{S}((S^{-1}_{S})^{(n)},F_{\fincell}(x,y))
\to
\aM_{S}((S^{-1}_{S}\sma
S^{1})^{(n)},F_{\fincell}(x,\Sigma^{n}y))
\]
induced by $n$-fold suspension
\[
F_{\fincell}(x,y)\to F_{\fincell}(\Sigma^{n}x,\Sigma^{n}y)
\iso \Omega^{n}F_{\fincell}(x,\Sigma^{n}y)
\]
and the adjunction
\[
\aM_{S}((S^{-1}_{S})^{(n)},\Omega^{n}F_{\fincell}(x,\Sigma^{n}y))
\iso 
\aM_{S}((S^{-1}_{S})^{(n)}\sma S^{n},F_{\fincell}(x,\Sigma^{n}y)).
\]
The functor $\fincell^{S}\to \Phi'$ is induced by the map
\[
\fincell(x,\Sigma^{n}y)=
\aM_{S}(S,F_{\fincell}(x,\Sigma^{n}y))\to 
\aM_{S}((S^{-1}_{S}\sma
S^{1})^{(n)},F_{\fincell}(x,\Sigma^{n}y))
\]
induced by the canonical collapse map $S^{-1}_{S}\sma S^{1}\to S$.  On
mapping spaces, both these functors are weak equivalences (in fact,
level equivalences) of symmetric spectra, and so the functors are
DK-equivalences.  This completes the proof of
Theorem~\ref{thm:compare}.

% Index cross references

% See (onlys)
\index{spectral enrichment!connective|seeonly{connective spectral enrichment}}
\index{spectral enrichment!non-connective|seeonly{non-connective spectral enrichment}}
\index{Waldhausen category!simplicially enriched|seeonly{simplicially enriched Waldhausen category}}
\index{Waldhausen category!enhanced|seeonly{enhanced simplicially enriched Waldhausen category}}
\index{Waldhausen category!simplicially tensored|seeonly{simplicially tensored Waldhausen category}}
\index{four term Puppe sequence|seeonly{Puppe sequence}}
\index{built from|seeonly{finitely cellularly built from}}
\index{pointwise ---|seeonly{---, pointwise}}

% See alsos
\index{degenerately based|see{non-degenerately based}}

%\appendix
%    Include appendix "chapters" here.
%\include{}

\backmatter
%    Bibliography styles amsplain or harvard are also acceptable.
\bibliographystyle{amsplain}
\bibliography{thhtc}
%    See note above about multiple indexes.
\printindex

\end{document}